\documentclass[11pt, reqno]{amsart}

\usepackage{amsmath,amssymb,latexsym}
\usepackage{color}
\usepackage{xcolor}
\usepackage{graphicx}
\usepackage{cite}
\usepackage[colorlinks=true]{hyperref}
\hypersetup{urlcolor=black, citecolor=black, linkcolor=black}
\usepackage{multirow}

\newcommand{\R}{\mathbb{R}}

\newcommand{\diff}{\mathrm{d}}

\DeclareMathOperator{\supp}{supp}

\newtheorem{theorem}{Theorem}[section]
\newtheorem{lemma}{Lemma}
\newtheorem{proposition}{Proposition}
\newtheorem{remark}{Remark}
\newtheorem{definition}{Definition}
\newtheorem{corollary}{Corollary}
\newtheorem*{main-theorem}{Main Theorem}
\newtheorem*{remark*}{Remark}
\newtheorem*{lemma*}{Lemma A.1}

\numberwithin{equation}{section}

\begin{document}

\title[The 2D viscous shallow water equations]{On the vacuum free boundary
problem of the viscous Saint-Venant system
for shallow water in two dimensions}

\author{Hai-Liang Li}

\author{Yuexun Wang}

\author{Zhouping Xin}

\address{ School of Mathematics and CIT, Capital Normal University, Beijing 100048, P. R. China.}
\email{hailiang.li.math@gmail.com}

\address{	
	School of Mathematics and Statistics,
	Lanzhou University,  Lanzhou 730000, P. R. China.}

\email{yuexunwang@lzu.edu.cn}

\address{The Institute of Mathematical Sciences, The Chinese University of Hong Kong, Hong Kong}

\email{zpxin@ims.cuhk.edu.hk}

\thanks{}

\begin{abstract}In this paper, 
we establish the local-in-time well-posedness of classical solutions to the vacuum free boundary
problem of the viscous Saint-Venant system
for shallow water in two dimensions. The solutions are shown to possess higher-order regularities uniformly up to the vacuum free boundary, although the depth degenerates as a singularity of the distance to the vacuum boundary.

Since the momentum equations degenerate in both the dissipation and time evolution, there are difficulties in constructing approximate solutions by the Galerkin's scheme and gaining higher-order regularities uniformly up to the vacuum boundary for the weak solution.
% and showing the convergence of the approximate solutions. 
To construct the approximate solutions, we introduce some degenerate-singular elliptic operator, 
%related to the degenerate structure of the momentum equations
whose eigenfunctions form an orthogonal basis of the projection space. Then the high-order regularities on the weak solution are obtained by using some carefully designed higher-order weighted energy functional.

\end{abstract}
\maketitle

%\tableofcontents

%\newpage

\section{Introduction }

This paper concerns the vacuum free boundary problem of the viscous Saint-Venant system
for shallow water in two dimensions (2D viscous shallow water equations) which is given by
\begin{equation}\label{eq:intro-vfb}
\begin{cases}
\partial_t\rho+\textrm{div}(\rho u)=0 &\quad \text{in}\ \Omega(t),\\
\partial_t(\rho u)+\textrm{div}(\rho u\otimes u)+\nabla \rho^2
=\textrm{div}(2\rho\mathbb{D}(u)) &\quad \text{in}\ \Omega(t),\\
\rho>0 &\quad \text{in}\ \Omega(t),\\
\rho=0\ \text{and}\ S\cdot N(t)=0 &\quad \text{on}\ \Gamma(t),\\
\mathcal{V}(\Gamma(t))=u\cdot N(t) &\quad \text{on}\ \Gamma(t),\\
(\rho,u)(\cdot,t=0)=(\rho_0,u_0)(\cdot) &\quad \text{on}\ \Omega,
\end{cases}
\end{equation}
where \(\mathbb{D}(u)\) is the stress tensor
\begin{align*}
\mathbb{D}(u)=\frac{1}{2}\big[\nabla u+(\nabla u)^T\big],
\end{align*}
and \(S\) is the total stress 
\begin{equation*}
\begin{aligned}
S=2\rho \mathbb{D}(u)-\rho^2 \mathbb{I}_2,
\end{aligned}
\end{equation*}
\(\rho\) and \(u\) stand for the \emph{depth} and velocity of the fluid, respectively, \(\Omega(t)\) denotes the changing domain occupied by the
fluid, \(\Gamma(t)=\colon \partial \Omega(t)\) denotes the moving vacuum boundary, \(\mathcal{V}(\Gamma(t))\) denotes the normal velocity of \(\Gamma(t)\), and \(N(t)\) denotes the outward unit normal vector to \(\Gamma(t)\),  and \(\eqref{eq:intro-vfb}_6\) provides the initial depth, velocity and domain, while the condition
\(S\cdot N(t)=0\) in \(\eqref{eq:intro-vfb}_4\) is the stress free condition on \(\Gamma(t)\) and 
\(\eqref{eq:intro-vfb}_5\) says that the speed of the moving vacuum boundary is the normal velocity of the fluid.

The 2D viscous shallow water equations 
describe vertically averaged flows in three dimensional shallow domains in terms of the mean velocity field and the
variation of the depth of the free surface  in geophysical flows (see \cite{MR1978317,  MR1637634}).

%Before precisely formulating our problem, we first briefly review some existence theories and blowup phenomena of viscous compressible flows with vacuum states. 

The study of vacuum is important in understanding viscous compressible flows \cite{MR1637634,MR1867887}. 
For Cauchy problems of the compressible Navier-Stokes equations with
constant viscosity coefficients, the appearance of vacuum leads to some singular behaviors
of solutions, such as the failure of continuous dependence on initial data of weak solutions \cite{MR1117422}, and the finite time blowup of smooth solutions \cite{MR1488513}.  Later, a modified compressible Navier-Stokes equations with density-dependent viscosity coefficients was introduced in \cite{MR1485360} to overcome the singular behaviors of solutions near vacuum.

For the compressible Navier-Stokes equations with degenerate viscosity coefficients, the vacuum free boundary problems have been extensively studied on the global existence of weak solutions, see \cite{MR1936794,  MR2259332, MR2410901, MR2771263, MR2864798} and references therein. However, since the initial profile connecting to the vacuum boundary is rather general, it seems difficult to improve the regularity of those weak solutions.

%For the 1D compressible Navier-Stokes equation, the free boundary problems have been extensively studied on the global existence of weak solutions, see \cite{MR2254008, MR1929151, MR1936794, MR2068308, MR2259332, MR2410901, MR2771263} and references therein. For the compressible Navier-Stokes equations,  Guo et al. \cite{MR2864798} showed the global existence of a weak solution and studied its Lagrangian structure, for spherically symmetric motions. 
%Since the initial profile connecting to the free boundary is rather general, it seems difficult to improve the regularity of those weak solutions.

Physical vacuum is an important class of vacuum states, naturally arising in
the study of the motion of gaseous stars or shallow water (see \cite{MR3026570, MR3280249, MR2591974, MR3218831, MR910225}), which is defined by the fact that
the boundary between the fluid and the vacuum moves with a nontrivial finite normal acceleration, i.e., 
\begin{align*}\tag{$*$}
-\infty<\frac{\partial c^2}{\partial N}<0 \quad \mbox{on}\ \Gamma,
\end{align*}
where \(c=c(x,t)=\sqrt{\partial p/\partial \rho}\) is the sound speed and \(N\) is the outward unit normal vector to the vacuum boundary \(\Gamma\) (see \cite{MR1377457, MR1483001, MR2980528, MR3280249, MR1485360}). %Indeed, Makino and Ukai \cite{MR910225} noted that the requirement that the sound speed \(c\) is continuously differentiable will exclude many interesting solutions such as the stationary solutions of the Euler-Poisson system whose boundary behavior is dictated by \((*)\). (see also \cite{MR2591974, MR3218831}.)  One may refer to \cite{MR3026570} for the effect of physical vacuum. 

Clearly \((*)\) shows that \(c\) is only \(C^{1/2}\)-H\"{o}lder continuous across vacuum boundaries. For inviscid compressible flows, 
\((*)\) implies that the characteristic speeds become singular with infinite spatial derivatives at vacuum boundaries, so the standard approach for symmetric hyperbolic systems fails to apply. This makes understanding the well-posedness of compressible flows with physical vacuum being a challenging problem. 
Until recently, the local
well-posedness of compressible Euler equations with physical vacuum was first given by Jang and Masmoudi \cite{MR2547977, MR3280249} and Coutand  et al. \cite{MR2779087, MR2608125, MR2980528} by different methods (hyperbolic-type
energy estimates and degenerate parabolic
regularizations, respectively) handling the degeneracy near physical vacuum in weighted energy spaces. Later, the unconditional uniqueness of classical solutions to the physical vacuum free boundary problem for the compressible Euler or Euler-Poisson equations was proved by Luo et al. \cite{MR3218831}, and the global existence of classical solutions
to the physical vacuum free boundary problem for the compressible Euler equations with damping was obtained by Luo and Zeng \cite{MR3503024} in the 1D case, and by Zeng \cite{MR3685999} in the 3D spherically symmetric setting, when
the initial data are small perturbations of the Barenblatt solutions. Very recently, this was further extended by Zeng \cite{MR4198724} in the sense of removing the symmetry assumptions to achieve the almost global existence. (See also \cite{MR3986536, MR3551252, MR3479188, MR1766564, MR3303172, MR4010642, MR4302128, MR4000214, MR3924614} for related works.)

Very little is rigorously known about the well-posedness theory for the physical vacuum free boundary problems of viscous compressible flows with degenerate viscosity coefficients.
Concerning the physical vacuum free boundary problem of the 1D shallow water equations derived in \cite{MR1821555},   
the global existence of strong solutions was obtained by Ou and Zeng \cite{MR3397339} based on certain weighted energy estimates with space and time weights using Hardy's inequality, when taking the effect of gravity force into account. However, it is not clear whether the regularity of this solution could be improved or not. 
For the local existence theory, it is not necessary to consider the effect of gravity force. Very recently,  Li et al. \cite{LWX2022} proved the local existence of classical solutions by using some carefully designed higher-order weighted energy functional.

%though a key observation that classical solutions must satisfy an additional Nuewmann boundary condition on the vacuum boundary, which plays a crucial role in constructing of approximate solutions by the Galerkin's scheme. 
%This boundary condition is stronger than the usual
%stress free condition (which does not work since it holds automatically on the vacuum boundary in this degenerate case), which is captured by the high regularity of classical solutions on the vacuum boundary, but not artificially added. To achieve this, some suitable weighted estimates for higher-order derivatives uniformly up to the vacuum boundary are necessary. 

The initial depth profile we are interested in this paper connects to vacuum as follows:
\begin{equation}\label{eq:intro-3}
\begin{cases}
\rho_0\in H^7(\Omega),\
\rho_0(x)=0 \quad \text{on}\ \Gamma,\\
C_1d(x)\leq \rho_0(x)\leq C_2d(x) \quad \text{in}\ \Omega,
\end{cases}
\end{equation}
for some positive constants \(C_1\) and \(C_2\), where \(d(x)=\colon d(x,\Gamma)\) is the distant function from \(x\) to the initial boundary \(\Gamma\). 

It is straightforward to check that \eqref{eq:intro-3} implies that the initial vacuum state is physical in the sense of \((*)\). 
%It is worth pointing out that the BD entropy estimate for the shallow water equations (see \cite{MR1989675}) rules out the initial profile \eqref{eq:intro-3}.

In the present paper, we aim to present a detailed proof on the local well-posedness of classical solutions to the vacuum free boundary problem \eqref{eq:intro-vfb}-\eqref{eq:intro-3}.  Since our classical solutions have very high regularity in \(\Omega(t)\), the usual stress free condition \(S\cdot N(t)=0\) holds automatically  on \(\Gamma(t)\) by the standard trace theorem. Hence the dynamical boundary conditions \(\eqref{eq:intro-vfb}_4\)
reduce to \(\rho=0\) on \(\Gamma(t)\), which is exactly the same one for the compressible Euler system. 

It seems rather nontrivial to extend the result in \cite{LWX2022} to the multi-D case. On the one hand, the appearance of the spatial horizontal derivatives give rise to new difficulties in some estimates involving time-horizontal and horizontal-normal mixed derivatives, which are solved by the crucial use of some weighted inequalities, in particular, the \(H^{1/2}\)-type inequalities and sharp \(L^\infty\)-type inequalities. On the other hand, one needs to construct an orthogonal basis of some weighted Hilbert space, whose regularity depends on a careful analysis on the eigenfunctions of a degenerate-singular elliptic operator, which is more involved than the 1D case, see the discussions below.

For convenience of the readers, we sketch our main strategy to solve the problem and outline the organization of the paper here. \\

{\underline{\emph{Methodology}.}}
We use the Lagrangian coordinates to fix the moving boundary, which enables us to transfer \eqref{eq:intro-vfb} into the following initial boundary problem (see also \eqref{eq:main-2}):
\begin{equation*}\tag{$*1$}
\begin{cases}
\rho_0\partial_tv^i+a_i^k(\rho_0^2 J^{-2}),_{k}=a_l^k(\rho_0J^{-2}a_l^jv^i,_{j}),_{k} &\quad \mbox{in}\ \Omega\times (0,T],\\
(v,\eta)=(u_0,e) &\quad \mbox{on}\ \Omega\times \{t=0\}.
\end{cases}
\end{equation*}
%The problem \((*1)\) is a degenerate quasilinear parabolic one, which can be solved mainly by two steps: constructing approximate solutions and improving regularity of the weak solution. 

Note that \((*1)\) is a second-order degenerate parabolic system in 2D. To be a classical solution of \((*1)\), \(v\) should have the minimum regularity in the standard Sobolev space  
\begin{align*}\tag{$*2$}
C([0,T]; H^4(\Omega))\cap C^1([0,T]; H^2(\Omega)). 
\end{align*}
However, \((*2)\) may be not a good choice to solve \((*1)\) since \((*1)\) degenerates near \(\Gamma\), which prevents the energy method in standard Sobolev spaces from working. In fact,  we will work with some weighted Sobolev space characterized by a higher-order energy functional \eqref{HOEF-0} consisting of both conormal and normal energies, which not only can capture the singular structure of the physical vacuum, but also is contained in \((*2)\) as a subspace. 
Let us first illustrate the main idea for the simplified case \(\Omega=\mathbb{T}\times(0,1)\). 
The tangential energy estimates consist of time and time-horizontal mixed derivatives estimates. The elliptic estimates involve normal and tangential-normal mixed derivatives, in which each normal derivative \(\partial_2\) cancels a weight \(\sqrt{\rho_0}\), which explains the reason why we take the spatial derivatives up to the eighth-order to reach the regularity \((*2)\) of the velocity.

To obtain a classical solution to \((*1)\), we start with constructing approximate solutions \(\{X^n\}_{n=1}^\infty\) (which converge to the weak solution \(v^{L}\)) to the linearized problem (see \eqref{existence-3}) of \((*1)\) by the Galerkin's scheme.
First, let \(\{e_n\}_{n=1}^\infty\) be an orthogonal basis of \(H^1(\Omega)\).
Assume that 
%\(\{X^n\}_{n=1}^\infty\) take the form
\begin{equation*}
	\begin{aligned}
		(X^i)^n(t,x)=\sum_{m_1=1}^n(\lambda_{m_1}^i)^n(t)e_{m_1}(x), \quad n=1,2,\dots,
	\end{aligned}
\end{equation*}
which satisfy
\begin{equation*}\tag{$*3$}
	\begin{aligned}
		&\big(\rho_0\partial_t(X^i)^n,e_{m_2}\big)
		+\big(\rho_0\bar{J}^{-2}\bar{a}_l^k\bar{a}_l^j(X^i)^n,_{j},(e_{m_2}),_{k}\big)\\
		&=\big(\rho_0^2 \bar{J}^{-2}\bar{a}_i^k,(e_{m_2}),_{k}\big),
		\quad  m_2=1,2,\dots, n.
	\end{aligned}
\end{equation*}
To get the convergence of \(\{X^n\}_{n=1}^\infty\), as usual, one needs to show that \(\|\rho_0 X_t^n\|_{L^2([0,T];H^{-1}(\Omega))}\) are uniformly bounded  in \(n\), that is   
\begin{equation*}\tag{$*4$}
	\begin{aligned}
		|\langle\rho_0X_t^n, \phi\rangle|\leq C
	\end{aligned}
\end{equation*}
for any \(\phi\in H^1(\Omega)\) with \(\|\phi\|_{H^1(\Omega)}\leq 1\). 
Since \((*3)\) degenerates in the time evolution, it seems hard to verify \((*4)\) for \(\{X^n\}_{n=1}^\infty\). In fact, decompose \(H^1(\Omega)\) as
\begin{equation*}
	\begin{aligned}
		H^1(\Omega)=\text{span}\{e_m\}_{m=1}^n\oplus \text{span}\{e_m\}_{m=n+1}^\infty:\quad
		\phi=\phi_1+\phi_2.
	\end{aligned}
\end{equation*}
Thus it holds that
\begin{equation*}
	\begin{aligned}
		\langle\rho_0\partial_t(X^i)^n, \phi\rangle=\big(\rho_0\partial_t(X^i)^n,\phi_1\big)
		+\big(\rho_0\partial_t(X^i)^n,\phi_2\big)=\colon \mathcal{I}_1+\mathcal{I}_2. 
	\end{aligned}
\end{equation*}
Note that \(\mathcal{I}_1\) can be estimated by \((*3)\),  however, there is no enough information to estimate \(\mathcal{I}_2\) since \(\phi_2\) and  \(\text{span}\{e_m\}_{m=1}^n\) are not perpendicular in \(L_{\rho_0}^2(\Omega)\). 

Next, to get around this difficulty, the key idea is to find an orthogonal basis \(\{e_n\}_{n=1}^\infty\) in a larger weighted Hilbert space \(H_{\rho_0}^1(\Omega)\) (\(\supseteq H^1(\Omega)\)) with the property
\begin{equation*}
	\begin{aligned}
		H_{\rho_0}^1(\Omega)=\text{span}\{e_m\}_{m=1}^n\oplus \text{span}\{e_m\}_{m=n+1}^\infty:\quad
\phi=\phi_1+\phi_2
	\end{aligned}
\end{equation*}
and 
\begin{equation*}
	\begin{aligned}
		\big(\rho_0e_m,\phi_2\big)=0\quad \mbox{for}\  m=1,2,\dots, n.
	\end{aligned}
\end{equation*}
In this case, the corresponding \(\mathcal{I}_2\)  will vanish. 

Then, to construct such an orthogonal basis of \(H_{\rho_0}^1(\Omega)\),  we introduce the \emph{degenerate-singular} elliptic operator
\begin{align*}
  \mathcal{L}w:=-\frac{\mathrm{div}(\rho_0Dw)}{\rho_0}+w \quad \text{in}\ \Omega,
\end{align*}
and show that the elliptic problem
\begin{equation*}
\begin{aligned}
\mathcal{L}w=g \quad \mathrm{in}\ \Omega
\end{aligned}
\end{equation*}
admits a unique weak solution in \(H_{\rho_0}^1(\Omega)\) for any given \(g\) in certain appropriate sense.  We find that the inverse operator \(\mathcal{L}^{-1}\) is a bounded compact operator on \(L_{\rho_0}^2(\Omega)\), and thus obtain an orthogonal basis \(\{w_n\}_{n=1}^\infty\) of \(H_{\rho_0}^1(\Omega)\) by solving 
the eigenvalue problems 
\begin{equation*}
\begin{aligned}
\mathcal{L}w_n=\sigma_nw_n\quad  \mathrm{in}\  \Omega.
\end{aligned}
\end{equation*}  
Consequently, \(\{w_n\}_{n=1}^\infty\) is our desired choice \(\{e_n\}_{n=1}^\infty\). 

On the other hand, since \((*3)\) also degenerates in the dissipation, it seems difficult to improve the spatial regularity of \(v^{L}\)  by obtaining uniform higher-order estimates for \(\{X^n\}_{n=1}^\infty\) in view of \((*3)\) straightforwardly as in the standard parabolic theory. To overcome this difficulty, we derive an equation for \(v^{L}\), which holds in \(\Omega\times (0,T]\) in the almost everywhere sense. 

Due to the degeneracy of \((*1)\), the iterative solutions \(\{v^{(n)}\}_{n=1}^\infty\) to \eqref{existence-21} satisfy as \( n\to\infty\)
{\small{\begin{equation*}
	\begin{aligned}
\underset{0\leq t\leq T}{\sup}\|\sqrt{\rho_0}(v^{(n+1)}-v^{(n)})(t)\|_{L^2(\Omega)}
+\|\sqrt{\rho_0}(Dv^{(n+1)}-Dv^{(n)})\|_{L^2([0,T];L^2(\Omega))}
\to 0,
	\end{aligned}
\end{equation*}}}
	
\noindent which implies only that \(\{v^{(n)}\}_{n=1}^\infty\) is a Cauchy sequence in \(L^2([0,T]; L^2(\Omega))\). 
However, this is insufficient to pass limit in \(n\) for \eqref{existence-21} in time pointwisely. To get over this difficulty, we need some weighted interpolation 
inequality (see Lemma \ref{le:W-9}) to obtain a pointwise convergence in time on \(\{v^{(n)}\}_{n=1}^\infty\).

Finally, we comment how to extend the result on \(\mathbb{T}\times(0,1)\) to a general domain \(\Omega\). 
First, we can construct approximate solutions on \(\Omega\) by the Galerkin’s scheme since it does not need to distinguish the tangential and normal derivatives. Next,  we flatten \(\partial\Omega\) by changing coordinates and transform \((*1)\) to \eqref{existence-100} in the new coordinates. The key observation is that \eqref{existence-100} enjoys the same degenerate structure with \((*1)\), which motives us to design a boundary energy functional \eqref{HOEF-102} for \eqref{existence-100} to carry out high-order estimates for the weak solution in a similar way as  \(\mathbb{T}\times(0,1)\).  One can refer to \cite{MR1780703, MR2388661} for other methods 
to handle free boundary problems on general domains. \\

{\underline{\emph{Outline of the paper}.}}
In Section \ref{main results}, we formulate \eqref{eq:intro-vfb}-\eqref{eq:intro-3} into a fixed boundary problem \eqref{eq:main-2} and 
state our main results. Section \ref{Some Preliminaries} lists some preliminaries: various kinds of weighted inequalities and a weighted interpolation inequality, which play important roles in the a priori estimates and the convergence of the iterative solutions sequence. Section \ref{sec:Jab} presents some useful a priori bounds on \(J^{-2}b^{kj}\) and \(J^{-2}a_i^k\), which together with the use of some weighted inequalities can simplify significantly the a priori estimates.  
In Sections \ref{Energy Estimates} and \ref{Elliptic Estimates}, respectively, we focus on the a priori estimates that consist of the energy
estimates and elliptic estimates, where the strategies handling tangential and normal derivatives are explored in detail. Sections \ref{The Linearized Problem}-\ref{Regularity} are devoted to the study on the existence, uniqueness and regularity of weak solutions to the linearized problem \eqref{existence-3}. More precisely, Section \ref{Degenerate-Singular Elliptic Operators} introduces some new degenerate-singular elliptic operator to construct an orthogonal basis of the projection space for the use of the Galerkin's scheme to obtain a weak solution in Section \ref{weak solution};  Section \ref{Regularity} describes the modifications which are necessary for improving the regularity of the weak solution obtained in Section \ref{weak solution}, comparing with the a priori estimates established in Sections \ref{Energy Estimates} and \ref{Elliptic Estimates}. Then Theorem \ref{th:main-1} is proved in 
Sections \ref{Existence Part} and \ref{Uniqueness Part}. In Section \ref{general domain case}, we
extend Theorem \ref{th:main-1} to a general domain.

\section{Reformulation and Main Results}\label{main results}

\subsection{Fixing the domain}
From here until Section \ref{Uniqueness Part}, we will focus on the reference domain being a slab given by \(\Omega(0)=\mathbb{T}\times(0,1)\), where \(\mathbb{T}=\mathbb{R}/\mathbb{Z}\) denotes the 1-torus,
which permits the use of one global Cartesian coordinate system. 
We will write \(\Omega(0)\) as \(\Omega\), and then \(\partial \Omega\) as \(\Gamma\),  for convenience. The case of general domains will be dealt in Section \ref{general domain case}. 

Denote by \(\eta\) the position of the fluid particle \(x\) at time \(t\)
\begin{equation}\label{fluid particle}
\begin{cases}
\partial_t\eta(x,t)=u(\eta(x,t),t),\\
\eta(x,0)=x.
\end{cases}
\end{equation}
and then define  Lagrangian variables as in \cite{MR2980528}
\begin{equation}\label{Lagrangian variable}
\begin{aligned}
&f(x,t)=\rho(\eta(x,t),t),\quad  
v(x,t)=u(\eta(x,t),t), \\
&J=\det [D\eta],\quad
A=[D\eta]^{-1},\quad
a=JA.
\end{aligned}
\end{equation}
Then   \eqref{eq:intro-vfb} is transformed to 
\begin{equation}\label{eq:main-1}
\begin{cases}
\partial_tf+fA_i^jv^i,_{ j}=0 &\quad \text{in}\ \Omega\times (0,T],\\
f\partial_tv^i+A_i^k(f^2),_{k}=A_l^k(f A_l^jv^i,_{j}),_{k} &\quad \text{in}\ \Omega\times (0,T],\\
f>0 &\quad \text{in}\ \Omega\times (0,T],\\
f=0 &\quad \text{on}\ \Gamma\times (0,T],\\
(f,v,\eta)=(\rho_0,u_0,e) &\quad \text{on}\ \Omega\times \{t=0\},
\end{cases}
\end{equation}
where \(e(x)=x\), \(F,_{k}:=\partial_kF=\frac{\partial F}{\partial_{x_k}}\), and we have used Einstein's summation convention: repeated \(i, j, k\), etc., are summed from 1 to 2.

Noting that \(J_t=JA_i^jv^i,_{j}\),  one  solves \(f\) from \(\eqref{eq:main-1}_1\) to get
\begin{align*}
f(x,t)=\rho_0(x)J^{-1}(x,t).
\end{align*}
This, together with \(A_i^k=J^{-1}a_i^k\) and \(\eqref{eq:main-1}_2\), implies 
\begin{equation}\label{eq:main-2}
\begin{cases}
\rho_0\partial_tv^i+a_i^k(\rho_0^2 J^{-2}),_{k}=a_l^k(\rho_0J^{-2}a_l^jv^i,_{j}),_{k} &\quad \mbox{in}\ \Omega\times (0,T],\\
(v,\eta)=(u_0,e) &\quad \mbox{on}\ \Omega\times \{t=0\}.
\end{cases}
\end{equation}

\eqref{eq:main-2} is a degenerate quasilinear parabolic problem.

%\begin{definition}[Classical Solution]\label{Classical Solution} {\rm{(a)}}  \(v\)  
%	is said to be a classical solution to  \eqref{eq:main-2} if \(v\) satisfies  \(\eqref{eq:main-2}_1\) in \(\Omega\times (0,T]\) pointwisely and takes on the initial data  continuously.
	
%{\rm{(b)}} The triple  \((\rho(x,t), u(x,t), \Gamma(t))\), for \(t\in[0,T]\),   
%is said to be a classical solution to \eqref{eq:intro-vfb} provided  that \((\rho(x,t), u(x,t), \Gamma(t))\) satisfies  \(\eqref{eq:intro-vfb}_1-\eqref{eq:intro-vfb}_5\) in \(\Omega(t)\) pointwisely  and takes on the initial data  continuously. 

%\end{definition}

\subsection{The higher-order energy functional}
To study \eqref{eq:main-2} in certain weighted Sobolev space with high regularity, we will consider the following higher-order energy functional:
{\small{\begin{equation}\label{HOEF-0}
\begin{aligned}
E(t,v)&=\sum_{l_0=0}^4\|\sqrt{\rho_0}\partial_t^{l_0}v\|_{L^2}^2+\sum_{\substack{2l_0+l_1\leq 6\\ l_0,\ l_1\geq 0}}(\|\sqrt{\rho_0}\partial_t^{l_0}\partial_1^{l_1}Dv\|_{L^2}^2+\|\sqrt{\rho_0}\partial_t^{l_0}\partial_1^{l_1+1}Dv\|_{L^2}^2)\\
&\quad+\sum_{\substack{2l_0+l_1+l_2\leq 8\\  l_0,\ l_1\geq 0,\ l_2\geq 2}}\big\|\sqrt{\rho_0^{l_2}}\partial_t^{l_0}\partial_1^{l_1}\partial_2^{l_2}v\big\|_{L^2}^2.
\end{aligned}
\end{equation}}}

\noindent The first line of the right-hand side of \eqref{HOEF-0} is the energy estimates for tangential derivatives, while the second line is the elliptic estimates that involve higher-order normal derivatives, which will be 
denoted as \(E_{\text{en}}(t,v)\) and \(E_{\text{el}}(t,v)\), respectively.

Define a polynomial function \(M_0\) by
\begin{align*}
M_0=P(E(0,v_0)),
\end{align*}
where \(P\) denotes a generic polynomial function of its argument.

\subsection{Main result on the fixed domain problem}

The main result in the paper can be stated as follows:
\begin{theorem}\label{th:main-1} Let \((\rho_0,v_0)\) satisfy \eqref{eq:intro-3} and \(E(0,v_0)<\infty\).  Then
there exist a suitably small \(T>0\) and a unique classical solution 
\begin{align}\label{regularity}
v\in C([0,T]; H^4(\Omega))\cap C^1([0,T]; H^2(\Omega))
\end{align}
 to  \eqref{eq:main-2} such that
\begin{align}\label{eq:inequality-1}
\sup_{0\leq t\leq T} E(t,v)\leq 2M_0.
\end{align}

Moreover, \(v\) satisfies the boundary condition
\begin{align}\label{Nuewmann boundary condition}
(\rho_0),_{k}a_l^ka_l^jv^i,_{j}=0 \quad \rm{on}\ \Gamma\times (0,T].
\end{align}

\end{theorem}

\subsection{Main result on the vacuum free
	boundary problem}

Due to \eqref{J-bound}, the flow map \(\eta(\cdot,t)\colon \Omega\rightarrow \Omega(t)\) is invertible for any \(t\in[0,T]\) and we denote its inverse by \(\tilde{\eta}(\cdot,t)\colon \Omega(t)\rightarrow \Omega\), where \(T\) is determined in Theorem \ref{th:main-1}. Let \((\eta,v)\) be the unique classical solution in Theorem \ref{th:main-1}. For \(t\in[0,T]\) and \(y\in \Omega(t)\), set
\begin{equation*}
\begin{aligned}
\rho(y,t)=\rho_0(\tilde{\eta}(y,t))\eta_x^{-1}(\tilde{\eta}(y,t),t),\
u(y,t)=v(\tilde{\eta}(y,t),t).
\end{aligned}
\end{equation*}
Then the triple \((\rho(y,t), u(y,t), \Gamma(t))\) \((t\in[0,T])\) defines a unique classical solution to  \eqref{eq:intro-vfb}. More precisely, Theorem \ref{th:main-1} yields

\begin{theorem}\label{th:main-2} Let \((\rho_0,u_0)\) satisfy \eqref{eq:intro-3} and \(E(0,v_0)<\infty\). Then
there exist a \(T>0\) and a unique classical solution \((\rho(y,t), u(y,t), \Gamma(t))\) for \(t\in[0,T]\) to  \eqref{eq:intro-vfb}. Moreover, \(\Gamma(t)\in C^2([0,T])\) and 
\begin{equation}\label{regularity-2}
\begin{aligned}
&\rho(y,t)\in C([0,T];H^4(\Omega(t)))\cap C^1([0,T];H^3(\Omega(t))),\\
&u(y,t)\in C([0,T];H^4(\Omega(t)))\cap C^1([0,T];H^2(\Omega(t)))
\end{aligned}
\end{equation}
for \(t\in[0,T]\) and \(y\in \Omega(t)\).

Furthermore, \(u\) satisfies the boundary condition
\begin{align}\label{Nuewmann boundary condition-2}
\nabla\rho\cdot\mathbb{D}(u)=0 \quad \rm{on}\ \Gamma(t).
\end{align}

\end{theorem}

\subsection{Remarks}

There are a few remarks in order:

\begin{remark}\label{re:main-0}

It follows from \eqref{regularity} and  \eqref{regularity-2} that for any \(\gamma\in (0,1)\),
	\begin{align*}
	&v\in C([0,T]; C^{2,\gamma}(\Omega))\cap C^1([0,T]; C^{0,\gamma}(\Omega)),\\
	&\rho\in C([0,T]; C^{2,\gamma}(\Omega(t)))\cap C^1([0,T]; C^{1,\gamma}(\Omega(t))),\\
	&u\in C([0,T]; C^{2,\gamma}(\Omega(t)))\cap C^1([0,T]; C^{0,\gamma}(\Omega(t))).
	\end{align*} 
\end{remark}

\begin{remark} Given \(\rho_0\) and \(u_0\), one can define \(\partial_tv|_{t=0}\)  by the following compatibility conditions through
\eqref{eq:main-2}: 
\begin{equation*}
	\begin{aligned}
		\partial_tv^i|_{t=0}=\rho_0^{-1}[(\rho_0
		(v_0)^i,_1),_1+(\rho_0
		(v_0)^i,_2),_2-(\rho_0^2),_{i}]
	\end{aligned}
\end{equation*}
and similarly when \(l_0=2,3,4\) 
\begin{equation*}
	\begin{aligned}
		\partial_t^{l_0}v^i|_{t=0}=\rho_0^{-1}\partial_t^{l_0-1}[a_l^k(\rho_0J^{-2}a_l^jv^i,_{j}),_{k}
		-a_i^k(\rho_0^2 J^{-2}),_{k}]|_{t=0}. 
	\end{aligned}
\end{equation*}
The other terms involving initial time derivatives in \(E(0,v_0)\) can be defined analogously.

\end{remark}

\begin{remark}\label{re:main-1}
 In \cite{MR2980528}, Coutand and Shkoller established the well-posedness of the physical vacuum free boundary problem for the isentropic compressible Euler equations, which may be written in the Lagrangian coordinates as
\begin{align}\label{r10}
	\rho_0\partial_tv^i+a_i^k(\rho_0^\gamma J^{-\gamma}),_{k}=0\quad \mathrm{in}\ \Omega\times (0,T].
\end{align}
%It is nontrivial to show the existence of solutions to \eqref{r10} due to its degenerate hyperbolicity near vacuum. The authors used a degenerate parabolic regularization method which matches well with the a priori estimates \cite{MR2608125}. 
The key point of their theory is the analysis for the case that \(\gamma=2\), in which the authors considered the following degenerate parabolic
regularization: 
\begin{equation}\label{r11}
\begin{aligned}
\rho_0\partial_tv^i+a_i^k(\rho_0^2 J^{-2}),_{k}
=\kappa \partial_t[a_i^k(\rho_0^2 J^{-2}),_{k}]\quad \mathrm{in}\ \Omega\times (0,T]
\end{aligned}
\end{equation}
for small \(\kappa>0\). 
The special structure of the artificial viscosity in \eqref{r11} does not break the structure of the a priori estimates for \eqref{r10} established in \cite{MR2608125}, and thus the solution \(v\) to \eqref{r10} will be obtained as a limit of the solutions \(v_\kappa\) to \eqref{r11} when \(\kappa\to 0\). To make use of the Galerkin's scheme to construct approximate solutions of \eqref{r11}, the authors considered a good unknown \(X=\rho_0 J^{-2}\mathrm{div}_{\eta}v\) which satisfies a nonlinear degenerate heat-type equation.

%that satisfies the Dirichlet boundary condition \(X=0\) on \(\Gamma\) (which plays a role in eliminating boundary terms in integration by parts). 

Yet, such a crucial idea seems not working for \eqref{eq:main-2} due to the different dissipation structures between \eqref{eq:main-2} and \eqref{r11}. Constructing an orthogonal basis of \(H_{\rho_0}^1(\Omega)\) for the use of the Galerkin's scheme to obtain approximate solutions
and designing a higher-order energy functional \eqref{HOEF-0} to gain high-order regularities of the weak solution constitute the main contributions of the present work.

%Different with \eqref{eq:main-2}, the classical solutions to \eqref{r11} (and also \eqref{r10}) do not satisfy necessarily any boundary conditions.   
%The normal stress free condition in \(\eqref{eq:intro-vfb}_4\) holds automatically at the vacuum boundary. 
%\eqref{Nuewmann boundary condition}

Different with \eqref{r11} (whose classical solutions do not satisfy necessarily any boundary conditions),
any classical solution  to \eqref{eq:main-2} with the regularity \eqref{eq:inequality-1} must satisfy \eqref{Nuewmann boundary condition}.  
	Indeed, note that \(\eqref{eq:main-2}_1\) implies
	\begin{equation}\label{r1}
	\begin{aligned}
		&\rho_0\partial_tv^i+2J^{-2}(\rho_0),_{k}\rho_0a_i^k-2J^{-3}
		\rho_0^2a_i^kJ,_{k}\\
		&=J^{-2}(\rho_0),_{k}a_l^ka_l^jv^i,_{j}
		+\rho_0a_l^k(-2J^{-3}J,_{k}a_l^jv^i,_{j}
		+J^{-2}a_l^j,_{k}v^i,_{j}+J^{-2}a_l^jv^i,_{jk})
	\end{aligned}
	\end{equation}
	for \((x,t)\in \ \Omega\times (0,T]\). 	
	It follows from \eqref{eq:inequality-1}, Lemmas \ref{le:Preliminary-1} and \ref{le:W-1} that
	\begin{align*}
		\rho_0\partial_tv(t),\ Dv(t),\ D\eta(t),\ \rho_0D^2v(t),\ \rho_0D^2\eta(t)\in H^3(\Omega)\quad  \mathrm{for}\ t\in (0,T],
	\end{align*}
	which, combined with \(H^3(\Omega)\hookrightarrow H^{5/2}(\Gamma)\), yields 
	\begin{align*}
		\rho_0\partial_tv(t),\ Dv(t),\ D\eta(t),\ \rho_0D^2v(t),\ \rho_0D^2\eta(t)\in H^{5/2}(\Gamma)\quad  \mathrm{for}\ t\in (0,T].
	\end{align*}
	This together with \eqref{J-formula}, \eqref{a-formula} and \eqref{Eta-formula}
	implies that \eqref{r1} is well-defined pointwisely on \( \Gamma\times (0,T]\).
	Using \eqref{eq:intro-3}, \eqref{J-bound}, and  letting \(x\) go to \(\Gamma\) in \eqref{r1}, then one obtains \eqref{Nuewmann boundary condition}.

\end{remark}

\section{Preliminaries}\label{Some Preliminaries}

\subsection{Notations} Denote by \(\diff x\)  the Lebesgue measure \(\diff x_1 \diff x_2\), and  by \(D\) the gradient operator \((\partial_1,\partial_2)\).
Let \(H^s(\Omega)\), $C([0,T]; H^s(\Omega))$ and $L^p([0,T]; H^s(\Omega))$  be the standard Sobolev spaces. \(\partial_1\) denotes the horizontal derivative, and \(N=(0,1)\) is the outward unit normal vector to \(\Gamma\).

For convenience, we often use \(u(t):=u(x,t)\) to highlight its dependence on time \(t\). Furthermore, we also frequently write 
\(\|u(t)\|_{L^2(\Omega)}\) as \(\|u\|_{L^2}\), and \(\|u\|_{L^2([0,t]; L^2(\Omega))}\) as \(\|u\|_{L_t^2L^2}\), if there is no confusion. 

\(C\) denotes a nonnegative generic constant which may be different from line to line but is
independent of the solutions. 
We write \(g\lesssim h\) (\(g\gtrsim h\)) when \(g\leq  Ch\) (\(g\geq  Ch\)), and \(g \sim h\) when \(g \lesssim h \lesssim g\).

\subsection{Weighted Sobolev inequalities} Let \(d(x)=\colon d(x,\Gamma)\) be the distance function to \(\Gamma\).
To handle the degeneracy near \(\Gamma\), we will need the following weighted Sobolev inequalities, whose proofs can be found for instance in \cite{MR0802206}.  

For given \(\alpha>0\) and nonnegative integer \(\beta\),  \(H^{\alpha, \beta}(\Omega)\) denotes the weighted Sobolev space
\begin{equation*}
\begin{aligned}
H^{\alpha, \beta}(\Omega)=\bigg\{d^\frac{\alpha}{2}w\in L^2(\Omega):  \int_\Omega d^\alpha(x)|D^\theta w|^2(x)\,\diff x<\infty\quad \text{for}\ 0\leq\theta\leq \beta \bigg\}
\end{aligned}
\end{equation*}
with the norm
\begin{equation*}
\begin{aligned}
\|w\|_{H^{\alpha, \beta}(\Omega)}^2=\sum_{\theta=0}^\beta \int_\Omega d^\alpha(x)|D^\theta w|^2(x)\,\diff x.
\end{aligned}
\end{equation*}
Then the following Hardy-type embedding
\begin{equation}\label{ineq:weighted Sobolev-0}
\begin{aligned}
H^{\alpha, \beta}(\Omega)\hookrightarrow H^{\beta-\frac{\alpha}{2}}(\Omega)
\end{aligned}
\end{equation}
holds for \(\beta\geq \alpha/2\), in particular, one has
\begin{equation}\label{ineq:weighted Sobolev-1}
\begin{aligned}
\|w\|_{H^{1/2}(\Omega)}^2\lesssim \int_\Omega d(x)(w^2+|Dw|^2)(x)\,\diff x.
\end{aligned}
\end{equation}

Also the following weighted Sobolev inequality holds:
\begin{equation}\label{ineq:weighted Sobolev-2}
\begin{aligned}
\int_\Omega d^{\alpha}(x)w^2(x)\,\diff x\lesssim \int_\Omega d^{\alpha+2}(x)(w^2+|Dw|^2)(x)\,\diff x\quad \mbox{for}\ \alpha=0,1,2,\dots .
\end{aligned}
\end{equation}

Note that \eqref{ineq:weighted Sobolev-1} and
\eqref{ineq:weighted Sobolev-2} have nothing to do with spatial dimensions. Besides,
it holds that
\begin{equation}\label{ineq:weighted Sobolev-3}
	\begin{aligned}
		\|w\|_{L^\infty(\Omega)}^2\lesssim \sum_{\theta=0}^2 \int_\Omega d(x)|D^\theta w|^2(x)\,\diff x,
	\end{aligned}
\end{equation}
which follows from \(H^{3/2}(\Omega)\hookrightarrow L^\infty(\Omega)\) and  \eqref{ineq:weighted Sobolev-0} (with \(\alpha=1\ \text{and}\ \beta=2\)). 

Due to \eqref{eq:intro-3},  \(d(x)\) can be replaced by \(\rho_0(x)\) in \eqref{ineq:weighted Sobolev-1}-\eqref{ineq:weighted Sobolev-3}.

\subsection{A tangent estimate on initial depth profile}
The following inequality is a consequence of the higher-order Hardy-type inequality \cite{MR2980528}
\begin{equation}\label{TE}
\begin{aligned}
 \|\partial_1^l(\rho_0)/\rho_0\|_{L^\infty(\Omega)}\leq C\quad \text{for}\ l=1,2,\dots .
\end{aligned}
\end{equation}

\subsection{Consequences of the energy functional}\label{consequency of higher-order energy function}

As a prerequisite for later use, we will use \eqref{ineq:weighted Sobolev-2} to deduce some useful consequences of \eqref{HOEF-0}. 
\begin{lemma}\label{le:Preliminary-1} It holds that
	\begin{align}\label{Preliminary-1}
		\|D^2v\|_{L^\infty}\lesssim\|v\|_{H^4}\lesssim E^{1/2}(t,v).
	\end{align}
Consequently, if \eqref{fluid particle} and \eqref{Lagrangian variable} hold, then
\begin{align}\label{Preliminary-2}
		\|D^2\eta\|_{L^\infty}\lesssim \|D^2\eta\|_{H^2}\lesssim t \sup_{0\leq s\leq t}E^{1/2}(s,v).
	\end{align}
\end{lemma}
\begin{proof} To show \eqref{Preliminary-1}, it suffices to prove its second inequality. Note that one needs only to show the highest-order estimate since other cases can be proven similarly. Indeed, it follows from \eqref{ineq:weighted Sobolev-2} and  \eqref{HOEF-0} that
		\begin{equation*}
			\begin{aligned}
				&\int_\Omega |D^4v|^2\,\diff x\lesssim \int_\Omega\rho_0^2(|D^4v|^2+|D^5v|^2)\,\diff x\\
				&\lesssim \int_\Omega\rho_0^4(|D^4v|^2+|D^5v|^2)\,\diff x
				+\int_\Omega\rho_0^4(|D^5v|^2+|D^6v|^2)\,\diff x\\
				&\lesssim E(t,v)+\int_\Omega\rho_0^6(|D^5v|^2+|D^6v|^2)\,\diff x+\int_\Omega\rho_0^6(|D^6v|^2+|D^7v|^2)\,\diff x\\
				&\lesssim E(t,v)+\int_\Omega\rho_0^8(|D^7v|^2+|D^8v|^2)\,\diff x
				\lesssim E(t,v).
			\end{aligned}
		\end{equation*}

\eqref{Preliminary-2} is a consequence of \eqref{Preliminary-1} and
	\begin{equation}\label{eta-L2}
		\begin{aligned}
			\|D^\theta\eta\|_{L^2}\leq \int_0^t\left(\int_\Omega |D^\theta v|^2\,\diff x\right)^{1/2}\diff s
			\lesssim t \sup_{0\leq s\leq t}E^{1/2}(t,v),\ \theta=2,3,4,
		\end{aligned}
	\end{equation}
where one has used Minkowski's inequality in the first inequality.	
	
\end{proof}

\subsection{An equivalent form of the main equations}

Note that
\begin{align}\label{J-formula}
J=\eta^1,_{1}\eta^2,_{2}-\eta^1,_{2}\eta^2,_{1},
\end{align}
\begin{equation}\label{a-formula}
a=
\begin{bmatrix}
\eta^2,_{2}& -\eta^1,_{2}\\
-\eta^2,_{1}&\eta^1,_{1}
\end{bmatrix}.
\end{equation}
Each column of  \(a\) 
satisfies the following Piola's identity:
\begin{equation}\label{Piola}
\begin{aligned}
a_i^k,_{k}=0,\quad i=1,2.
\end{aligned}
\end{equation}

Set \(b^{kj}=a_l^ka_l^j\).
By \eqref{Piola},   $\eqref{eq:main-2}_1$ becomes
\begin{align}\label{eq:main-3}
\rho_0\partial_tv^i+(\rho_0^2 J^{-2}a_i^k),_{k}=(\rho_0J^{-2}b^{kj}v^i,_{j}),_{k}.
\end{align}

\subsection{A priori assumptions}\label{The a priori assumption} Let \(c_1\) and \(c_2\) be the constants in the first and second inequality of \eqref{Preliminary-1}, respectively. Set \(M_1=2M_0\).
Let \((v,\eta)\) satisfy \eqref{fluid particle} and \eqref{Lagrangian variable}. 

Assume that there exists some suitably small \(T\in(0,1/(20c_1c_2\sqrt{M_1}+10c_1^2c_2^2M_1)]\cap (0,1)\) such that 
\begin{align}\label{a priori assumption}
	\sup_{0\leq t\leq T}E(t,v)\leq M_1.
\end{align}
Then one has 
\begin{align}\label{J-bound}
	9/10\leq J(x,t)\leq 11/10\quad \text{for}\ (x,t)\in \Omega\times [0,T].
\end{align}
Indeed, it follows from \eqref{fluid particle} by direct calculations that
\begin{equation}\label{Eta-formula}
D\eta(x,t)=
\begin{bmatrix}
1+\int_0^tv^1,_{1}(x,s)\,\diff s& \int_0^tv^2,_{1}(x,s)\,\diff s\\
\int_0^tv^1,_{2}(x,s)\,\diff s&1+\int_0^tv^2,_{2}(x,s)\,\diff s
\end{bmatrix}.
\end{equation}
By means of \eqref{J-formula},  \eqref{a priori assumption} and \eqref{Eta-formula}, one gets that
\begin{equation*}
\begin{aligned}
	|J(x,t)-1|&\leq 2\int_0^t\|Dv(s)\|_{L^\infty(\Omega)}\,\diff s+\bigg(\int_0^t\|Dv(s)\|_{L^\infty(\Omega)}\,\diff s\bigg)^2\\
&\leq 2c_1c_2T\sup_{0\leq t\leq T}E^{1/2}(t,v)+c_1^2c_2^2T^2\sup_{0\leq t\leq T}E(t,v)\\
&\leq (2c_1c_2\sqrt{M_1}+c_1^2c_2^2M_1)T
\leq 1/10 \quad \text{for}\ (x,t)\in \Omega\times [0,T].
\end{aligned}
\end{equation*}
Hence \eqref{J-bound} follows. 

On the other hand, it follows from \eqref{Eta-formula} and \eqref{a-formula} that
\begin{equation}\label{b-formula}
\begin{aligned}
&b^{11}=1+2\int_0^tv^2,_{2}\,\diff s+\bigg(\int_0^tv^2,_{2}\,\diff s\bigg)^2+\bigg(\int_0^tv^1,_{2}\,\diff s\bigg)^2,\\
&b^{12}=b^{21}=-\int_0^t(v^2,_{1}+v^1,_{2})\,\diff s-\int_0^tv^1,_{1}\,\diff s\int_0^tv^1,_{2}\,\diff s,\\
&\qquad\qquad\quad-\int_0^tv^2,_{1}\,\diff s\int_0^tv^2,_{2}\,\diff s,\\
&b^{22}=1+2\int_0^tv^1,_{1}\,\diff s+\bigg(\int_0^tv^1,_{1}\,\diff s\bigg)^2+\bigg(\int_0^tv^2,_{1}\,\diff s\bigg)^2.
\end{aligned}
\end{equation}
By \eqref{b-formula} and  \(T\in(0,1/(20c_1c_2\sqrt{M_1}+10c_1^2c_2^2M_1)]\cap (0,1)\), one finds that the minimum eigenvalue
of \((b^{kj})\) is bounded below by \(1/5\), which implies that
\begin{align}\label{a-bound}
	b^{kj}(x,t)\xi_k\xi_j\geq |\xi|^2/5 \quad \text{for}\ (x,t)\in \Omega\times [0,T]\ \text{and}\ \ \xi\in \R^2. 
\end{align}
In particular,   
\begin{align}\label{b-bound}
	b^{22}(x,t)\geq 1/5 \quad \text{for}\ (x,t)\in \Omega\times [0,T]. 
\end{align}

%\begin{remark}\label{determin M_1 and T} The a priori assumption \eqref{a priori assumption} will be closed by the a priori bound [].
%\end{remark}

\subsection{Some Weighted Inequalities}\label{Some Weighted inequalities}

This subsection lists some weighted inequalities which will be frequently used later on.

\subsubsection{\(L^2\) and \(L^\infty\)-type inequalities}\label{sub:W-1}
\begin{lemma}[\(L^2\)-estimate]\label{le:W-1} It holds that
 
\begin{equation}
\begin{aligned}\label{Weighted-1}
	&\sum_{l_2=0}^1\bigg(\sum_{l_1=0}^4\|\rho_0^{l_2}\partial_{1}^{l_1}D^{l_2+2}v\|_{L^2}
	+\sum_{l_1=0}^2\|\rho_0^{l_2}\partial_{1}^{l_1}D^{l_2+3}v\|_{L^2}\bigg) \lesssim E^{1/2}(t,v),\\
	&\sum_{l_1=0}^2\|\rho_0^3\partial_{1}^{l_1}D^6v\|_{L^2}+\sum_{l_2=1}^4\|\rho_0^{l_2}D^{l_2+4}v\|_{L^2}
	\lesssim E^{1/2}(t,v).
	\end{aligned}
\end{equation}
Consequently, if \eqref{fluid particle} and \eqref{Lagrangian variable} hold, then
\begin{equation}
\begin{aligned}\label{Weighted-2}
	&\sum_{l_2=0}^1\bigg(\sum_{l_1=0}^4\|\rho_0^{l_2}\partial_{1}^{l_1}D^{l_2+2}\eta\|_{L^2}
	+\sum_{l_1=0}^2\|\rho_0^{l_2}\partial_{1}^{l_1}D^{l_2+3}\eta\|_{L^2}\bigg)\lesssim t\sup_{0\leq s\leq t}E^{1/2}(s,v),\\
	&\sum_{l_1=0}^2\|\rho_0^3\partial_{1}^{l_1}D^6\eta\|_{L^2}+\sum_{l_2=1}^4\|\rho_0^{l_2}D^{l_2+4}\eta\|_{L^2}
	\lesssim t\sup_{0\leq s\leq t}E^{1/2}(s,v).
	\end{aligned}
\end{equation}

\end{lemma}

\begin{proof} It suffices to show \eqref{Weighted-1} due to \eqref{eta-L2}. 
We only prove \(\eqref{Weighted-1}_1\) for \(l_2=0\). It follows from  \eqref{HOEF-0} and \eqref{ineq:weighted Sobolev-2} that
	\begin{equation*}
	\begin{aligned}
	\int_\Omega |\partial_{1}^{l_1}D^2v|^2\,\diff x
	&\lesssim \int_\Omega\rho_0^2(|\partial_{1}^{l_1}D^2v|^2+|\partial_{1}^{l_1}D^3v|^2)\,\diff x\\
	&\lesssim E(t,v)+\int_\Omega\rho_0^4(|\partial_{1}^{l_1}D^3v|^2+|\partial_{1}^{l_1}D^4v|^2)\,\diff x\\
	&\lesssim E(t,v), \quad l_1=0,1,2,3,4
	\end{aligned}
	\end{equation*}
and
\begin{equation*}
\begin{aligned}
\int_\Omega |\partial_{1}^{l_1}D^3v|^2\,\diff x
&\lesssim \int_\Omega\rho_0^2(|\partial_{1}^{l_1}D^3v|^2+|\partial_{1}^{l_1}D^4v|^2)\,\diff x
\lesssim \cdots \\
&\lesssim E(t,v)+\int_\Omega\rho_0^6(|\partial_{1}^{l_1}D^5v|^2+|\partial_{1}^{l_1}D^6v|^2)\,\diff x\\
&\lesssim E(t,v),\quad l_1=0,1,2.
\end{aligned}
\end{equation*}

\end{proof}

\begin{corollary}[\(L^\infty\)-estimate]\label{le:W-2} It holds that

\begin{equation}
\begin{aligned}\label{Weighted-3}
	&\sum_{l_2=0}^1\bigg(\sum_{l_1=0}^4\|\rho_0^{l_2}\partial_{1}^{l_1}D^{l_2}v\|_{L^\infty}
	+\sum_{l_1=0}^2\|\rho_0^{l_2}\partial_{1}^{l_1}D^{l_2+1}v\|_{L^\infty}\bigg)\lesssim E^{1/2}(t,v),\\
	&\sum_{l_1=0}^2\|\rho_0^3\partial_{1}^{l_1}D^4v\|_{L^\infty}+\sum_{l_2=1}^4\|\rho_0^{l_2}D^{l_2+2}v\|_{L^\infty}
	\lesssim E^{1/2}(t,v).
	\end{aligned}
\end{equation}
Consequently, if \eqref{fluid particle} and \eqref{Lagrangian variable} hold, then
 \begin{equation}
\begin{aligned}\label{Weighted-4}
&\|D\eta\|_{L^\infty}\lesssim 1+t\sup_{0\leq \tau\leq t}E^{1/2}(\tau,v),\\
	&\sum_{l_1=2}^4\|\partial_{1}^{l_1}\eta\|_{L^\infty}
	+\sum_{l_1=1}^2\|\partial_{1}^{l_1}D\eta\|_{L^\infty}
	\lesssim t\sup_{0\leq s\leq t}E^{1/2}(s,v),\\
	&\sum_{l_1=1}^4\|\rho_0\partial_{1}^{l_1}D\eta\|_{L^\infty}
	+\sum_{l_1=0}^2\|\rho_0\partial_{1}^{l_1}D^2\eta\|_{L^\infty}\lesssim t\sup_{0\leq s\leq t}E^{1/2}(s,v),\\
	&\sum_{l_1=0}^2\|\rho_0^3\partial_{1}^{l_1}D^4\eta\|_{L^\infty}+\sum_{l_2=1}^4\|\rho_0^{l_2}D^{l_2+2}\eta\|_{L^\infty}
	\lesssim t\sup_{0\leq s\leq t}E^{1/2}(s,v).
	\end{aligned}
\end{equation}

\end{corollary}

\begin{proof}
\eqref{Weighted-3}  is a direct consequence of \eqref{Weighted-1}. 
\(\eqref{Weighted-4}_1\) follows from 
\begin{equation*}
\begin{aligned}
\|D\eta\|_{L^\infty}\leq 1+\int_0^t\|Dv\|_{L^\infty}\,\diff\tau\lesssim 1+t\sup_{0\leq \tau\leq t}E^{1/2}(\tau,v).
\end{aligned}
\end{equation*}
\(\eqref{Weighted-4}_2\)-\(\eqref{Weighted-4}_5\) can be proven similarly by noting that more than one spatial derivative falling on \(\eta\) will eliminate the constant `1'.  
\end{proof}

\begin{lemma}\label{le:W-3} It holds that 
\begin{equation}\label{Weighted-5}
\begin{aligned}
\|\partial_tDv\|_{L^\infty}\lesssim E^{1/2}(t,v).
\end{aligned}
\end{equation}
\end{lemma}

\begin{proof} One uses \eqref{ineq:weighted Sobolev-2} repeatedly  to deduce that
\begin{equation*}
\begin{aligned}
&\|\partial_tDv\|_{L^\infty}^2
\lesssim \|\partial_tDv\|_{L^2}^2+\|\partial_tD^2v\|_{L^2}^2+\|\partial_tD^3v\|_{L^2}^2\\
&\lesssim \|\rho_0\partial_tDv\|_{L^2}^2+\|\rho_0\partial_tD^2v\|_{L^2}^2+\|\rho_0\partial_tD^3v\|_{L^2}^2+\|\rho_0\partial_tD^4v\|_{L^2}^2
\lesssim \cdots \\
&\lesssim E(t,v)+\|\rho_0^3\partial_tD^5v\|_{L^2}^2+\|\rho_0^3\partial_tD^6v\|_{L^2}^2\lesssim E(t,v).
\end{aligned}
\end{equation*}

\end{proof}

\subsubsection{\(H^{1/2}\)-type inequalities}\label{sub:W-2}

\begin{lemma}\label{le:W-4}
It holds that
\begin{equation}\label{Weighted-6}
\begin{aligned}
\|w_1w_2\|_{L^2(\Omega)}\lesssim \|w_1\|_{H^{1/2}(\Omega)}\|w_2\|_{H^{1/2}(\Omega)}.
\end{aligned}
\end{equation}
\end{lemma}

\begin{proof}
\eqref{Weighted-6} follows from 
\begin{equation*}
\begin{aligned}
\|w_1w_2\|_{L^2(\Omega)}\leq \|w_1\|_{L^4(\Omega)}\|w_2\|_{L^4(\Omega)}
\lesssim \|w_1\|_{H^{1/2}(\Omega)}\|w_2\|_{H^{1/2}(\Omega)},
\end{aligned}
\end{equation*}
where one has used  (by choosing \(s=1/2\)) 
\begin{equation*}
\begin{aligned}
\|w\|_{L^{2/(1-s)}(\Omega)}\lesssim \|w\|_{H^s(\Omega)}\quad \mbox{for}\ 0<s<1. 
\end{aligned}
\end{equation*}
\end{proof}

\begin{lemma}\label{le:W-5} It holds that
\begin{equation}\label{Weighted-7}
\begin{aligned}
\sum_{l_1=0}^5\|\partial_1^{l_1}Dv\|_{H^{1/2}}\lesssim E^{1/2}(t,v). 
\end{aligned}
\end{equation}
Consequently, if \eqref{fluid particle} and \eqref{Lagrangian variable} hold, then
\begin{equation}\label{Weighted-8}
\begin{aligned}
&\|D\eta\|_{H^{1/2}}
\lesssim  1+t\sup_{0\leq s \leq t}E^{1/2}(s,v),\\
&\sum_{l_1=1}^5\|\partial_1^{l_1}D\eta\|_{H^{1/2}}
\lesssim  t\sup_{0\leq s \leq t}E^{1/2}(s,v).
\end{aligned}
\end{equation}
\end{lemma}

\begin{proof} One uses \eqref{ineq:weighted Sobolev-1} and \eqref{ineq:weighted Sobolev-2} to estimate
\begin{equation*}
\begin{aligned}
\|\partial_1^{l_1}Dv\|_{H^{1/2}}^2
&\lesssim \|\rho_0^{1/2}\partial_1^{l_1}Dv\|_{L^2}^2+\|\rho_0^{1/2}\partial_1^{l_1}D^2v\|_{L^2}^2\\
&\lesssim E(s,v)+(\|\rho_0^{3/2}\partial_1^{l_1}D^2v\|_{L^2}^2+\|\rho_0^{3/2}\partial_1^{l_1}D^3v\|_{L^2}^2)\\
&\lesssim E(s,v),\quad l_1=0,1, \dots, 5. 
\end{aligned}
\end{equation*}
\end{proof}

Similarly, one can also show 

\begin{lemma}\label{le:W-6} It holds that
\begin{equation}\label{Weighted-9}
\begin{aligned}
\sum_{l_1=0}^3\|\partial_t\partial_1^{l_1}Dv\|_{H^{1/2}}\lesssim E^{1/2}(t,v)
\end{aligned}
\end{equation}
and
\begin{equation}\label{Weighted-9-1}
	\begin{aligned}
		\sum_{l_1=0}^1\|\partial_t^2\partial_1^{l_1}Dv\|_{H^{1/2}}\lesssim E^{1/2}(t,v)
	\end{aligned}
\end{equation}
\end{lemma}

\subsubsection{Sharp \(L^\infty\)-type inequalities}\label{sub:W-3}

\begin{lemma}\label{le:W-7} It holds that
\begin{equation}\label{Weighted-10}
\begin{aligned}
\|\partial_1^3Dv\|_{L^\infty}\lesssim E^{1/2}(t,v).
\end{aligned}
\end{equation}
Consequently, if \eqref{fluid particle} and \eqref{Lagrangian variable} hold, then
\begin{equation}\label{Weighted-11}
\begin{aligned}
\|\partial_1^3D\eta\|_{L^\infty}
\lesssim  t\sup_{0\leq s \leq t}E^{1/2}(s,v).
\end{aligned}
\end{equation}
\end{lemma}

\begin{proof} It follows from \eqref{ineq:weighted Sobolev-2} and \eqref{ineq:weighted Sobolev-3} that 
\begin{equation*}
\begin{aligned}
&\|\partial_1^3Dv\|_{L^\infty}^2
\lesssim \sum_{k=0}^2\int_\Omega \rho_0|\partial_1^3D^{k+1}v|^2\,\diff x\\
&\lesssim E(s,v)+\bigg(\int_\Omega \rho_0^3|\partial_1^3D^2v|^2\,\diff x+\int_\Omega \rho_0^3|\partial_1^3D^3v|^2\,\diff x\bigg)\\
&\quad+\bigg(\int_\Omega \rho_0^3|\partial_1^3D^3v|^2\,\diff x+\int_\Omega \rho_0^3|\partial_1^3D^4v|^2\,\diff x\bigg)\\
&\lesssim E(s,v)+\bigg(\int_\Omega \rho_0^5|\partial_1^3D^4v|^2\,\diff x+\int_\Omega \rho_0^5|\partial_1^3D^5v|^2\,\diff x\bigg)
\lesssim E(t,v).
\end{aligned}
\end{equation*}
\end{proof}

Similarly, one also has

\begin{lemma}\label{le:W-8} It holds that
\begin{equation}\label{Weighted-12} 
\begin{aligned}
\|\partial_t\partial_1Dv\|_{L^\infty}\lesssim E^{1/2}(t,v). 
\end{aligned}
\end{equation}
\end{lemma}

\eqref{Weighted-10}-\eqref{Weighted-12} refine the corresponding ones in Corollary \ref{le:W-2} and \eqref{Weighted-5}, respectively, in the sense of gaining one tangential derivative \(\partial_1\). 

\eqref{Weighted-6}-\eqref{Weighted-9-1} and \eqref{Weighted-10}-\eqref{Weighted-12} will play a crucial role in the a priori estimates in Sections \ref{Energy Estimates} and \ref{Elliptic Estimates}.

\subsubsection{A weighted interpolation inequality}\label{Weighted Interpolation Inequality}

The following weighted interpolation inequality will be used to show that the iterative solutions sequence obtained in Section \ref{Existence Part} forms a Cauchy sequence in  \(C([0,T]; H^s(\Omega))\).

\begin{lemma}\label{le:W-9} It holds that
	\begin{align}\label{Weighted-13}
		\|g\|_{L^2(\Omega)}\lesssim \|g\|_{L_{\rho_0}^2(\Omega)}^{1/2}\|g\|_{H_{\rho_0}^1(\Omega)}^{1/2},
	\end{align}
	where
	\begin{equation*}
		\begin{aligned}
			\|g\|_{L_{\rho_0}^2(\Omega)}^2=\int_\Omega\rho_0g^2\,\diff x\quad {\rm{and}}\quad 	\|g\|_{H_{\rho_0}^1(\Omega)}^2=\int_\Omega\rho_0(g^2+|Dg|^2)\,\diff x.
		\end{aligned}
	\end{equation*}	
\end{lemma}

\begin{proof} Notice that
\begin{equation*}
		\begin{aligned}
d(x, \partial(\mathbb{T}\times [0,1/2]))=x_2\quad \text{and}\quad  d(x, \partial(\mathbb{T}\times [1/2, 1]))=1-x_2.			
		\end{aligned}
	\end{equation*}	
Due to \eqref{eq:intro-3}, it suffices to prove \eqref{Weighted-13}  for \(\rho_0(x_1,x_2)=x_2\) on \(\mathbb{T}\times [0,1/2]\) and \(1-x_2\) on \(\mathbb{T}\times [1/2,1]\).

	Integration by parts yields 
	\begin{equation}\label{Weighted-14}
		\begin{aligned}
			\int_0^{1/2}\int_\mathbb{T}g^2\,\diff x_1\diff x_2
			=\frac{1}{2}\int_\mathbb{T}g^2(x_1,1/2)\,\diff x_1-2\int_0^{1/2}\int_\mathbb{T}\rho_0g\partial_2g\,\diff x_1\diff x_2.
		\end{aligned}
	\end{equation}
To estimate \(g(x_1,1/2)\), one has
	\begin{equation}\label{Weighted-15}
		\begin{aligned}
			&\int_0^{1/2}\int_\mathbb{T}\rho_0g^2\,\diff x_1\diff x_2=\int_0^{1/2}\int_\mathbb{T}x_2g^2\,\diff x_1\diff x_2\\
			&=\frac{1}{8}\int_\mathbb{T}g^2(x_1,1/2)\,\diff x_1-\int_0^{1/2}\int_\mathbb{T}\rho_0^2g\partial_2g\,\diff x_1\diff x_2.
		\end{aligned}
	\end{equation}
It follows from \eqref{Weighted-14} and \eqref{Weighted-15} that
{\small	\begin{equation}\label{Weighted-16}
		\begin{aligned}
			&\int_0^{1/2}\int_\mathbb{T}g^2\,\diff x_1\diff x_2\\
			&=4\int_0^{1/2}\int_\mathbb{T}\rho_0g^2\,\diff x_1\diff x_2+4\int_0^{1/2}\int_\mathbb{T}\rho_0^2g\partial_2g\,\diff x_1\diff x_2
-2\int_0^{1/2}\int_\mathbb{T}\rho_0g\partial_2g\,\diff x_1\diff x_2\\
			&\lesssim \int_0^{1/2}\int_\mathbb{T}\rho_0g^2\,\diff x_1\diff x_2
+\bigg(\int_0^{1/2}\int_\mathbb{T}\rho_0^2g^2\,\diff x_1\diff x_2\bigg)^{1/2}\bigg(\int_0^{1/2}\int_\mathbb{T}\rho_0^2|Dg|^2\,\diff x_1\diff x_2\bigg)^{1/2}\\
			&\quad+\bigg(\int_0^{1/2}\int_\mathbb{T}\rho_0g^2\,\diff x_1\diff x_2\bigg)^{1/2}\bigg(\int_0^{1/2}\int_\mathbb{T}\rho_0|Dg|^2\,\diff x_1\diff x_2\bigg)^{1/2}\\
			&\lesssim \bigg(\int_\Omega\rho_0g^2\,\diff x\bigg)^{1/2}\bigg(\int_\Omega\rho_0(g^2+|Dg|^2)\,\diff x\bigg)^{1/2}.
		\end{aligned}
	\end{equation}}

Similarly, one can obtain
	\begin{equation}\label{Weighted-17}
		\begin{aligned}
			&\int_{1/2}^1\int_\mathbb{T}g^2\,\diff x_1\diff x_2
			\lesssim \bigg(\int_\Omega\rho_0g^2\,\diff x\bigg)^{1/2}\bigg(\int_\Omega\rho_0(g^2+|Dg|^2)\,\diff x\bigg)^{1/2}.
		\end{aligned}
	\end{equation}

	Finally, \eqref{Weighted-13} follows from \eqref{Weighted-16} and \eqref{Weighted-17}.
	
\end{proof}

\section{A Priori Bounds on \(J^{-2}b^{kj}\) and \(J^{-2}a_i^k\)}\label{sec:Jab}
Let \((v,\eta)\) be a solution to \eqref{eq:main-2} satisfying \eqref{a priori assumption}. 
By \eqref{J-formula}, \eqref{a-formula} and \eqref{J-bound}, after some long but elementary calculations, 
one can show the following lemmas.
\begin{lemma}\label{le:Jab-bound-1} It holds that
{\small{\begin{equation}\label{Jab-bound-1}
\begin{aligned}
&|\partial_t^{l_0}(J^{-2}b^{kj})|+|\partial_t^{l_0}(J^{-2}a_i^k)|\lesssim Q_1(t),
\quad l_0=1,2,\\
&|\partial_t^{l_0}(J^{-2}b^{kj})|+|\partial_t^{l_0}(J^{-2}a_i^k)|\lesssim Q_1(t)(\sum_{m=2}^{l_0-1}|\partial_t^mDv|+1),\quad l_0=3,4
\end{aligned}
\end{equation}}}
with some generic polynomial function
{\small{\begin{equation}\label{Jab-bound-2}
\begin{aligned}
Q_1(t)=Q_1(|Dv|,|D\eta|, |\partial_tDv|).
\end{aligned}
\end{equation}}}

\end{lemma}

\begin{lemma}\label{le:Jab-bound-2}
It holds that
{\small{\begin{equation}\label{Jab-bound-5}
\begin{aligned}
&|\partial_1^{l_1}(J^{-2}b^{kj})|+|\partial_1^{l_1}(J^{-2}a_i^k)|\lesssim Q_2(t), \quad l_1=1,2,3,\\
&|\partial_1^{l_1}(J^{-2}b^{kj})|+|\partial_1^{l_1}(J^{-2}a_i^k)|\lesssim Q_2(t)\bigg(\sum_{m=4}^{l_1}|\partial_1^mD\eta|+1\bigg),\quad l_1=4,5,6,7
\end{aligned}
\end{equation}}}
with some generic polynomial function
{\small{\begin{equation}\label{Jab-bound-6}
\begin{aligned}
Q_2(t):=Q_2(|\partial_1^sD\eta|),\quad s=0,1,2,3.
\end{aligned}
\end{equation}}}

\end{lemma}

\begin{lemma}\label{le:Jab-bound-3}
It holds that
{\small{\begin{equation}\label{Jab-bound-3}
\begin{aligned}
&|\partial_t^2\partial_1(J^{-2}b^{kj})|+|\partial_t^2\partial_1(J^{-2}a_i^k)|\lesssim Q_3(t),\\
&|\partial_t^3\partial_1(J^{-2}b^{kj})|+|\partial_t^3\partial_1(J^{-2}a_i^k)|\lesssim Q_3(t)(|\partial_t^2\partial_1Dv|+|\partial_t^2Dv|+1),\\
&|\partial_t^2\partial_1^2(J^{-2}b^{kj})|+|\partial_t^2\partial_1^2(J^{-2}a_i^k)|\lesssim Q_3(t)(|\partial_t\partial_1^2Dv|+1),\\
&|\partial_t^3\partial_1^2(J^{-2}b^{kj})|+|\partial_t^3\partial_1^2(J^{-2}a_i^k)|\lesssim Q_3(t)
\bigg(\sum_{\alpha=0}^2|\partial_t^2\partial_1^\alpha Dv|+|\partial_t\partial_1^2Dv|+1\bigg)
\end{aligned}
\end{equation}}}
with some generic polynomial function
{\small{\begin{equation}\label{Jab-bound-4}
\begin{aligned}
Q_3(t):=Q_3(|\partial_1^{s_1}Dv|, |\partial_1^{s_2}D\eta|, |\partial_t\partial_1^{s_3}Dv|),\quad s_1, s_2=0,1,2;\ s_3=0,1.
\end{aligned}
\end{equation}}}

\end{lemma}

\begin{lemma}\label{le:Jab-bound-4}
It holds that
{\small{\begin{equation}\label{Jab-bound-7}
\begin{aligned}
&|\partial_t\partial_1^{l_1}(J^{-2}b^{kj})|+|\partial_t\partial_1^{l_1}(J^{-2}a_i^k)|\lesssim Q_4(t), \quad l_1=1,2,3,\\
\\
&|\partial_t\partial_1^{l_1}(J^{-2}b^{kj})|+|\partial_t\partial_1^{l_1}(J^{-2}a_i^k)|\lesssim Q_4(t)\bigg(\sum_{m=4}^{l_1}(|\partial_1^mDv|+|\partial_1^mD\eta|)+1\bigg)
, \\
&\qquad\qquad\qquad\qquad\qquad\qquad\qquad\qquad\qquad\quad l_1=4,5,6,7
\end{aligned}
\end{equation}}}
with some generic polynomial function
{\small{\begin{equation}\label{Jab-bound-8}
\begin{aligned}
Q_4(t):=Q_4(|\partial_1^{s_1}Dv|, |\partial_1^{s_2}D\eta|),\quad s_1, s_2=0,1,2,3. 
\end{aligned}
\end{equation}}}

\end{lemma}

\begin{lemma}\label{le:Jab-bound-5}
 It holds that
{\small{\begin{equation}\label{Jab-bound-9}
\begin{aligned}
|\partial_tD(J^{-2}b^{kj})|+|\partial_tD(J^{-2}a_i^k)|&\lesssim Q_5(t),\\
|\partial_tD^{l_2}(J^{-2}b^{kj})|+|\partial_tD^{l_2}(J^{-2}a_i^k)|&\lesssim Q_5(t)\bigg(\sum_{m=3}^{l_2+1}(|D^mv|+|D^m\eta|)+g_{l_2}(v,\eta)+1\bigg)
, \\
&\qquad\qquad\quad l_2=2,3,4,5,
\end{aligned}
\end{equation}}}
where 
{\small{\begin{equation*}
\begin{aligned}
&g_2(v,\eta)=g_3(v,\eta)=0,\quad g_4(v,\eta)=|D^3vD^3\eta|,\\
&g_5(v,\eta)=|D^4vD^3\eta|+|D^3vD^4\eta|+|D^3vD^3\eta|
\end{aligned}
\end{equation*}}}
with some generic polynomial function
{\small{\begin{equation}\label{Jab-bound-10}
\begin{aligned}
Q_5(t):=Q_5(|Dv|, |D^2v|, |D\eta|, |D^2\eta|).
\end{aligned}
\end{equation}}}

\end{lemma}

\begin{lemma}\label{le:Jab-bound-6}
 It holds that
{\small{\begin{equation}\label{Jab-bound-11}
\begin{aligned}
|D(J^{-2}b^{kj})|+|D(J^{-2}a_i^k)|&\lesssim Q_6(t),\\
|D^{l_2}(J^{-2}b^{kj})|+|D^{l_2}(J^{-2}a_i^k)|&\lesssim Q_6(t)\bigg(\sum_{m=3}^{l_2+1}|D^m\eta|+g_{l_2}(\eta)+1\bigg),\\
&\qquad\qquad\quad l_2=2,3,4,5,6,7,
\end{aligned}
\end{equation}}}
where
{\small{\begin{equation*}
\begin{aligned}
&g_2(\eta)=g_3(\eta)=0,\quad g_4(\eta)=|D^3\eta|^2,\\
&g_5(\eta)=|D^4\eta D^3\eta|+|D^3\eta|^2,\\
&g_6(\eta)=|D^5\eta D^3\eta|+|D^4\eta|^2+|D^3\eta|^2,\\
&g_7(\eta)=|D^6\eta D^3\eta|+|D^5\eta D^4\eta|+|D^5\eta D^3\eta|\\
&\qquad\quad\ +|D^4\eta|^2+|D^4\eta||D^3\eta|^2+|D^3\eta|^3+|D^3\eta|^2
\end{aligned}
\end{equation*}}}
with some generic polynomial function
{\small{\begin{equation}\label{Jab-bound-12}
\begin{aligned}
Q_6(t):=Q_6(|D\eta|, |D^2\eta|).
\end{aligned}
\end{equation}}}

\end{lemma}

The a priori bounds on \(J^{-2}b^{kj}\) and \(J^{-2}a_i^k\) in Lemmas \ref{le:Jab-bound-1}-\ref{le:Jab-bound-6} together with the use of the \(L^\infty\)-estimates in  Subsections \ref{sub:W-1} and \ref{sub:W-3}
will be useful for the estimates in Sections \ref{Energy Estimates} and \ref{Elliptic Estimates}.

\section{Energy Estimates}\label{Energy Estimates}

\subsection{The strategy of the energy estimates} 
To derive the basic energy estimates controlling the tangential derivatives,
we will start with the weighted estimates on
\(\partial_t^4v\), after this,  the remaining energy estimates are divided into two types: the type-I energy estimates are    
obtained by first estimating \(\partial_t^3Dv\), then \(\partial_t^2\partial_1^2Dv\), and so on, until \(\partial_1^6Dv\); the type-II energy estimates are obtained by first estimating \(\partial_t^3\partial_1Dv\), then \(\partial_t^2\partial_1^3Dv\), and so on, until \(\partial_1^7Dv\). 
Let \((v,\eta)\) be a solution to \eqref{eq:main-2} satisfying \eqref{a priori assumption}. 

Throughout this section, we denote by
\begin{equation*}
\begin{aligned}
P_{*}(t):=P_{*}(1+t\sup_{0\leq \tau\leq t}E^{1/2}(\tau,v))
\end{aligned}
\end{equation*}
and
\begin{equation*}
\begin{aligned}
P_{**}(t):=P_{**}(\sup_{0\leq \tau\leq t}E^{1/2}(\tau,v), 1+t\sup_{0\leq \tau\leq t}E^{1/2}(\tau,v))
\end{aligned}
\end{equation*}
some generic polynomial functions of their arguments, which may be different from line to line. We also frequently use ``the desired bound'' to denote \(M_0+Ct P(\sup_{0\leq s\leq t}E^{1/2}(s,v))\) for convenience.

\subsection{Estimates on \(\partial_t^{l_0}v\ (l_0=0,1,2,3,4)\)}\label{ENP-1} 
\begin{proposition}\label{pr:ene-1} It holds that
\begin{equation}\label{EEP-1}
\begin{aligned}
\sum_{l_0=0}^4\int_\Omega \rho_0|\partial_t^{l_0}v|^2(t)\,\diff x+\int_0^t\int_\Omega\rho_0|\partial_t^4Dv|^2\,\diff x\diff s
\leq M_0+Ct P(\sup_{0\leq s\leq t}E^{1/2}(s,v)).
\end{aligned}
\end{equation}
\end{proposition}
\begin{proof} Note that
\begin{equation*}
\begin{aligned}
g(x,t)=g(x,0)+\int_0^t\partial_tg(x,s)\,\diff s.
\end{aligned}
\end{equation*}
It then follows from Cauchy's inequality and Fubini's theorem that
\begin{equation}\label{FTC}
\begin{aligned}
\int_\Omega |g|^2(t)\,\diff x
&\lesssim \int_\Omega |g|^2(0)\,\diff x
+t\int_0^t\int_\Omega |\partial_tg|^2\,\diff x\diff s.
\end{aligned}
\end{equation}

By \eqref{HOEF-0} and \eqref{FTC}, it suffices to show that the highest-order terms on the left-hand side (abbreviated as LHS) of \eqref{EEP-1} satisfy the desired bound. 

To this end, one  applies \(\partial_t^4\) to \eqref{eq:main-3}, by taking the \(L^2\)-inner product with \(\partial_t^4v^i\) and integrating by parts with respect to spatial variables (the boundary terms vanish due to \eqref{eq:intro-3}), to obtain that
\begin{equation}\label{EE-1}
\begin{aligned}
&\frac{1}{2}\int_\Omega \rho_0|\partial_t^4v|^2(t)\,\diff x
+\underline{\int_0^t\int_\Omega \rho_0J^{-2}b^{kj}\partial_t^4v^i,_{j}\partial_t^4v^i,_{k}\,\diff x\diff s}_{:=I_1}\\
&=\frac{1}{2}\int_\Omega \rho_0|\partial_t^4v|^2(0)\,\diff x+\underline{\int_0^t\int_\Omega\rho_0^2\partial_t^4(J^{-2}a_i^k)\partial_t^4v^i,_{k}\,\diff x\diff s}_{:=I_2}\\
&\quad-\sum_{m=1}^4\binom{4}{m}\underline{\int_0^t\int_\Omega\rho_0\partial_t^m(J^{-2}b^{kj})\partial_t^{4-m}v^i,_{j}\partial_t^4v^i,_{k}\,\diff x\diff s.}_{:=I_{3_m}}
\end{aligned}
\end{equation}

First, by \eqref{J-bound} and \eqref{a-bound}, one finds that
\begin{equation*}
	\begin{aligned}
		I_1\geq \frac{2}{11}\int_0^t\int_\Omega\rho_0|\partial_t^4Dv|^2\,\diff x\diff s.
	\end{aligned}
\end{equation*}
Now we estimate each term on the RHS of \eqref{EE-1}. It is easy to find that the first term on the RHS is contained in \(M_0\). 

To handle \(I_2\) and \(I_{3_m}\),  we note that
\begin{equation}\label{Q-bound-1}
\begin{aligned}
Q_1(t)\lesssim P_{**}(t),
\end{aligned}
\end{equation}
which follows from \eqref{Weighted-3}-\eqref{Weighted-5} and \eqref{Jab-bound-2}.  
Moreover, it is easy to see that
\begin{equation}\label{Q-bound-2}
\begin{aligned}
\int_0^tP_{**}(s)\,\diff s\leq Ct P(\sup_{0\leq s\leq t}E^{1/2}(s,v)).
\end{aligned}
\end{equation}

For convenience, we will use \(\mathcal{R}\) to denote contributions of space (space-time) integrals of lower-order terms (abbreviated as l.o.t.), which are easily shown to satisfy the desired bound.

By \eqref{Jab-bound-1}, \eqref{Q-bound-1} and \eqref{Q-bound-2}, one may estimate 
%\footnote{We use the \(H^{1/2}\)-type inequalities \(\|\partial_t^2Dv\|_{H^{1/2}}^2\|\partial_tDv\|_{H^{1/2}}^2\) to estimate \(I_{3_3}\) for paying the least price to modify the estimates in Section \ref{Regularity}, although here one can use a simpler estimate \(\|\sqrt{\rho_0}\partial_t^2Dv\|_{L^2}^2\|\partial_tDv\|_{L^\infty}^2\).}
\begin{equation}\label{EE-2}
\begin{aligned}
|I_2|
&\leq C\int_0^tP_{**}^2\|\sqrt{\rho_0}\partial_t^3Dv\|_{L^2}^2\,\diff s
+G,\\
|I_{3_m}|
&\leq 
C\int_0^tP_{**}^2\|\sqrt{\rho_0}\partial_t^{4-m}Dv\|_{L^2}^2\, \diff s+G,
\quad m=1,2,\\
|I_{3_3}|
&\leq C\int_0^tP_{**}^2\|\partial_t^2Dv\|_{H^{1/2}}^2
\|\partial_tDv\|_{H^{1/2}}^2\,\diff s 
+G,\\
|I_{3_4}|
&\leq C\int_0^tP_{**}^2\|\sqrt{\rho_0}\partial_t^3Dv\|_{L^2}^2\|Dv\|_{L^\infty}^2\,\diff s
+G
\end{aligned}
\end{equation}
with \(G=\frac{1}{100}\int_0^t\int_\Omega\rho_0|\partial_t^4Dv|^2\,\diff x\diff s+\mathcal{R}\), which can be bounded by
\begin{equation*}
	\begin{aligned} 
	G+M_0+Ct P(\sup_{0\leq s\leq t}E^{1/2}(s,v)),
\end{aligned}
\end{equation*}
where \eqref{Weighted-9} and \eqref{Weighted-9-1} have been used. 

Taking all the cases above into account proves
the case \(l_0=4\) in \eqref{EEP-1}.

\end{proof}

\begin{remark}
To unify the estimates for both \eqref{eq:main-2} and the linearized problem \eqref{existence-3}, we have used  \(\|\partial_t^2Dv\|_{H^{1/2}}^2\|\partial_tDv\|_{H^{1/2}}^2\) to estimate \(I_{3_3}\) in \eqref{EE-2}, although a simpler estimate \(\|\sqrt{\rho_0}\partial_t^2Dv\|_{L^2}^2\|\partial_tDv\|_{L^\infty}^2\) works here. Many other similar manipulations will be used in this and next sections.
\end{remark}

\subsection{The type-I energy estimates} \label{ENP-2} 

\subsubsection{Estimates on \(\partial_t^{l_0}Dv\ (l_0=0,1,2,3)\)} \label{fte-1}
\begin{proposition}\label{pr:ene-2} It holds that
\begin{equation}\label{EEP-2}
\begin{aligned}
\sum_{l_0=0}^3\int_\Omega\rho_0|\partial_t^{l_0}Dv|^2(t)\,\diff x
\leq M_0+Ct P(\sup_{0\leq s\leq t}E^{1/2}(s,v)).
\end{aligned}
\end{equation}
\end{proposition}

\begin{proof} By \eqref{FTC},  it suffices to show the case \(l_0=3\) in \eqref{EEP-2}. 

Applying \(\partial_t^3\) to  \eqref{eq:main-3}, taking the \(L^2\)-inner product with \(\partial_t^4v^i\), and integrating by parts with respect to spatial variables, one gets that
\begin{equation*}
	\begin{aligned}
	&\int_0^t\int_\Omega \rho_0(\partial_t^4v)^2\,\diff x\diff s
	+\frac{1}{2}\int_\Omega \rho_0J^{-2}b^{kj}\partial_t^3v^i,_{j}\partial_t^3v^i,_{k}(t)\,\diff x \\
	&=\frac{1}{2}\int_\Omega \rho_0J^{-2}b^{kj}\partial_t^3v^i,_{j}\partial_t^3v^i,_{k}(0)\,\diff x+\underline{\int_0^t\int_\Omega \rho_0^2\partial_t^3(J^{-2}a_i^k)\partial_t^4v^i,_{k}\,\diff x\diff s.}_{:=I_1}\\
&\quad+\frac{1}{2}\underline{\int_0^t\int_\Omega\rho_0\partial_t(J^{-2}b^{kj})\partial_t^3v^i,_{j}\partial_t^3v^i,_{k}\,\diff x\diff s}_{:=I_2}\\
&\quad-\sum_{m=1}^3\binom{3}{m}\underline{\int_0^t\int_\Omega \rho_0\partial_t^m(J^{-2}b^{kj})\partial_t^{3-m}v^i,_{j}\partial_t^4v^i,_{k}\,\diff x\diff s}_{:=I_{3_m}}
	\end{aligned}
	\end{equation*}

First,  one uses \eqref{Jab-bound-1}, \eqref{Q-bound-1} and \eqref{Q-bound-2} to estimate 
\begin{equation*}
\begin{aligned}
|I_2|
\lesssim \int_0^tP_{**}^2\|\sqrt{\rho_0}\partial_t^3Dv\|_{L^2}^2\,\diff s
\leq M_0+Ct P(\sup_{0\leq s\leq t}E^{1/2}(s,v)).
\end{aligned}
\end{equation*}
Next, it holds that
\begin{equation*}
\begin{aligned}
|I_{3_m}|
&\lesssim \int_0^t\int_\Omega\rho_0|\partial_t^4Dv|^2\,\diff x\diff s+\int_0^tP_{**}^2\|\sqrt{\rho_0}\partial_t^{3-m}Dv\|_{L^2}^2\,\diff s,
\quad m=1,2,\\
|I_{3_3}|
&\lesssim \int_0^t\int_\Omega\rho_0|\partial_t^4Dv|^2\,\diff x\diff s
+\int_0^tP_{**}^2\|\sqrt{\rho_0}\partial_t^2Dv\|_{L^2}^2\|Dv\|_{L^\infty}^2\,\diff s+\mathcal{R},\\
\end{aligned}
\end{equation*}
which, together with \eqref{EEP-1}, implies that \(I_{3_m}\) satisfies 
the desired bound. 

\(I_1\) can be handed similarly as \(I_{3_3}\) (in fact, \(I_1\) is much easier since it contains a higher-order weight \(\rho_0^2\)). 

Thus, the proof of
the case \(l_0=3\) in \eqref{EEP-2} is completed.

\end{proof}

\subsubsection{Estimates on \(\partial_t^{l_0}\partial_1^2Dv \ (l_0=0,1,2)\) and  \(\partial_t^{l_0}\partial_1^4Dv \ (l_0=0,1)\)}
Following a similar argument for \(\partial_t^3Dv\) in Proposition \ref{pr:ene-2}, 
one can show the following estimates. 
\begin{proposition}\label{pr:ene-3} It holds that
\begin{equation}\label{EEP-add-1}
\begin{aligned}
&\sum_{l_0=0}^2\int_\Omega\rho_0|\partial_t^{l_0}\partial_1^2Dv|^2(t)\,\diff x
+\sum_{l_0=1}^1\int_\Omega \rho_0|\partial_t^{l_0}\partial_1^4Dv|^2(t)\,\diff x\\
&\leq M_0+Ct P(\sup_{0\leq s\leq t}E^{1/2}(s,v)).
\end{aligned}
\end{equation}
\end{proposition}

\subsubsection{Estimates on \(\partial_1^6Dv\)}
\begin{proposition}\label{pr:ene-4} It holds that
\begin{equation}\label{EEP-3}
\begin{aligned}
\int_\Omega\rho_0|\partial_1^6Dv|^2(t)\,\diff x
\leq M_0+Ct P(\sup_{0\leq s\leq t}E^{1/2}(s,v)).
\end{aligned}
\end{equation}
\end{proposition}

\begin{proof}

Applying \(\partial_1^6\) to \eqref{eq:main-3}, and taking the \(L^2\)-inner product with \(\partial_t\partial_1^6v^i\), one can obtain by some manipulations that
{\small\begin{equation*}
	\begin{aligned}
	&\int_0^t\int_\Omega \rho_0|\partial_t\partial_1^6v|^2\,\diff x\diff s+\frac{1}{2}\int_\Omega \rho_0J^{-2}b^{kj}\partial_1^6v^i,_{j}\partial_1^6v^i,_{k}(t)\,\diff x\\
&=\frac{1}{2}\int_\Omega \rho_0J^{-2}b^{kj}\partial_1^6v^i,_{j}\partial_1^6v^i,_{k}(0)\,\diff x-\int_0^t\int_\Omega \partial_1^7(\rho_0^2J^{-2}a_i^k)\partial_t\partial_1^5v^i,_{k}\,\diff x\diff s.\\
&\quad-\underline{\sum_{m=1}^6\binom{6}{m}\int_0^t\int_\Omega \partial_1^m(\rho_0)\partial_t\partial_1^{6-m}v^i\partial_t\partial_1^6v^i\,\diff x\diff s}_{:=I_1}\\
&\quad+\frac{1}{2}\underline{\int_0^t\int_\Omega \rho_0\partial_t(J^{-2}b^{kj})\partial_1^6v^i,_{j}\partial_1^6v^i,_{k}\,\diff x\diff s}_{:=I_2}\\
&\quad+\sum_{m=1}^6\binom{6}{m}\underline{\int_0^t\int_\Omega \partial_1^{m+1}(\rho_0J^{-2}b^{kj})\partial_1^{6-m}v^i,_{j}\partial_t\partial_1^5v^i,_{k}\,\diff x\diff s}_{:=I_{3_m}}\\
&\quad+\sum_{m=1}^6\binom{6}{m}\underline{\int_0^t\int_\Omega \partial_1^m(\rho_0J^{-2}b^{kj})\partial_1^{7-m}v^i,_{j}\partial_t\partial_1^5v^i,_{k}\,\diff x\diff s.}_{:=I_{4_m}}
	\end{aligned}
	\end{equation*}}
One needs only to take care of the underlined terms.

Recalling \eqref{TE}, and using Cauchy's inequality, one may deduce that
\begin{equation*}
	\begin{aligned}
|I_1|&\lesssim \sum_{m=1}^6\bigg\|\frac{\partial_1^m(\rho_0)}{\rho_0}\bigg\|_{L^\infty}\int_0^t\int_\Omega \rho_0|\partial_t\partial_1^{6-m}v|^2\,\diff x\diff s
+\int_0^t\int_\Omega \rho_0|\partial_t\partial_1^6v|^2\,\diff x\diff s\\
&\lesssim \sum_{m=0}^6\int_0^t\|\sqrt{\rho_0}\partial_t\partial_1^mv\|_{L^2}^2\,\diff s
\leq Ct P(\sup_{0\leq s\leq t}E^{1/2}(s,v)).
	\end{aligned}
	\end{equation*}

With \eqref{Jab-bound-1}, \eqref{Q-bound-1} and \eqref{Q-bound-2}, one can  obtain easily that
\begin{equation*}
	\begin{aligned}
|I_2|&\lesssim \int_0^tP_{**}^2\|\sqrt{\rho_0}\partial_1^6Dv\|_{L^2}^2\,\diff s\leq Ct P(\sup_{0\leq s\leq t}E^{1/2}(s,v)).
	\end{aligned}
	\end{equation*}

In order to estimate \(I_{3_m}\) and \(I_{4_m}\), one observes that
\begin{equation}\label{Q-bound-3}
\begin{aligned}
Q_2(t)\lesssim P_{*}(t),
\end{aligned}
\end{equation}
which can be verified by \eqref{Weighted-4}, \eqref{Weighted-11} and \eqref{Jab-bound-6}. 
Moreover, it holds that 
\begin{equation}\label{Q-bound-4}
\begin{aligned}
\int_0^tP_{*}(s)\,\diff s\leq Ct P(\sup_{0\leq s\leq t}E^{1/2}(s,v))
\end{aligned}
\end{equation}
and
\begin{equation}\label{Q-bound-5}
\begin{aligned}
P_{*}(t)\leq 1+Ct P(\sup_{0\leq s\leq t}E^{1/2}(s,v)).
\end{aligned}
\end{equation}

Recall \eqref{TE} and note that 
\begin{equation*}
\begin{aligned}
\partial_1^{m+1}(\rho_0J^{-2}b^{kj})=\sum_{k=0}^{m+1}\binom{m+1}{k}\partial_1^k(\rho_0)\partial_1^{m+1-k}(J^{-2}b^{kj}).
\end{aligned}
\end{equation*}
Then one can write
\begin{equation}\label{lot}
\begin{aligned}
\partial_1^{m+1}(\rho_0J^{-2}b^{kj})=\underbrace{\rho_0\partial_1^{m+1}(J^{-2}b^{kj})}_{\text{the\ leading-order\ term}}+\text{l.o.t.}.
\end{aligned}
\end{equation}

Now, with \eqref{Jab-bound-5} and \eqref{Q-bound-3}-\eqref{lot}, we 
can handle \(I_{3_m}\).  Indeed, one has
\begin{equation}\label{EE-30}
\begin{aligned}
|I_{3_m}|
&\lesssim \int_0^tP_{*}^2\|\sqrt{\rho_0}\partial_1^{6-m}Dv\|_{L^2}^2\,\diff s+G, \quad m=1,2,\\
|I_{3_m}|
&\lesssim \int_0^tP_{*}^2
\|\partial_1^{m+1}D\eta\|_{H^{1/2}}^2
\|\partial_1^{6-m}Dv\|_{H^{1/2}}^2\,\diff s
+G,\quad m=3,4,\\
|I_{3_m}|
&\lesssim \int_0^tP_{*}^2\|\sqrt{\rho_0}\partial_1^{m+1}D\eta\|_{L^2}^2\|\partial_1^{6-m}Dv\|_{L^\infty}^2\,\diff s
+G, \quad m=5,6
\end{aligned}
\end{equation}
with \(G=\int_0^t\|\sqrt{\rho_0}\partial_t\partial_1^5Dv\|_{L^2}^2\,\diff s+\mathcal{R}\), 
which give the desired bound.
Here one has used \eqref{Weighted-7}. 

\(I_{4_m}\) can be estimated similarly as \(I_{3_m}\).

Collecting all the estimates above completes the proof of \eqref{EEP-3}.

\end{proof}

\subsection{The type-II energy estimates} \label{ENP-3}

\subsubsection{Estimates on \(\partial_t^{l_0}\partial_1Dv \ (l_0=0,1,2,3)\)}
\begin{proposition}\label{pr:ene-5} It holds that
\begin{equation}\label{EEP-5}
\begin{aligned}
\sum_{l_0=0}^3\int_\Omega\rho_0|\partial_t^{l_0}\partial_1Dv|^2(t)\,\diff x
\leq M_0+Ct P(\sup_{0\leq s\leq t}E^{1/2}(s,v)).
\end{aligned}
\end{equation}
\end{proposition}

\begin{proof} By \eqref{FTC}, it remains to show the case \(l_0=3\) in \eqref{EEP-5}.

Applying \(\partial_t^3\partial_1\) to  \eqref{eq:main-3},
 taking the \(L^2\)-inner product with \(\partial_t^4\partial_1v^i\), and integrating by parts,
one can get as for \(\partial_t^3Dv\) that
{\small\begin{equation*}
	\begin{aligned}
	&\int_0^t\int_\Omega \rho_0|\partial_t^4\partial_1v|^2\,\diff x\diff s+\frac{1}{2}\int_\Omega \rho_0J^{-2}b^{kj}\partial_t^3\partial_1v^i,_{j}\partial_t^3\partial_1v^i,_{k}(t)\,\diff x\\
&=\frac{1}{2}\int_\Omega \rho_0J^{-2}b^{kj}\partial_t^3\partial_1v^i,_{j}\partial_t^3\partial_1v^i,_{k}(0)\,\diff x-\int_0^t\int_\Omega \partial_1^2[\rho_0^2\partial_t^3(J^{-2}a_i^k)]\partial_t^4v^i,_{k}\,\diff x\diff s\\
&\quad-\int_0^t\int_\Omega \partial_1(\rho_0)\partial_t^4v^i\partial_t^4\partial_1v^i\,\diff x\diff s
+\frac{1}{2}\int_0^t\int_\Omega \rho_0\partial_t(J^{-2}b^{kj})\partial_t^3\partial_1v^i,_{j}\partial_t^3\partial_1v^i,_{k}\,\diff x\diff s\\
&\quad+\underline{\int_0^t\int_\Omega \partial_1[\partial_1(\rho_0J^{-2}b^{kj})\partial_t^3v^i,_{j}]\partial_t^4v^i,_{k}\,\diff x\diff s}_{:=I_1}\\
&\quad+\sum_{m=1}^3\binom{3}{m}\underline{\int_0^t\int_\Omega \partial_1^2[\rho_0\partial_t^m(J^{-2}b^{kj})\partial_t^{3-m}v^i,_{j}]\partial_t^4v^i,_{k}\,\diff x\diff s.}_{:=I_{2_m}}
	\end{aligned}
	\end{equation*}}

For \(I_1\), by \eqref{Jab-bound-5}, \eqref{EEP-1} and \eqref{Q-bound-3}-\eqref{lot},  one may estimate 
\begin{equation}\label{spacetime-estimate}
	\begin{aligned}
|I_1|&\lesssim \bigg|\int_0^t\int_\Omega \rho_0[\partial_1(J^{-2}b^{kj})\partial_t^3\partial_1v^i,_{j}+\partial_1^2(J^{-2}b^{kj})\partial_t^3v^i,_{j}]\partial_t^4v^i,_{k}\,\diff x\diff s\bigg|
+\mathcal{R}\\
&\lesssim \int_0^tP_{*}^2\|\sqrt{\rho_0}\partial_t^3\partial_1Dv\|_{L^2}^2\,\diff s
+\int_0^tP_{*}^2\|\sqrt{\rho_0}\partial_t^3Dv\|_{L^2}^2\,\diff s\\
&\quad+\int_0^t\int_\Omega\rho_0|\partial_t^4Dv|^2\,\diff x\diff s+\mathcal{R}
\leq M_0+Ct P(\sup_{0\leq s\leq t}E^{1/2}(s,v)).
	\end{aligned}
	\end{equation} 
Considering \(I_{2_m}\), it holds that
\begin{equation*}
	\begin{aligned}
|I_{2_m}|	&\lesssim \bigg|\int_0^t\int_\Omega \rho_0\partial_1^2[\partial_t^m(J^{-2}b^{kj})\partial_t^{3-m}v^i,_{j}]\partial_t^4v^i,_{k}\,\diff x\diff s\bigg|+\mathcal{R}\\
&\lesssim \underline{\int_0^t\int_\Omega\rho_0\big|\partial_t^m(J^{-2}b^{kj})\partial_t^{3-m}\partial_1^2v^i,_{j}\big|^2\,\diff x\diff s}_{:=J_{1_m}}\\
&\quad+\underline{\int_0^t\int_\Omega\rho_0\big|\partial_t^m\partial_1(J^{-2}b^{kj})\partial_t^{3-m}\partial_1v^i,_{j}\big|^2\,\diff x\diff s}_{:=J_{2_m}}\\
&\quad+\underline{\int_0^t\int_\Omega\rho_0\big|\partial_t^m\partial_1^2(J^{-2}b^{kj})\partial_t^{3-m}v^i,_{j}\big|^2\,\diff x\diff s}_{:=J_{3_m}}\\
&\quad+\int_0^t\int_\Omega\rho_0|\partial_t^4Dv|^2\,\diff x\diff s+\mathcal{R}. 
	\end{aligned}
	\end{equation*}

First, one has
\begin{equation*}
	\begin{aligned}
&|J_{1_m}|\lesssim \int_0^tP_{**}^2\|\sqrt{\rho_0}\partial_t^{3-m}\partial_1^2Dv\|_{L^2}^2\,\diff s,\quad m=1,2,\\
&|J_{1_3}|\lesssim \int_0^tP_{**}^2\|\partial_t^2Dv\|_{H^{1/2}}^2\|\partial_1^2Dv\|_{H^{1/2}}^2\,\diff s+\mathcal{R},
	\end{aligned}
	\end{equation*}
which, together with \eqref{Q-bound-1} and \eqref{Q-bound-2}, implies that \(J_{1_m}\) has the desired bound.

Next, to handle \(J_{2_m}\) and \(J_{3_m}\), one notes that 
\begin{equation}\label{Q-bound-6}
\begin{aligned}
Q_3(t)\lesssim P_{**}(t),
\end{aligned}
\end{equation}
which follows from \eqref{Weighted-3}-\eqref{Weighted-5}, \eqref{Weighted-12} and \eqref{Jab-bound-4}. 
Moreover, it holds that 
\begin{equation}\label{Q-bound-7}
\begin{aligned}
\int_0^tP_{**}(s)\,\diff s\leq Ct P(\sup_{0\leq s\leq t}E^{1/2}(s,v)).
\end{aligned}
\end{equation}

By \eqref{Q-bound-6} and \eqref{Q-bound-7}, one can estimate
\begin{equation*}
	\begin{aligned}
|J_{2_1}|+|J_{2_3}|
&\lesssim \int_0^t P_{**}^2\|\sqrt{\rho_0}\partial_t^2\partial_1Dv\|_{L^2}^2\|\partial_1Dv\|_{L^\infty}^2\,\diff s+\mathcal{R},\\
|J_{2_2}|
&\lesssim \int_0^t P_{**}^2\|\sqrt{\rho_0}\partial_t\partial_1Dv\|_{L^2}^2\,\diff s+\mathcal{R},\\
|J_{3_1}|
&\lesssim \int_0^t P_{**}^2\|\sqrt{\rho_0}\partial_t^2Dv\|_{L^2}^2\,\diff s+\mathcal{R},\\
|J_{3_2}|
&\lesssim \int_0^t P_{**}^2\|\partial_t\partial_1^2Dv\|_{H^{1/2}}^2\|\partial_tDv\|_{H^{1/2}}^2\,\diff s
+\mathcal{R},\\
|J_{3_3}|
&\lesssim \int_0^t P_{**}^2\|\sqrt{\rho_0}\partial_t^2\partial_1^2Dv\|_{L^2}^2\|Dv\|_{L^\infty}^2\,\diff s
+\mathcal{R},
	\end{aligned}
	\end{equation*}
which and \eqref{Weighted-9} imply that \(J_{2_m}\) and \(J_{3_m}\) satisfy the desired bound.

\end{proof}

\subsubsection{Estimates on \(\partial_t^{l_0}\partial_1^3Dv \ (l_0=0,1,2)\) and  \(\partial_t^{l_0}\partial_1^5Dv \ (l_0=0,1)\)} 
Following a similar argument as for \(\partial_t^3\partial_1Dv\) in Proposition \ref{pr:ene-5}, one can show
\begin{proposition}\label{pr:ene-6} It holds that
\begin{equation}\label{EEP-add-2}
\begin{aligned}
&\sum_{l_0=0}^2\int_\Omega\rho_0|\partial_t^{l_0}\partial_1^3Dv|^2(t)\,\diff x
+\sum_{l_0=1}^1\int_\Omega \rho_0|\partial_t^{l_0}\partial_1^5Dv|^2(t)\,\diff x\\
&\leq M_0+Ct P(\sup_{0\leq s\leq t}E^{1/2}(s,v)).
\end{aligned}
\end{equation}
\end{proposition}

\bigskip
\subsubsection{Estimates on \(\partial_1^7Dv\)}

To estimate \(\partial_1^7Dv\), one first needs to bound
\(\|\sqrt{\rho_0}\partial_t\partial_1^6Dv\|_{L_t^2L^2}\). 
\begin{proposition}\label{pr:ene-add} It holds that
	\begin{equation}\label{EEP-6}
		\begin{aligned}
			\int_\Omega \rho_0|\partial_t\partial_1^6v|^2(t)\,\diff x
			+\int_0^t\int_\Omega\rho_0|\partial_t\partial_1^6Dv|^2\,\diff x\diff s
			\leq M_0+Ct P(\sup_{0\leq s\leq t}E^{1/2}(s,v)).
		\end{aligned}
	\end{equation}
\end{proposition}
\begin{proof}
	Since \(\|\sqrt{\rho_0}\partial_t\partial_1^5Dv\|_{L^2}\) is contained in \(E(t,v)\), then one can take \(\|\sqrt{\rho_0}\partial_t\partial_1^6v(0)\|_{L^2(\Omega)}\) as the initial data to show \eqref{EEP-6} by following the argument for \(\partial_t^4v\) in Proposition \ref{pr:ene-1}. 
\end{proof}

\begin{proposition}\label{pr:ene-7} It holds that
\begin{equation}\label{EEP-7}
\begin{aligned}
\int_\Omega\rho_0|\partial_1^7Dv|^2(t)\,\diff x
\leq M_0+Ct P(\sup_{0\leq s\leq t}E^{1/2}(s,v)).
\end{aligned}
\end{equation}
\end{proposition}

\begin{proof}
Letting \(\partial_1^7\) act on \eqref{eq:main-3}, and taking the \(L^2\)-inner product with \(\partial_t\partial_1^7v^i\), one gets by some direct calculations that
 {\small \begin{equation*}\label{EE-55}
	\begin{aligned}
	&\int_0^t\int_\Omega \rho_0|\partial_t\partial_1^7v|^2\,\diff x\diff s+\frac{1}{2}\int_\Omega \rho_0J^{-2}b^{kj}\partial_1^7v^i,_{j}\partial_1^7v^i,_{k}(t)\,\diff x\\
&=\frac{1}{2}\int_\Omega \rho_0J^{-2}b^{kj}\partial_1^7v^i,_{j}\partial_1^7v^i,_{k}(0)\,\diff x-\int_0^t\int_\Omega \partial_t\partial_1^7(\rho_0^2J^{-2}a_i^k)\partial_1^7v^i,_{k}\,\diff x\diff s\\
&\quad+\int_\Omega \partial_1^7(\rho_0^2J^{-2}a_i^k)\partial_1^7v^i,_{k}(s)\,\diff x\big|_{s=0}^{s=t}\\
&\quad-\sum_{m=1}^7\binom{7}{m}\underline{\int_\Omega \partial_1^m(\rho_0J^{-2}b^{kj})\partial_1^{7-m}v^i,_{j}\partial_1^7v^i,_{k}(s)\,\diff x\big|_{s=0}^{s=t}}_{:=I_{1_m}}\\
&\quad-\sum_{m=1}^7\binom{7}{m}\int_0^t\int_\Omega \partial_1^m(\rho_0)\partial_t\partial_1^{7-m}v^i\partial_t\partial_1^7v^i\,\diff x\diff s\\
&\quad+\frac{1}{2}\int_0^t\int_\Omega \rho_0\partial_t(J^{-2}b^{kj})\partial_1^7v^i,_{j}\partial_1^7v^i,_{k}\,\diff x\diff s\\
&\quad+\sum_{m=1}^7\binom{7}{m}\underline{\int_0^t\int_\Omega \partial_1^m(\rho_0J^{-2}b^{kj})\partial_t\partial_1^{7-m}v^i,_{j}\partial_1^7v^i,_{k}\,\diff x\diff s}_{:=I_{2_m}}\\
&\quad+\sum_{m=1}^7\binom{7}{m}\underline{\int_0^t\int_\Omega \partial_t\partial_1^m(\rho_0J^{-2}b^{kj})\partial_1^{7-m}v^i,_{j}\partial_1^7v^i,_{k}\,\diff x\diff s.}_{:=I_{3_m}}
	\end{aligned}
	\end{equation*}}

\noindent As can be checked easily, it suffices to estimate the underlined terms.

For \(I_{1_m}\), it suffices to consider the case \(s=t\), and we will not write out \(s=t\) for convenience. One first uses Cauchy's inequality to get
     \begin{equation*}
	\begin{aligned}
|I_{1_m}|
\leq \frac{1}{100}\int_\Omega \rho_0|\partial_1^7Dv|^2\,\diff x+\mathcal{R}+\underline{\int_\Omega \rho_0|\partial_1^m(J^{-2}b^{kj})\partial_1^{7-m}v^i,_{j}|^2\,\diff x,}_{:=J_{1_m}}
	\end{aligned}
	\end{equation*}
and then estimates 
\begin{equation*}
\begin{aligned}
|J_{1_m}|
&\lesssim
P_{*}^2
\|\sqrt{\rho_0}\partial_1^{7-m}Dv\|_{L^2}^2
+\mathcal{R},\quad m=1,2,3,\\
|J_{1_m}|
&\lesssim P_{*}^2\|\partial_1^mD\eta\|_{H^{1/2}}^2\|\partial_1^{7-m}Dv\|_{H^{1/2}}^2+\mathcal{R},\quad m=4,5,\\
|J_{1_m}|
&\lesssim P_{*}^2\|\sqrt{\rho_0}\partial_1^mD\eta\|_{L^2}^2\|\partial_1^{7-m}Dv\|_{L^\infty}^2+\mathcal{R},\quad m=6,7,
\end{aligned}
\end{equation*}
which satisfy the desired bound.

Next, by \eqref{Jab-bound-5} and \eqref{Q-bound-3}-\eqref{lot}, one may estimate 
\begin{equation}\label{EE-60}
\begin{aligned}
|I_{2_1}|
&\lesssim \sup_{0\leq s\leq t}P_{*}^2(s)\int_0^t\int_\Omega\rho_0|\partial_t\partial_1^6Dv|^2\,\diff x\diff s+G,\\
|I_{2_m}|
&\lesssim \int_0^t P_{*}^2\|\sqrt{\rho_0}\partial_t\partial_1^{7-m}Dv\|_{L^2}^2\,\diff s+G,\quad m=2,3,\\
|I_{2_m}|
&\lesssim \int_0^t P_{*}^2\|\partial_1^mD\eta\|_{H^{1/2}}^2\|\partial_t\partial_1^{7-m}Dv\|_{H^{1/2}}^2\,\diff s
+G, \quad m=4,5,\\
|I_{2_m}|
&\lesssim \int_0^t P_{*}^2\|\sqrt{\rho_0}\partial_1^mD\eta\|_{L^2}^2\|\partial_t\partial_1^{7-m}Dv\|_{L^\infty}^2\,\diff s
+G, \quad m=6,7
\end{aligned}
\end{equation}
with \(G=\int_0^t\|\sqrt{\rho_0}\partial_1^7Dv\|_{L^2}^2\,\diff s+\mathcal{R}\), which have the desired bound. 
Here one has used \eqref{EEP-6} in \(\eqref{EE-60}_1\).

Finally, for \(I_{3_m}\), one can use the observation that
\begin{equation*}
\begin{aligned}
Q_4(t)\lesssim P_{**}(t)
\end{aligned}
\end{equation*}
and a similar argument for \(I_{2_m}\) to show that \(I_{3_m}\) has the desired bound.

\end{proof}

\bigskip

Collecting all the inequalities in Propositions \ref{pr:ene-1}-\ref{pr:ene-7}, one concludes that
\begin{equation}\label{APBEN}
\begin{aligned}
E_{\text{en}}(t,v)
\leq M_0+CtP(\sup_{0\leq s\leq t}E^{1/2}(s,v))\quad \mathrm{for\ all}\ t\in [0,T].
\end{aligned}
\end{equation}

\section{Elliptic Estimates}\label{Elliptic Estimates}

Having energy estimates on the tangential derivatives in Section \ref{Energy Estimates}, we will use the
elliptic theory to gain regularities for the normal derivatives of the solutions in this section.  
Let \((v,\eta)\) be a solution to \eqref{eq:main-2} satisfying \eqref{a priori assumption}.

Throughout this section, \(Q_*(t)\) denotes a generic polynomial function of the argument \([M_0+Ct P(\sup_{0\leq s\leq t}E^{1/2}(s,v))]^{1/2}\), which may be different from line to line.

\subsection{Lower-order estimates}\label{LOP}

\begin{proposition}\label{pr:ell-1} It holds that
\begin{equation}\label{EEQ-1}
\begin{aligned}
\int_\Omega \rho_0^2|\partial_2^2v|^2(t)\,\diff x
\leq M_0+Ct P(\sup_{0\leq s\leq t}E^{1/2}(s,v)).
\end{aligned}
\end{equation}
\end{proposition}
\begin{proof}
By Lemma \ref{le:W-1} and Corollary \ref{le:W-2}, one finds that
\begin{equation}\label{ELE-1}
\begin{aligned}
\|(\rho_0^2 J^{-2}a_i^k),_{k}\|_{L^2}^2
&\lesssim 
\|\partial_k(\rho_0^2) J^{-2}a_i^k\|_{L^2}^2+\|\rho_0^2 \partial_k(J^{-2}a_i^k)\|_{L^2}^2\\
&\lesssim \|D\eta\|_{L^2}^2+Q_*^2
\leq M_0+Ct P(\sup_{0\leq s\leq t}E^{1/2}(s,v)).
\end{aligned}
\end{equation}
It follows from \eqref{eq:main-3} that
\begin{equation}\label{ELE-2}
\begin{aligned}
\|(\rho_0J^{-2}b^{kj}v^i,_{j}),_{k}\|_{L^2}^2
&\leq \|\rho_0\partial_tv^i\|_{L^2}^2+\|(\rho_0^2 J^{-2}a_i^k),_{k}\|_{L^2}^2\\
&\leq M_0+Ct P(\sup_{0\leq s\leq t}E^{1/2}(s,v)),
\end{aligned}
\end{equation}
where \eqref{EEP-1} and \eqref{ELE-1} have been used.

Note that
\begin{equation}\label{ELE-3}
\begin{aligned}
(\rho_0J^{-2}b^{22}v^i,_{2}),_{2}
=(\rho_0J^{-2}b^{kj}v^i,_{j}),_{k}-\sum_{(k,j)\neq (2,2)}(\rho_0J^{-2}b^{kj}v^i,_{j}),_{k}
\end{aligned}
\end{equation}
and 
{\small \begin{equation}\label{ELE-3-add}
\begin{aligned}
\sum_{(k,j)\neq (2,2)}(\rho_0J^{-2}b^{kj}v^i,_{j}),_{k}
&=2\underbrace{\rho_0J^{-2}b^{21}\partial_2v^i,_{1}}_{I_1}+\sum_{(k,j)\neq (2,2)}\underbrace{\rho_0\partial_k(J^{-2}b^{kj})v^i,_{j}}_{I_2}\\
&\quad+\sum_{(k,j)\neq (2,2)}\underbrace{\partial_k(\rho_0)J^{-2}b^{kj}v^i,_{j}}_{I_3},
\end{aligned}
\end{equation}}
where one has used the fact that \((b^{kj})\) is a symmetric matrix. 

For \(I_1\) and \(I_2\), one may use \eqref{EEP-5} and \eqref{EEP-2} to estimate 
\begin{equation}\label{ELE-4}
\begin{aligned}
\|I_1\|_{L^2}^2
\lesssim Q_*^2\|\rho_0\partial_1Dv\|_{L^2}^2
\leq M_0+Ct P(\sup_{0\leq s\leq t}E^{1/2}(s,v))
\end{aligned}
\end{equation}
and
\begin{equation}\label{ELE-5}
\begin{aligned}
\|I_2\|_{L^2}^2
\lesssim Q_*^2\|\rho_0Dv\|_{L^2}^2
\leq M_0+Ct P(\sup_{0\leq s\leq t}E^{1/2}(s,v)),
\end{aligned}
\end{equation}
respectively.
For \(I_3\), one shall consider two cases: 
when \((k,j)=(1,2)\)
\begin{equation}\label{ELE-6}
\begin{aligned}
\|I_3\|_{L^2}^2\lesssim Q_*^2\bigg\|\frac{\partial_1(\rho_0)}{\rho_0}\bigg\|_{L^\infty}^2\|\rho_0Dv\|_{L^2}^2
\leq M_0+Ct P(\sup_{0\leq s\leq t}E^{1/2}(s,v)),
\end{aligned}
\end{equation}
where one has used \eqref{TE} and \eqref{EEP-2}, and when \((k,j)=(2,1)\)
\begin{equation}\label{ELE-7}
\begin{aligned}
\|I_3\|_{L^2}^2\lesssim Q_*^2\|\partial_1v\|_{L^2}^2
\leq M_0+Ct P(\sup_{0\leq s\leq t}E^{1/2}(s,v)),
\end{aligned}
\end{equation}
where one has used \eqref{ineq:weighted Sobolev-2} and \eqref{EEP-5} to estimate
\begin{equation*}
\begin{aligned}
\|\partial_1v\|_{L^2}^2
\lesssim \|\rho_0\partial_1v\|_{L^2}^2+\|\rho_0\partial_1Dv\|_{L^2}^2
\leq M_0+Ct P(\sup_{0\leq s\leq t}E^{1/2}(s,v)).
\end{aligned}
\end{equation*}

It follows from \eqref{ELE-2}-\eqref{ELE-7} that
\begin{equation}\label{ELE-7-add}
\begin{aligned}
\|(\rho_0J^{-2}b^{22}v^i,_{2}),_{2}\|_{L^2}^2
\leq M_0+CtP(\sup_{0\leq s\leq t}E^{1/2}(s,v)).
\end{aligned}
\end{equation}

Note that
\begin{equation*}
\begin{aligned}
\rho_0J^{-2}b^{22}\partial_2v^i,_{2}+\partial_2(\rho_0)J^{-2}b^{22}v^i,_{2}
=(\rho_0J^{-2}b^{22}v^i,_{2}),_{2}-\rho_0\partial_2(J^{-2}b^{22})v^i,_{2}.
\end{aligned}
\end{equation*}
Clearly, the last term can be dealt as \(I_2\). This fact together with \eqref{ELE-7-add} gives 
\begin{equation}\label{ELE-8}
\begin{aligned}
\|\rho_0J^{-2}b^{22}\partial_2v^i,_{2}+\partial_2(\rho_0)J^{-2}b^{22}v^i,_{2}\|_{L^2}^2
\leq M_0+CtP(\sup_{0\leq s\leq t}E^{1/2}(s,v)).
\end{aligned}
\end{equation}

To get an estimate of \(\rho_0J^{-2}b^{22}\partial_2v^i,_{2}\) in \(L^2\), one first notices that
\begin{equation}\label{ELE-9}
\begin{aligned}
\|\rho_0J^{-2}b^{22}\partial_2v^i,_{2}\|_{L^2}^2
&=\|\rho_0J^{-2}b^{22}\partial_2v^i,_{2}+\partial_2(\rho_0)J^{-2}b^{22}v^i,_{2}\|_{L^2}^2\\
&\quad-\|\partial_2(\rho_0)J^{-2}b^{22}v^i,_{2}\|_{L^2}^2\\
&\quad-\int_\Omega\rho_0\partial_2(\rho_0)(J^{-2}b^{22})^2\partial_2[(v^i,_{2})^2]\,\diff x,
\end{aligned}
\end{equation}
and then integrates by parts on the last term of \eqref{ELE-9} to deduce that
\begin{equation}\label{ELE-10}
\begin{aligned}
\|\rho_0J^{-2}b^{22}\partial_2v^i,_{2}\|_{L^2}^2
&=\|\rho_0J^{-2}b^{22}\partial_2v^i,_{2}+\partial_2(\rho_0)J^{-2}b^{22}v^i,_{2}\|_{L^2}^2\\
&\quad+
\underline{\int_\Omega\rho_0\partial_2[\partial_2(\rho_0)(J^{-2}b^{22})^2](v^i,_{2})^2\,\diff x}_{:=K}\\
&\leq M_0+Ct P(\sup_{0\leq s\leq t}E^{1/2}(s,v)),
\end{aligned}
\end{equation}
where one has used \eqref{ELE-8} and the estimate
\begin{equation*}
\begin{aligned}
|K|
\lesssim Q_{*}^2\int_\Omega\rho_0|Dv|^2\,\diff x
\leq M_0+Ct P(\sup_{0\leq s\leq t}E^{1/2}(s,v)),
\end{aligned}
\end{equation*}
in which Corollary \ref{le:W-2} and \eqref{EEP-2} have been used. 

Hence \eqref{EEQ-1} follows from \eqref{J-bound}, \eqref{b-bound} and \eqref{ELE-10}.

\end{proof}

\subsection{Intermediated estimates-I}\label{IP-I}

\subsubsection{Estimates on \(\partial_1\partial_2^2v\) and \(\partial_2^3v\)} \label{P2}
\begin{proposition}\label{pr:ell-2}
 It holds that
\begin{equation}\label{EEQ-2}
\begin{aligned}
\sum_{l=2}^3\int_\Omega \rho_0^l|\partial_1^{3-l}\partial_2^{l}v|^2(t)\,\diff x
\leq M_0+Ct P(\sup_{0\leq s\leq t}E^{1/2}(s,v)).
\end{aligned}
\end{equation}

\end{proposition}

\begin{proof}

By Lemma \ref{le:W-1}, Corollary \ref{le:W-2} and \eqref{Jab-bound-12}, one deduces that
\begin{equation*}
\begin{aligned}
Q_6(t)\leq Q_*(t).
\end{aligned}
\end{equation*}
This, together with \eqref{Jab-bound-11}, leads to
\begin{equation}\label{ELE-11}
\begin{aligned}
\|D(\rho_0^2 J^{-2}a_i^k),_{k}\|_{L^2}^2
&\lesssim 
\|D^2(\rho_0^2)J^{-2}a_i^k\|_{L^2}^2+\|D(\rho_0^2) D(J^{-2}a_i^k)\|_{L^2}^2\\
&\quad+\|\rho_0^2 D^2(J^{-2}a_i^k)\|_{L^2}^2\\
&\lesssim \|D\eta\|_{L^\infty}^2+Q_{*}^2+Q_{*}^2(\|D^3\eta\|_{L^2}+1)^2\\
&\leq M_0+Ct P(\sup_{0\leq s\leq t}E^{1/2}(s,v)).
\end{aligned}
\end{equation}

We first prove \eqref{EEQ-2} for \(l=2\). 
Applying \(\partial_1\) to \eqref{eq:main-3} yields
\begin{equation}\label{ELE-12}
\begin{aligned}
\|\partial_1(\rho_0J^{-2}b^{kj}v^i,_{j}),_{k}\|_{L^2}^2
&\leq \|\partial_1(\rho_0\partial_tv^i)\|_{L^2}^2+\|\partial_1(\rho_0^2 J^{-2}a_i^k),_{k}\|_{L^2}^2\\
&\lesssim \|\rho_0\partial_t\partial_1v\|_{L^2}^2+\|D(\rho_0^2 J^{-2}a_i^k),_{k}\|_{L^2}^2\\
&\leq M_0+Ct P(\sup_{0\leq s\leq t}E^{1/2}(s,v)),
\end{aligned}
\end{equation}
where \eqref{TE}, \eqref{EEP-2} and \eqref{ELE-11} have been used.

By \eqref{TE}, one may write
\begin{equation}\label{ELE-13}
\begin{aligned}
&(\rho_0J^{-2}b^{kj}\partial_1v^i,_{j}),_{k}
=\partial_1(\rho_0J^{-2}b^{kj}v^i,_{j}),_{k}-[\partial_1(\rho_0J^{-2}b^{kj})v^i,_{j}],_{k}\\
&=\partial_1(\rho_0J^{-2}b^{kj}v^i,_{j}),_{k}-[\rho_0\partial_1(J^{-2}b^{kj})v^i,_{j}],_{k}+\text{l.o.t.}
\end{aligned}
\end{equation}
with
\begin{equation}\label{ELE-13-add}
	\begin{aligned}
		{[\rho_0\partial_1(J^{-2}b^{kj})v^i,_{j}]},_{k}
		&=\underbrace{\partial_k(\rho_0)\partial_1(J^{-2}b^{kj})v^i,_{j}}_{I_4}+\underbrace{\rho_0\partial_1\partial_k(J^{-2}b^{kj})v^i,_{j}}_{I_5}\\
		&\quad+\underbrace{\rho_0\partial_1(J^{-2}b^{kj})\partial_kv^i,_{j}}_{I_6}.
	\end{aligned}
\end{equation}

The RHS terms in \eqref{ELE-13-add} can be estimated as follows: 
\begin{equation}\label{ELE-14}
\begin{aligned}
\|I_4\|_{L^2}^2+\|I_5\|_{L^2}^2
\lesssim Q_{*}^2\|Dv\|_{L^2}^2
\leq M_0+Ct P(\sup_{0\leq s\leq t}E^{1/2}(s,v))
\end{aligned}
\end{equation}
and
\begin{equation}\label{ELE-15}
\begin{aligned}
\|I_6\|_{L^2}^2
\lesssim Q_{*}^2\|\rho_0D^2v\|_{L^2}^2
\leq M_0+Ct P(\sup_{0\leq s\leq t}E^{1/2}(s,v)),
\end{aligned}
\end{equation}
where one has used \eqref{EEQ-1} and 
\begin{equation*}
\begin{aligned}
\|Dv\|_{L^2}^2
\lesssim \|\rho_0Dv\|_{L^2}^2+\|\rho_0D^2v\|_{L^2}^2
\leq M_0+Ct P(\sup_{0\leq s\leq t}E^{1/2}(s,v)).
\end{aligned}
\end{equation*}

It follows from \eqref{ELE-13}-\eqref{ELE-15} that
\begin{equation*}
\begin{aligned}
\|[\rho_0\partial_1(J^{-2}b^{kj})v^i,_{j}],_{k}\|_{L^2}^2
\leq M_0+Ct P(\sup_{0\leq s\leq t}E^{1/2}(s,v)),
\end{aligned}
\end{equation*}
which, together with \eqref{ELE-12}, yields
\begin{equation}\label{ELE-16}
\begin{aligned}
\|(\rho_0J^{-2}b^{kj}\partial_1v^i,_{j}),_{k}\|_{L^2}^2
\leq M_0+Ct P(\sup_{0\leq s\leq t}E^{1/2}(s,v)).
\end{aligned}
\end{equation}

Similar to \eqref{ELE-3} and \eqref{ELE-3-add}, one may further write
\begin{equation}\label{ELE-17}
\begin{aligned}
(\rho_0J^{-2}b^{22}\partial_1v^i,_{2}),_{2}
=(\rho_0J^{-2}b^{kj}\partial_1v^i,_{j}),_{k}-\sum_{(k,j)\neq (2,2)}(\rho_0J^{-2}b^{kj}\partial_1v^i,_{j}),_{k}
\end{aligned}
\end{equation}
with
{\small\begin{equation}\label{ELE-18}
\begin{aligned}
\sum_{(k,j)\neq (2,2)}(\rho_0J^{-2}b^{kj}\partial_1v^i,_{j}),_{k}
&=2\underbrace{\rho_0J^{-2}b^{21}\partial_1\partial_2v^i,_{1}}_{I_7}+\sum_{(k,j)\neq (2,2)}\underbrace{\rho_0\partial_k(J^{-2}b^{kj})\partial_1v^i,_{j}}_{I_8}\\
&\quad+\sum_{(k,j)\neq (2,2)}\underbrace{\partial_k(\rho_0)J^{-2}b^{kj}\partial_1v^i,_{j}}_{I_9}.
\end{aligned}
\end{equation}}

 First, \(I_8\) can be handled as \(I_6\) (since  only the structure of \(D(J^{-2}b^{kj})\) was used in \eqref{ELE-15}). 
For \(I_7\), one may use \eqref{EEP-add-1} to estimate  
\begin{equation}\label{ELE-19}
\begin{aligned}
\|I_7\|_{L^2}^2\lesssim Q_{*}^2\|\rho_0\partial_1^2Dv\|_{L^2}^2\leq M_0+Ct P(\sup_{0\leq s\leq t}E^{1/2}(s,v)).
\end{aligned}
\end{equation}
For \(I_9\), when \((k,j)=(1,2)\)
\begin{equation}\label{ELE-20}
\begin{aligned}
\|I_9\|_{L^2}^2\lesssim Q_{*}^2\|\rho_0\partial_1Dv\|_{L^2}^2
\leq M_0+Ct P(\sup_{0\leq s\leq t}E^{1/2}(s,v))
\end{aligned}
\end{equation}
and when \((k,j)=(2,1)\)
\begin{equation}\label{ELE-21}
\begin{aligned}
\|I_9\|_{L^2}^2
\lesssim Q_{*}^2(\|\rho_0\partial_1^2v\|_{L^2}^2+\|\rho_0\partial_1^2Dv\|_{L^2}^2)
\leq M_0+Ct P(\sup_{0\leq s\leq t}E^{1/2}(s,v)).
\end{aligned}
\end{equation}

It follows from \eqref{ELE-16}-\eqref{ELE-21} that
\begin{equation}\label{ELE-22}
\begin{aligned}
\|(\rho_0J^{-2}b^{22}\partial_1v^i,_{2}),_{2}\|_{L^2}^2
\leq M_0+Ct P(\sup_{0\leq s\leq t}E^{1/2}(s,v)).
\end{aligned}
\end{equation}

Note that
\begin{equation*}
\begin{aligned}
&\rho_0J^{-2}b^{22}\partial_1\partial_2v^i,_{2}+\partial_2(\rho_0)J^{-2}b^{22}\partial_1v^i,_{2}\\
&=(\rho_0J^{-2}b^{22}\partial_1v^i,_{2}),_{2}-\rho_0\partial_2(J^{-2}b^{22})\partial_1v^i,_{2}.
\end{aligned}
\end{equation*}
Since the last term can be estimated as \(I_6\), which together with \eqref{ELE-22} yields   
\begin{equation}\label{ELE-23}
\begin{aligned}
\|\rho_0J^{-2}b^{22}\partial_1\partial_2v^i,_{2}+\partial_2(\rho_0)J^{-2}b^{22}\partial_1v^i,_{2}\|_{L^2}^2
\leq M_0+CtP(\sup_{0\leq s\leq t}E^{1/2}(s,v)).
\end{aligned}
\end{equation}

Then, by \eqref{EEP-5} and \eqref{ELE-23}, one integrates by parts to deduce that
\begin{equation}\label{ELE-24}
\begin{aligned}
&\|\rho_0J^{-2}b^{22}\partial_1\partial_2v^i,_{2}\|_{L^2}^2\\
&=\|\rho_0J^{-2}b^{22}\partial_1\partial_2v^i,_{2}+\partial_2(\rho_0)J^{-2}b^{22}\partial_1v^i,_{2}\|_{L^2}^2\\
&\quad+
\int_\Omega\rho_0\partial_2[\partial_2(\rho_0)(J^{-2}b^{22})^2](\partial_1v^i,_{2})^2\,\diff x\\
&\lesssim\|\rho_0J^{-2}b^{22}\partial_1\partial_2v^i,_{2}+\partial_2(\rho_0)J^{-2}b^{22}\partial_1v^i,_{2}\|_{L^2}^2\\
&\quad+ Q_{*}^2
\int_\Omega\rho_0|\partial_1Dv|^2\,\diff x
\leq M_0+Ct P(\sup_{0\leq s\leq t}E^{1/2}(s,v)).
\end{aligned}
\end{equation}

Finally, \eqref{EEQ-2} for \(l=2\) is a consequence of \eqref{J-bound}, \eqref{b-bound} and \eqref{ELE-24}.  \\

Now, we show \eqref{EEQ-2} for \(l=3\).  Applying \(\partial_2\) to \eqref{eq:main-3} and using \eqref{ELE-11} yield
\begin{equation}\label{ELE-25}
\begin{aligned}
\|\partial_2(\rho_0J^{-2}b^{kj}v^i,_{j}),_{k}\|_{L^2}^2
&\lesssim \|\partial_tv^i\|_{L^2}^2+\|\rho_0\partial_tDv^i\|_{L^2}^2+\|D(\rho_0^2 J^{-2}a_i^k),_{k}\|_{L^2}^2\\
&\lesssim \|\rho_0\partial_tv\|_{L^2}^2+\|\rho_0\partial_tDv\|_{L^2}^2+\|D(\rho_0^2 J^{-2}a_i^k),_{k}\|_{L^2}^2\\
&\leq M_0+Ct P(\sup_{0\leq s\leq t}E^{1/2}(s,v)).
\end{aligned}
\end{equation}

As for \(\partial_1\partial_2^2v\), we first note that
\begin{equation*}
\begin{aligned}
(\rho_0J^{-2}b^{kj}\partial_2v^i,_{j}),_{k}
&=\partial_2(\rho_0J^{-2}b^{kj}v^i,_{j}),_{k}-[\rho_0\partial_2(J^{-2}b^{kj})v^i,_{j}],_{k}\\
&\quad-[\partial_2(\rho_0)J^{-2}b^{kj}v^i,_{j}],_{k}
\end{aligned}
\end{equation*}
with
\begin{equation*}
\begin{aligned}
{[\rho_0\partial_2(J^{-2}b^{kj})v^i,_{j}]},_{k}
&=\partial_k(\rho_0)\partial_2(J^{-2}b^{kj})v^i,_{j}+\rho_0\partial_2\partial_k(J^{-2}b^{kj})v^i,_{j}\\
&\quad+\rho_0\partial_2(J^{-2}b^{kj})\partial_kv^i,_{j}
\end{aligned}
\end{equation*}
and
\begin{equation*}
\begin{aligned}
{[\partial_2(\rho_0)J^{-2}b^{kj}v^i,_{j}]},_{k}&=\underbrace{\partial_k\partial_2(\rho_0)J^{-2}b^{kj}v^i,_{j}}_{I_{10}}+\underbrace{\partial_2(\rho_0)\partial_k(J^{-2}b^{kj})v^i,_{j}}_{I_{11}}\\
&\quad+\underbrace{\partial_2(\rho_0)J^{-2}b^{kj}\partial_kv^i,_{j}}_{I_{12}}.
\end{aligned}
\end{equation*}

Clearly, \([\rho_0\partial_2(J^{-2}b^{kj})v^i,_{j}],_{k}\) can be dealt as \([\rho_0\partial_1(J^{-2}b^{kj})v^i,_{j}],_{k}\) (since only the structure of \(D(J^{-2}b^{kj})\) and \(D^2(J^{-2}b^{kj})\) in \eqref{ELE-14}-\eqref{ELE-15} has been used). Similarly, \(I_{10}\) and \(I_{11}\) can be handled as \(I_4\). 

It remains to analyze \(I_{12}\). When \((k,j)\neq (2,2)\), it holds that
\begin{equation*}
\begin{aligned}
\|I_{12}\|_{L^2}^2\lesssim Q_{*}^2\|\partial_1Dv\|_{L^2}^2
\leq M_0+Ct P(\sup_{0\leq s\leq t}E^{1/2}(s,v)),
\end{aligned}
\end{equation*}
where one has used \eqref{ineq:weighted Sobolev-2} and \eqref{EEQ-2} for \(l=2\) to estimate 
\begin{equation*}
\begin{aligned}
\|\partial_1Dv\|_{L^2}^2
\lesssim \|\rho_0\partial_1Dv\|_{L^2}^2+\|\rho_0\partial_1D^2v\|_{L^2}^2
\leq M_0+Ct P(\sup_{0\leq s\leq t}E^{1/2}(s,v)).
\end{aligned}
\end{equation*}

Collecting all the cases above and using \eqref{ELE-25} lead to
\begin{equation}\label{ELE-26}
\begin{aligned}
(\rho_0J^{-2}b^{kj}\partial_2v^i,_{j}),_{k}+\partial_2(\rho_0)J^{-2}b^{22}\partial_2v^i,_{2}
=\mathcal{G}_1
\end{aligned}
\end{equation}
with 
\begin{equation*}
\begin{aligned}
\|\mathcal{G}_1\|_{L^2}^2
\leq M_0+Ct P(\sup_{0\leq s\leq t}E^{1/2}(s,v)).
\end{aligned}
\end{equation*}

Next, one further writes \eqref{ELE-26} as 
\begin{equation*}
\begin{aligned}
(\rho_0J^{-2}b^{22}\partial_2v^i,_{2}),_{2}+\partial_2(\rho_0)J^{-2}b^{22}\partial_2v^i,_{2}
=\mathcal{G}_1-\sum_{(k,j)\neq (2,2)}(\rho_0J^{-2}b^{kj}\partial_2v^i,_{j}),_{k}
\end{aligned}
\end{equation*}
with
{\small\begin{equation*}
\begin{aligned}
\sum_{(k,j)\neq (2,2)}(\rho_0J^{-2}b^{kj}\partial_2v^i,_{j}),_{k}
&=2\underbrace{\rho_0J^{-2}b^{21}\partial_2^2v^i,_{1}}_{I_{13}}+\sum_{(k,j)\neq (2,2)}\underbrace{\rho_0\partial_k(J^{-2}b^{kj})\partial_2v^i,_{j}}_{I_{14}}\\
&\quad+\sum_{(k,j)\neq (2,2)}\underbrace{\partial_k(\rho_0)J^{-2}b^{kj}\partial_2v^i,_{j}}_{I_{15}}.
\end{aligned}
\end{equation*}}

 Note that \(I_{14}\) can be handled as \(I_6\). For \(I_{13}\), \eqref{EEQ-2} for \(l=2\) implies that 
\begin{equation*}
\begin{aligned}
\|I_{13}\|_{L^2}^2\lesssim Q_{*}^2\|\rho_0\partial_1D^2v\|_{L^2}^2\leq M_0+Ct P(\sup_{0\leq s\leq t}E^{1/2}(s,v)).
\end{aligned}
\end{equation*}
For \(I_{15}\), if \((k,j)=(1,2)\), then
\begin{equation*}
\begin{aligned}
\|I_{15}\|_{L^2}^2\lesssim Q_{*}^2\|\rho_0D^2v\|_{L^2}^2
\leq M_0+Ct P(\sup_{0\leq s\leq t}E^{1/2}(s,v)),
\end{aligned}
\end{equation*}
while for \((k,j)=(2,1)\), it holds that
\begin{equation*}
\begin{aligned}
\|I_{15}\|_{L^2}^2
\lesssim Q_{*}^2\|\partial_1Dv\|_{L^2}^2
\leq M_0+Ct P(\sup_{0\leq s\leq t}E^{1/2}(s,v)).
\end{aligned}
\end{equation*}

Consequently, 
\begin{equation*}
\begin{aligned}
(\rho_0J^{-2}b^{22}\partial_2v^i,_{2}),_{2}+\partial_2(\rho_0)J^{-2}b^{22}\partial_2v^i,_{2}
=\mathcal{G}_2
\end{aligned}
\end{equation*}
with 
\begin{equation*}
\begin{aligned}
\|\mathcal{G}_2\|_{L^2}^2
\leq M_0+Ct P(\sup_{0\leq s\leq t}E^{1/2}(s,v)),
\end{aligned}
\end{equation*}
which can be rewritten as 
\begin{equation*}
\begin{aligned}
\rho_0J^{-2}b^{22}\partial_2^2v^i,_{2}+2\partial_2(\rho_0)J^{-2}b^{22}\partial_2v^i,_{2}
=\mathcal{G}_2-\rho_0\partial_2(J^{-2}b^{22})\partial_2v^i,_{2}.
\end{aligned}
\end{equation*}
Note that the last term has the desired bound. Hence, 
\begin{equation}\label{ELE-28}
\begin{aligned}
\|\rho_0J^{-2}b^{22}\partial_2^2v^i,_{2}+2\partial_2(\rho_0)J^{-2}b^{22}\partial_2v^i,_{2}\|_{L^2}^2
\leq M_0+CtP(\sup_{0\leq s\leq t}E^{1/2}(s,v)).
\end{aligned}
\end{equation}

Then, one uses a weight \(\sqrt{\rho_0}\)
and integrates by parts to deduce that
\begin{equation}\label{ELE-29}
\begin{aligned}
&\|\rho_0^{3/2}J^{-2}b^{22}\partial_2^2v^i,_{2}\|_{L^2}^2\\
&=\|\rho_0^{3/2}J^{-2}b^{22}\partial_2^2v^i,_{2}+2\partial_2(\rho_0)J^{-2}b^{22}\partial_2v^i,_{2}\|_{L^2}^2\\
&\quad+
2\int_\Omega\rho_0^2\partial_2[\partial_2(\rho_0)(J^{-2}b^{22})^2](\partial_2v^i,_{2})^2\,\diff x\\
&\lesssim\|\rho_0J^{-2}b^{22}\partial_2^2v^i,_{2}+2\partial_2(\rho_0)J^{-2}b^{22}\partial_2v^i,_{2}\|_{L^2}^2\\
&\quad+ Q_{*}^2
\int_\Omega\rho_0^2|\partial_2^2v|^2\,\diff x
\leq M_0+Ct P(\sup_{0\leq s\leq t}E^{1/2}(s,v)),
\end{aligned}
\end{equation}
where \eqref{ELE-28} and \eqref{EEQ-1} have been used.

Finally,  \eqref{EEQ-2} for \(l=3\) follows from \eqref{J-bound}, \eqref{b-bound} and \eqref{ELE-29}.

\end{proof}

\subsubsection{Estimates on \(\partial_t\partial_2^2v\)}\label{TN} 
\begin{proposition}\label{pr:ell-3} It holds that
\begin{equation}\label{EEQ-4}
\begin{aligned}
\int_\Omega \rho_0^2|\partial_t\partial_2^2v|^2(t)\,\diff x
\leq M_0+Ct P(\sup_{0\leq s\leq t}E^{1/2}(s,v)).
\end{aligned}
\end{equation}
\end{proposition}

\begin{proof}
By \eqref{EEP-1}, \eqref{EEP-2} and \eqref{EEQ-1}, it is easy to deduce by 
applying \(\partial_t\) to \eqref{eq:main-3} that
\begin{equation}\label{ELE-31}
\begin{aligned}
\|\partial_t(\rho_0J^{-2}b^{kj}v^i,_{j}),_{k}\|_{L^2}^2
&\leq \|\rho_0\partial_t^2v^i\|_{L^2}^2+\|\partial_t(\rho_0^2 J^{-2}a_i^k),_{k}\|_{L^2}^2\\
&\lesssim \|\rho_0\partial_t^2v\|_{L^2}^2+Q_{*}^2(\|\rho_0D^2v\|_{L^2}^2+\|\rho_0Dv\|_{L^2}^2)\\
&\leq M_0+Ct P(\sup_{0\leq s\leq t}E^{1/2}(s,v)).
\end{aligned}
\end{equation}

Note that
\begin{equation}\label{ELE-32}
\begin{aligned}
(\rho_0J^{-2}b^{kj}\partial_tv^i,_{j}),_{k}
&=\partial_t(\rho_0J^{-2}b^{kj}v^i,_{j}),_{k}-\underbrace{\partial_k(\rho_0)\partial_t(J^{-2}b^{kj})v^i,_{j}}_{I_{16}}\\
&\quad-\underbrace{\rho_0\partial_t\partial_k(J^{-2}b^{kj})v^i,_{j}}_{I_{17}}-\underbrace{\rho_0\partial_t(J^{-2}b^{kj})\partial_kv^i,_{j}}_{I_{18}}.
\end{aligned}
\end{equation}

To handle \(I_{16}\), \(I_{17}\) and \(I_{18}\), one first uses \eqref{ineq:weighted Sobolev-1}, \eqref{ineq:weighted Sobolev-2} and \eqref{EEQ-2} to estimate 
%\footnote{The proof is similar to \eqref{Weighted-explain-1}, but the output \(E(t,v)\) should be replaced by a better bound \(M_0+Ct P(\sup_{0\leq s\leq t}E^{1/2}(s,v)\)).}
\begin{equation*}
\begin{aligned}
\|Dv\|_{H^{1/2}}^2
&\lesssim \|\rho_0^{1/2}Dv\|_{L^2}^2+\|\rho_0^{1/2}D^2v\|_{L^2}^2\\
&\lesssim \|\rho_0^{1/2}Dv\|_{L^2}^2+(\|\rho_0^{3/2}D^2v\|_{L^2}^2+\|\rho_0^{3/2}D^3v\|_{L^2}^2)\\
&\leq  M_0+Ct P(\sup_{0\leq s\leq t}E^{1/2}(s,v)).
\end{aligned}
\end{equation*}
 Similarly, one can also show
\begin{equation*}
	\begin{aligned}
		\|\rho_0D^2v\|_{H^{1/2}}^2
		\leq  M_0+Ct P(\sup_{0\leq s\leq t}E^{1/2}(s,v)).
	\end{aligned}
\end{equation*}
Then it holds that
\begin{equation}\label{ELE-33}
\begin{aligned}
&\|I_{16}\|_{L^2}^2\lesssim  Q_{*}^2\|Dv\|_{H^{1/2}}^4,\\
&\|I_{17}\|_{L^2}^2\lesssim  Q_{*}^2(\|\rho_0D^2v\|_{H^{1/2}}^2\|Dv\|_{H^{1/2}}^2
+\|Dv\|_{L^2}^2),\\
&\|I_{18}\|_{L^2}^2\lesssim Q_{*}^2\|Dv\|_{H^{1/2}}^2\|\rho_0D^2v\|_{H^{1/2}}^2,
\end{aligned}
\end{equation}
which satisfy the desired bound. 

It follows \eqref{ELE-31}-\eqref{ELE-33} that
\begin{equation}\label{ELE-36}
\begin{aligned}
\|(\rho_0J^{-2}b^{kj}\partial_tv^i,_{j}),_{k}\|_{L^2}^2
\leq M_0+Ct P(\sup_{0\leq s\leq t}E^{1/2}(s,v)).
\end{aligned}
\end{equation}

Note that
\begin{equation}\label{ELE-37}
\begin{aligned}
&\rho_0J^{-2}b^{22}\partial_t\partial_2v^i,_{2}+\partial_2(\rho_0)J^{-2}b^{22}\partial_tv^i,_{2}\\
&=(\rho_0J^{-2}b^{kj}\partial_tv^i,_{j}),_{k}-\sum_{(k,j)\neq (2,2)}(\rho_0J^{-2}b^{kj}\partial_tv^i,_{j}),_{k}
\end{aligned}
\end{equation}
and
{\small\begin{equation}\label{ELE-38}
\begin{aligned}
\sum_{(k,j)\neq (2,2)}(\rho_0J^{-2}b^{kj}\partial_tv^i,_{j}),_{k}
&=2\underbrace{\rho_0J^{-2}b^{21}\partial_t\partial_2v^i,_{1}}_{I_{19}}+
\sum_{(k,j)\neq (2,2)}\underbrace{\rho_0\partial_k(J^{-2}b^{kj})\partial_tv^i,_{j}}_
{I_{20}}\\
&\quad+\sum_{(k,j)\neq (2,2)}\underbrace{\partial_k(\rho_0)J^{-2}b^{kj}\partial_tv^i,_{j}}_{I_{21}}.
\end{aligned}
\end{equation}}

First, one may use  \eqref{EEP-5} and \eqref{EEP-2} to estimate
\begin{equation}\label{ELE-39}
\begin{aligned}
&\|I_{19}\|_{L^2}^2
\lesssim Q_{*}^2\|\rho_0\partial_t\partial_1Dv\|_{L^2}^2
\leq M_0+Ct P(\sup_{0\leq s\leq t}E^{1/2}(s,v)),\\
&\|I_{20}\|_{L^2}^2
\lesssim Q_{*}^2\|\rho_0\partial_tDv\|_{L^2}^2
\leq M_0+Ct P(\sup_{0\leq s\leq t}E^{1/2}(s,v)).
\end{aligned}
\end{equation}
Next, it remains to consider \(I_{21}\). In fact, 
for \((k,j)=(1,2)\),
\begin{equation}\label{ELE-41}
\begin{aligned}
\|I_{21}\|_{L^2}^2\lesssim Q_{*}^2\|\rho_0\partial_tDv\|_{L^2}^2 
\leq M_0+Ct P(\sup_{0\leq s\leq t}E^{1/2}(s,v))
\end{aligned}
\end{equation}
and for \((k,j)=(2,1)\),
\begin{equation}\label{ELE-42}
\begin{aligned}
\|I_{21}\|_{L^2}^2\lesssim Q_{*}^2\|\partial_t\partial_1v\|_{L^2}^2 
\leq M_0+Ct P(\sup_{0\leq s\leq t}E^{1/2}(s,v)),
\end{aligned}
\end{equation}
where one has used 
\begin{equation*}
\begin{aligned}
\|\partial_t\partial_1v\|_{L^2}^2 
\lesssim \|\rho_0\partial_t\partial_1v\|_{L^2}^2+\|\rho_0\partial_t\partial_1Dv\|_{L^2}^2
\leq M_0+Ct P(\sup_{0\leq s\leq t}E^{1/2}(s,v)).
\end{aligned}
\end{equation*}

It follows from \eqref{ELE-36}-\eqref{ELE-42} that
\begin{equation*}
\begin{aligned}
\|\rho_0J^{-2}b^{22}\partial_t\partial_2v^i,_{2}+\partial_2(\rho_0)J^{-2}b^{22}\partial_tv^i,_{2}\|_{L^2}^2
\leq M_0+CtP(\sup_{0\leq s\leq t}E^{1/2}(s,v)).
\end{aligned}
\end{equation*}

This together with \eqref{EEP-2} and 
integration by parts yields
\begin{equation}\label{ELE-44}
\begin{aligned}
&\|\rho_0J^{-2}b^{22}\partial_t\partial_2v^i,_{2}\|_{L^2}^2\\
&\lesssim\|\rho_0J^{-2}b^{22}\partial_t\partial_2v^i,_{2}+\partial_2(\rho_0)J^{-2}b^{22}\partial_tv^i,_{2}\|_{L^2}^2\\
&\quad+ Q_{*}^2
\int_\Omega\rho_0|\partial_tDv|^2\,\diff x
\leq M_0+Ct P(\sup_{0\leq s\leq t}E^{1/2}(s,v)).
\end{aligned}
\end{equation}

Hence \eqref{EEQ-4} follows from \eqref{J-bound}, \eqref{b-bound} and \eqref{ELE-44}.

\end{proof}

\subsubsection{Estimates on \(\partial_1^2\partial_2^2v\), \(\partial_1\partial_2^3v\) and \(\partial_2^4v\)} 
\begin{proposition}\label{pr:ell-4} 
It holds that
\begin{equation}\label{EEQ-5}
\begin{aligned}
\sum_{\substack{ l_1+l_2=4\\  l_1\geq 0,\ l_2\geq 2}}\int_\Omega \rho_0^{l_2}|\partial_1^{l_1}\partial_2^{l_2}v|^2(t)\,\diff x
\leq M_0+Ct P(\sup_{0\leq s\leq t}E^{1/2}(s,v)).
\end{aligned}
\end{equation}
\end{proposition}

\begin{proof}  Based on \eqref{EEQ-2} and \eqref{EEQ-4},  \eqref{EEQ-5} can be verified by following a similar argument for \(\partial_1\partial_2^2v\) and \(\partial_2^3v\) in Proposition \ref{pr:ell-2}. 

\end{proof}

\subsection{Intermediated estimates-II}\label{IP-II} 
{\small\begin{proposition}\label{pr:ell-5} It holds that
\begin{equation}\label{EEQ-6}
\begin{aligned}
&\sum_{\substack{ l_1+l_2=3+2l_0\\  l_1\geq 0,\ l_2\geq 2}}\int_\Omega \rho_0^{l_2}|\partial_t^{1-l_0}\partial_1^{l_1}\partial_2^{l_2}v|^2(t)\,\diff x
\leq M_0+Ct P(\sup_{0\leq s\leq t}E^{1/2}(s,v)),\quad l_0=0,1,\\
&\int_\Omega \rho_0^2|\partial_t^2\partial_2^2v|^2(t)\,\diff x
\leq M_0+Ct P(\sup_{0\leq s\leq t}E^{1/2}(s,v)),\\
&\sum_{\substack{ l_1+l_2=4+2l_0\\   l_1\geq 0,\ l_2\geq 2}}\int_\Omega \rho_0^{l_2}|\partial_t^{1-l_0}\partial_1^{l_1}\partial_2^{l_2}v|^2(t)\,\diff x
\leq M_0+Ct P(\sup_{0\leq s\leq t}E^{1/2}(s,v)),\quad l_0=0,1.
\end{aligned}
\end{equation}

\end{proposition}}

\begin{proof}
The estimates \(\eqref{EEQ-6}_1\)-\(\eqref{EEQ-6}_5\) correspond successively to the following estimates in each row respectively: 
{\small \begin{equation}\label{Pic-1}
\begin{aligned}
&\partial_t\partial_1\partial_2^2v\rightarrow\ \ \partial_t\partial_2^3v\rightarrow\\
&\partial_1^3\partial_2^2v\rightarrow\quad\ \partial_1^2\partial_2^3v\rightarrow\quad \partial_1\partial_2^4v\rightarrow\ \partial_2^5v\rightarrow\\
&\partial_t^2\partial_2^2v\rightarrow\\
&\partial_t\partial_1^2\partial_2^2v\rightarrow\ \ \partial_t\partial_1\partial_2^3v\rightarrow \partial_t\partial_2^4v\rightarrow\\
&\partial_1^4\partial_2^2v\rightarrow\quad\ \ \partial_1^3\partial_2^3v\rightarrow\ \ \partial_1^2\partial_2^4v\rightarrow\ \partial_1\partial_2^5v\rightarrow\ \partial_2^6v
\end{aligned}
\end{equation}}

The strategy in showing Proposition \ref{pr:ell-5} is an application of a bootstrapping procedure: One first uses Propositions \ref{pr:ell-1}-\ref{pr:ell-4} to show the estimates in \(\eqref{EEQ-6}_1\) (\(l_0=0\)) along the arrow direction in the first row of \eqref{Pic-1}, and then uses \(\eqref{EEQ-6}_1\) to show \(\eqref{EEQ-6}_2\) (\(l_0=1\)) along the arrow direction in the second row of \eqref{Pic-1}, and so on, until one reaches \(\eqref{EEQ-6}_5\). \\

It is worth pointing out that one shall use a sharp estimate on \(\|Dv\|_{L^\infty}\) to show \(\eqref{EEQ-6}_3\) based on \(\eqref{EEQ-6}_2\), namely 
{\small\begin{equation}\label{Dv-estimate-2}
\begin{aligned}
&\|Dv\|_{L^\infty}^2
\lesssim \sum_{k=0}^2\int_\Omega \rho_0|D^{k+1}v|^2\,\diff x\\
&\lesssim M_0+Ct P(\sup_{0\leq s\leq t}E^{1/2}(s,v))+\bigg(\int_\Omega \rho_0^3|D^2v|^2\,\diff x+\int_\Omega \rho_0^3|D^3v|^2\,\diff x\bigg)\\
&\quad+\bigg(\int_\Omega \rho_0^3|D^3v|^2\,\diff x+\int_\Omega \rho_0^3|D^4v|^2\,\diff x\bigg)\\
&\lesssim M_0+Ct P(\sup_{0\leq s\leq t}E^{1/2}(s,v))+\bigg(\int_\Omega \rho_0^5|D^4v|^2\,\diff x+\int_\Omega \rho_0^5|D^5v|^2\,\diff x\bigg)\\
&\leq M_0+Ct P(\sup_{0\leq s\leq t}E^{1/2}(s,v)),
\end{aligned}
\end{equation}}
where \eqref{ineq:weighted Sobolev-2}, \eqref{ineq:weighted Sobolev-3} and \(\eqref{EEQ-6}_2\) have been used.

\end{proof}

\subsection{Higher-order estimates}\label{HOP}

{\small\begin{proposition}\label{pr:ell-6} It holds that
\begin{equation}\label{EEQ-11}
\begin{aligned}
&\sum_{\substack{ l_1+l_2=5+2l_0\\   l_1\geq 0,\ l_2\geq 2}}\int_\Omega \rho_0^{l_2}|\partial_t^{1-l_0}\partial_1^{l_1}\partial_2^{l_2}v|^2(t)\,\diff x
\leq M_0+Ct P(\sup_{0\leq s\leq t}E^{1/2}(s,v)), \quad l_0=0,1,2,\\
&\int_\Omega \rho_0^2|\partial_t^3\partial_2^2v|^2(t)\,\diff x
\leq M_0+Ct P(\sup_{0\leq s\leq t}E^{1/2}(s,v)),\\
&\sum_{\substack{ l_1+l_2=6+2l_0\\   l_1\geq 0,\ l_2\geq 2}}\int_\Omega \rho_0^{l_2}|\partial_t^{1-l_0}\partial_1^{l_1}\partial_2^{l_2}v|^2(t)\,\diff x
\leq M_0+Ct P(\sup_{0\leq s\leq t}E^{1/2}(s,v)), \quad l_0=0,1,2.
\end{aligned}
\end{equation}

\end{proposition}}

\begin{proof}

The estimates \(\eqref{EEQ-11}_1\)-\(\eqref{EEQ-11}_7\) correspond successively to the following estimates in each row respectively: 
{\small \begin{equation}\label{Pic-2}
\begin{aligned}
&\partial_t^2\partial_1\partial_2^2v\rightarrow\  \partial_t^2\partial_2^3v\rightarrow\\
&\partial_t\partial_1^3\partial_2^2v\rightarrow\   \partial_t\partial_1^2\partial_2^3v\rightarrow\ \ \partial_t\partial_1\partial_2^4v\rightarrow \partial_t\partial_2^5v\rightarrow\\
&\partial_1^5\partial_2^2v\rightarrow\ \quad \partial_1^4\partial_2^3v\rightarrow\ \quad \partial_1^3\partial_2^4v\rightarrow\ \ \ \partial_1^2\partial_2^5v\rightarrow\ \ \ \ \partial_1\partial_2^6v\rightarrow\ \partial_2^7v\rightarrow\\
&\partial_t^3\partial_2^2v\rightarrow\\
&\partial_t^2\partial_1^2\partial_2^2v\rightarrow\  \partial_t^2\partial_1\partial_2^3v\rightarrow\ \partial_t^2\partial_2^4v\rightarrow\\
&\partial_t\partial_1^4\partial_2^2v\rightarrow\  \partial_t\partial_1^3\partial_2^3v\rightarrow\ \ \partial_t\partial_1^2\partial_2^4v\rightarrow\ \partial_t\partial_1\partial_2^5v\rightarrow\ \partial_t\partial_2^6v\rightarrow\\
&\partial_1^6\partial_2^2v\rightarrow\ \quad \partial_1^5\partial_2^3v\rightarrow\ \quad \partial_1^4\partial_2^4v\rightarrow\ \quad \partial_1^3\partial_2^5v\rightarrow\ \ \  \partial_1^2\partial_2^6v
\rightarrow\ \partial_1\partial_2^7v\rightarrow\ \partial_2^8v
\end{aligned}
\end{equation}}

The strategy in showing Proposition \ref{pr:ell-6} is also an application of a bootstrapping procedure: One first uses Proposition \ref{pr:ell-5} to show the estimates in \(\eqref{EEQ-11}_1\) (\(l_0=0\)) along the arrow direction the first row of \eqref{Pic-2}, and then uses \(\eqref{EEQ-11}_1\) to show \(\eqref{EEQ-11}_2\) (\(l_0=1\)) along the arrow direction in the second row of \eqref{Pic-2}, and so on, until one reaches \(\eqref{EEQ-11}_7\). \\

In the following, we will give the proof of \(\eqref{EEQ-11}_4\), 
which depends on the sharp estimates on \(\|\partial_tDv\|_{L^\infty}\) 
and \(\|D^2v\|_{L^\infty}\); we will treat the endpoint case: the estimates on \(\partial_2^8v\) in \(\eqref{EEQ-11}_7\),  which is the most difficult case. \\

\noindent{\bf{Estimates on \(\partial_t^3\partial_2^2v\)}.}
As for \eqref{ELE-31}, one can show that 
\begin{equation}\label{ELE-200}
\begin{aligned}
\|\partial_t^3(\rho_0J^{-2}b^{kj}v^i,_{j}),_{k}\|_{L^2}^2
\leq M_0+Ct P(\sup_{0\leq s\leq t}E^{1/2}(s,v)).
\end{aligned}
\end{equation}

Note that
\begin{equation}\label{ELE-201}
\begin{aligned}
&(\rho_0J^{-2}b^{kj}\partial_t^3v^i,_{j}),_{k}
-\partial_t^3(\rho_0J^{-2}b^{kj}v^i,_{j}),_{k}\\
&=-\sum_{m=1}^3\binom{3}{m}\underbrace{\partial_k(\rho_0)\partial_t^{m}(J^{-2}b^{kj})\partial_t^{3-m}v^i,_{j}}_{I_{1_m}}\\
&\quad-\sum_{m=1}^3\binom{3}{m}\underbrace{\rho_0\partial_t^{m}\partial_k(J^{-2}b^{kj})\partial_t^{3-m}v^i,_{j}}_{I_{2_m}}\\
&\quad-\sum_{m=1}^3\binom{3}{m}\underbrace{\rho_0\partial_t^{m}(J^{-2}b^{kj})\partial_t^{3-m}\partial_kv^i,_{j}}_{I_{3_m}}.
\end{aligned}
\end{equation}

To deal with the RHS of \eqref{ELE-201}, more precise information on \(J^{-2}b^{kj}\) than that in Section \ref{sec:Jab} is needed. Elementary calculations show that
\begin{equation}\label{Q-bound-13}
\begin{aligned}
&|\partial_t(J^{-2}b^{kj})|+|\partial_tD(J^{-2}b^{kj})|+|\partial_t^2(J^{-2}b^{kj})|\lesssim  Q_7(t),\\
&|\partial_t^3(J^{-2}b^{kj})|\lesssim  Q_7(t)(|\partial_t^2Dv|+1),\\
&|\partial_t^{l_0}D(J^{-2}b^{kj})|\lesssim  Q_7(t)\sum_{m=1}^{l_0-1}(|\partial_t^mD^2v|+1),\quad l_0=2,3
\end{aligned}
\end{equation}
with some generic polynomial function 
\begin{equation*}
\begin{aligned}
Q_7(t):=Q_7(|Dv|, |D\eta|, |D^2\eta|, |\partial_tDv|, |D^2v|).
\end{aligned}
\end{equation*}

We claim that 
\begin{equation}\label{Q-bound-14}
\begin{aligned}
Q_7(t)\leq Q_{*}(t).
\end{aligned}
\end{equation}
It suffices to show
\begin{equation}\label{tDv-estimate-1}
\begin{aligned}
\|\partial_tDv\|_{L^\infty}^2+\|D^2v\|_{L^\infty}^2
\leq M_0+Ct P(\sup_{0\leq s\leq t}E^{1/2}(s,v)).
\end{aligned}
\end{equation}
\eqref{tDv-estimate-1}
can be shown as \eqref{Dv-estimate-2} by 
\eqref{ineq:weighted Sobolev-2} and \eqref{ineq:weighted Sobolev-3}.
Indeed, one has
\begin{equation}\label{tDv-estimate-proof}
\begin{aligned}
&\|\partial_tDv\|_{L^\infty}^2\lesssim \sum_{k=0}^2\int_\Omega \rho_0|\partial_tD^{k+1}v|^2\,\diff x\lesssim \cdots\\
&\lesssim \mathcal{R}+\int_\Omega \rho_0^5|\partial_tD^5v|^2\,\diff x
\leq M_0+Ct P(\sup_{0\leq s\leq t}E^{1/2}(s,v))
\end{aligned}
\end{equation}
and
\begin{equation*}
\begin{aligned}
&\|D^2v(t)\|_{L^\infty}^2\lesssim \sum_{k=0}^2\int_\Omega \rho_0|D^{k+2}v|^2\,\diff x\lesssim \cdots\\
&\lesssim \mathcal{R}+
\int_\Omega \rho_0^7|D^7v|^2\,\diff x
\leq M_0+Ct P(\sup_{0\leq s\leq t}E^{1/2}(s,v)).
\end{aligned}
\end{equation*}
where \(\eqref{EEQ-11}_2\) and \(\eqref{EEQ-11}_3\) have been used, respectively. 

According to \eqref{Q-bound-13}-\eqref{tDv-estimate-1}, it holds that
\begin{equation*}
\begin{aligned}
&\|I_{1_{l_0}}\|_{L^2}^2\leq Q_{*}^2\|\partial_t^{3-l_0}Dv\|_{L^2}^2,\quad l_0=1,2,\\
&\|I_{1_3}\|_{L^2}^2\leq Q_{*}^2(\|\partial_t^2Dv\|_{L^2}^2+1)\|Dv\|_{L^\infty}^2,\\
&\|I_{3_{l_0}}\|_{L^2}^2\leq Q_{*}^2\|\rho_0\partial_t^{3-l_0}D^2v\|_{L^2}^2,\quad l_0=1,2,\\
&\|I_{3_3}\|_{L^2}^2\leq Q_{*}^2(\|\partial_t^2Dv\|_{L^2}^2+1)\|D^2v\|_{L^\infty}^2
\end{aligned}
\end{equation*}
and
\begin{equation*}
\begin{aligned}
&\|I_{2_1}\|_{L^2}^2\leq Q_{*}^2\|\partial_t^2Dv\|_{L^2}^2,\\
&\|I_{2_2}\|_{L^2}^2\leq Q_{*}^2(\|\rho_0\partial_tD^2v\|_{L^2}^2+1)\|\partial_tDv\|_{L^\infty}^2,\\
&\|I_{2_3}\|_{L^2}^2\leq Q_{*}^2(\|\rho_0\partial_t^2D^2v\|_{L^2}^2+\|\rho_0\partial_tD^2v\|_{L^2}^2+1)\|Dv\|_{L^\infty}^2,
\end{aligned}
\end{equation*}
which lead to
\begin{equation}\label{ELE-202}
\begin{aligned}
\sum_{m=1}^3(\|I_{1_m}\|_{L^2}^2+\|I_{2_m}\|_{L^2}^2+\|I_{3_m}\|_{L^2}^2)
\leq M_0+Ct P(\sup_{0\leq s\leq t}E^{1/2}(s,v)). 
\end{aligned}
\end{equation}
Here one has used \(\eqref{EEQ-6}_3\) to estimate
\begin{equation*}
\begin{aligned}
\|\partial_t^2Dv\|_{L^2}^2\lesssim \|\rho_0\partial_t^2Dv\|_{L^2}^2+\|\rho_0\partial_t^2D^2v\|_{L^2}^2
\leq M_0+Ct P(\sup_{0\leq s\leq t}E^{1/2}(s,v)).
\end{aligned}
\end{equation*}

It follows from \eqref{ELE-200}, \eqref{ELE-201} and \eqref{ELE-202} that
\begin{equation}\label{ELE-203}
\begin{aligned}
\|(\rho_0J^{-2}b^{kj}\partial_t^3v^i,_{j}),_{k}\|_{L^2}^2
\leq M_0+Ct P(\sup_{0\leq s\leq t}E^{1/2}(s,v)). 
\end{aligned}
\end{equation}

Based on \eqref{ELE-203}, one can complete the remaining proof by following a similar argument for \(\partial_t\partial_2^2v\)
in Proposition \ref{pr:ell-3}. \\

\noindent{\bf{Estimates on \(\partial_2^8v\)}.} We will show
\begin{equation}\label{ELE-300}
\begin{aligned}
\|\rho_0^4\partial_2^8v\|_{L^2}^2
\leq M_0+Ct P(\sup_{0\leq s\leq t}E^{1/2}(s,v)).
\end{aligned}
\end{equation}
%To this end, we will carry out a procedure of peeling off the normal derivatives from \(\partial_2^6(\rho_0J^{-2}b^{kj}v^i,_{j}),_{k}\) step by step and compensating some proper weight in each step as for \eqref{EEQ-3}. 

We start from estimating \(\partial_2^6(\rho_0J^{-2}b^{kj}v^i,_{j}),_{k}\) with a weight \(\rho_0^2\) in \(L^2\). 
Note that
\begin{equation*}
\begin{aligned}
\partial_2^6(\rho_0\partial_tv^i)=\rho_0\partial_t\partial_2^6v+6\partial_2(\rho_0)\partial_t\partial_2^5v+\text{l.o.t.},
\end{aligned}
\end{equation*}
which, together with \eqref{ineq:weighted Sobolev-2} and \(\eqref{EEQ-11}_6\), implies  
\begin{equation}\label{ELE-300-add-1}
\begin{aligned}
\|\rho_0^2\partial_2^6(\rho_0\partial_tv^i)\|_{L^2}^2
&\lesssim (\|\rho_0^3\partial_t\partial_2^5v\|_{L^2}^2+\|\rho_0^3\partial_t\partial_2^6v\|_{L^2}^2)+\|\rho_0^3\partial_t\partial_2^6v\|_{L^2}^2+\mathcal{R}\\
&\leq M_0+Ct P(\sup_{0\leq s\leq t}E^{1/2}(s,v)).
\end{aligned}
\end{equation} 
Since 
\begin{equation*}
\begin{aligned}
\partial_2^6(\rho_0^2 J^{-2}a_i^k),_{k}=\partial_k(\rho_0^2)\partial_2^6(J^{-2}a_i^k)+\rho_0^2\partial_2^6(J^{-2}a_i^k),_{k}+\text{l.o.t.},
\end{aligned}
\end{equation*}
then it follows that
\begin{equation}\label{ELE-300-add-2}
\begin{aligned}
\|\rho_0^2\partial_2^6(\rho_0^2 J^{-2}a_i^k),_{k}\|_{L^2}^2
&\lesssim 
Q_{*}^2(\|\rho_0^3D^7\eta\|_{L^2}^2+\|\rho_0^4D^8\eta\|_{L^2}^2\\
&\quad+\|\rho_0^3g_6(\eta)\|_{L^2}^2+\|\rho_0^4g_7(\eta)\|_{L^2}^2)+\mathcal{R}\\
&\leq M_0+Ct P(\sup_{0\leq s\leq t}E^{1/2}(s,v)),
\end{aligned}
\end{equation}
where one has used \eqref{Weighted-2} and \eqref{Jab-bound-11}  to estimate
\begin{equation*}
\begin{aligned}
\|\rho_0^3g_6(\eta)\|_{L^2}^2\lesssim \|\rho_0D^5\eta\|_{L^2}^2\|\rho_0D^3\eta\|_{L^\infty}^2+\mathcal{R}
\end{aligned}
\end{equation*}
and
\begin{equation*}
\begin{aligned}
\|\rho_0^4g_7(\eta)\|_{L^2}^2&\lesssim \|\rho_0^2D^6\eta\|_{L^2}^2\|\rho_0D^3\eta\|_{L^\infty}^2+\|\rho_0D^5\eta\|_{L^2}^2\|\rho_0^2D^4\eta\|_{L^\infty}^2
\\
&\quad+\|D^4\eta\|_{L^2}^2\|\rho_0D^3\eta\|_{L^\infty}^4+\mathcal{R}.
\end{aligned}
\end{equation*}
Therefore, applying \(\partial_2^6\) to \eqref{eq:main-3}, and using \eqref{ELE-300-add-1}-\eqref{ELE-300-add-2}, one obtains
\begin{equation}\label{ELE-301}
\begin{aligned}
\|\rho_0^2\partial_2^6(\rho_0J^{-2}b^{kj}v^i,_{j}),_{k}\|_{L^2}^2
\leq M_0+Ct P(\sup_{0\leq s\leq t}E^{1/2}(s,v)).
\end{aligned}
\end{equation}

Next, we focus on \(\partial_2^6(\rho_0J^{-2}b^{22}v^i,_{2}),_{2}\), which will be estimated in \(L^2\) by using a higher-order weight \(\rho_0^3\). Notice that  
\begin{equation*}
\begin{aligned}
\partial_2^m(\rho_0J^{-2}b^{kj})&=\rho_0\partial_2^m(J^{-2}b^{kj})+m\partial_2(\rho_0)\partial_2^{m-1}(J^{-2}b^{kj})+\text{l.o.t.}\\
&=\underbrace{\rho_0\partial_2^m(J^{-2}b^{kj})}_{\text{the\ leading-order\ term}}+\text{s.l.t.}+\text{l.o.t.},
\end{aligned}
\end{equation*}
where \text{s.l.t.} denotes the same level terms that are as difficult as the leading-order term.
Then one may write
\begin{equation*}
\begin{aligned}
\partial_2^6(\rho_0J^{-2}b^{22}v^i,_{2}),_{2}
=\partial_2^6(\rho_0J^{-2}b^{kj}v^i,_{j}),_{k}-\sum_{(k,j)\neq (2,2)}\partial_2^6(\rho_0J^{-2}b^{kj}v^i,_{j}),_{k}
\end{aligned}
\end{equation*}
with
{\small\begin{equation*}
\begin{aligned}
\sum_{(k,j)\neq (2,2)}\partial_2^6(\rho_0J^{-2}b^{kj}v^i,_{j}),_{k}
&=2\sum_{m=0}^6\binom{6}{m}\underbrace{\rho_0\partial_2^m(J^{-2}b^{21})\partial_2^{7-m}v^i,_{1}}_{I_{{4}_m}}\\
&\quad+\sum_{(k,j)\neq (2,2)}\sum_{m=0}^6\binom{6}{m}\underbrace{\rho_0\partial_k\partial_2^m(J^{-2}b^{kj})\partial_2^{6-m}v^i,_{j}}_{I_{{5}_m}}\\
&\quad+\sum_{(k,j)\neq (2,2)}\sum_{m=0}^6\binom{6}{m}\underbrace{\partial_k(\rho_0)\partial_2^m(J^{-2}b^{kj})\partial_2^{6-m}v^i,_{j}}_{I_{{6}_m}}\\
&\quad+\text{s.l.t.}+\text{l.o.t.}.
\end{aligned}
\end{equation*}}
 We analyze only  the two endpoint cases \(m=0, 6\), in which one shall handle the highest-order derivatives of \(v\) and the nonlinear terms involving higher-order derivatives of \(\eta\) in \eqref{Jab-bound-11}. 
 
First, it is easy to get
\begin{equation*}
\begin{aligned}
&\|\rho_0^3I_{{4}_0}\|_{L^2}^2\lesssim Q_{*}^2\|\rho_0^4\partial_1D^7v\|_{L^2}^2,\
\|\rho_0^3I_{{5}_0}\|_{L^2}^2\lesssim Q_{*}^2\|\rho_0^4D^7v\|_{L^2}^2,\\
&\|\rho_0^3I_{{4}_6}\|_{L^2}^2\lesssim Q_{*}^2(\|\rho_0^3D^7\eta\|_{L^2}^2
+\|\rho_0^3g_6(\eta)\|_{L^2}^2)\|D^2v\|_{L^\infty}^2+\mathcal{R},\\
&\|\rho_0^3I_{{5}_6}\|_{L^2}^2\lesssim Q_{*}^2(\|\rho_0^4D^8\eta\|_{L^2}^2
+\|\rho_0^4g_7(\eta)\|_{L^2}^2)\|Dv\|_{L^\infty}^2+\mathcal{R},\\
&\|\rho_0^3I_{{6}_6}\|_{L^2}^2\lesssim Q_{*}^2(\|\rho_0^3D^7\eta\|_{L^2}^2
+\|\rho_0^3g_6(\eta)\|_{L^2}^2)\|Dv\|_{L^\infty}^2+\mathcal{R},
\end{aligned}
\end{equation*}
which satisfy the desired bound. 
Here one has used \(\eqref{EEQ-11}_7\) with \((l_1,l_2)=(1,7)\) in estimating \(I_{{4}_0}\), and \(\eqref{EEQ-11}_3\) in estimating \(I_{{5}_0}\). Then, it remains to consider \(I_{{6}_0}\). 
In fact, one has for \((k,j)= (1,2)\),
\begin{equation*}
\begin{aligned}
\|\rho_0^3I_{{6}_0}\|_{L^2}^2\lesssim Q_{*}^2\|\rho_0^4D^7v\|_{L^2}^2
\leq M_0+Ct P(\sup_{0\leq s\leq t}E^{1/2}(s,v)),
\end{aligned}
\end{equation*}
while for \((k,j)= (2,1)\),
\begin{equation*}
\begin{aligned}
\|\rho_0^3I_{{6}_0}\|_{L^2}^2\lesssim Q_{*}^2\|\rho_0^3\partial_1D^6v\|_{L^2}^2
\leq M_0+Ct P(\sup_{0\leq s\leq t}E^{1/2}(s,v)),
\end{aligned}
\end{equation*}
where \(\eqref{EEQ-11}_3\) has been used. 

Collecting all the cases and using \eqref{ELE-301} yield
\begin{equation}\label{ELE-302}
\begin{aligned}
\|\rho_0^3\partial_2^6(\rho_0J^{-2}b^{22}v^i,_{2}),_{2}\|_{L^2}^2
\leq M_0+Ct P(\sup_{0\leq s\leq t}E^{1/2}(s,v)).
\end{aligned}
\end{equation}

Then, to further peel off the normal derivatives from \(\partial_2^6(\rho_0J^{-2}b^{22}v^i,_{2}),_{2}\), one notices that
\begin{equation*}
\begin{aligned}
&\rho_0\partial_2^7(J^{-2}b^{22}v^i,_{2})+7\partial_2(\rho_0)\partial_2^6(J^{-2}b^{22}v^i,_{2})\\
&=\partial_2^6(\rho_0J^{-2}b^{22}v^i,_{2}),_{2}
-\sum_{m=2}^7\binom{7}{m}\partial_2^{m}(\rho_0)\partial_2^{7-m}(J^{-2}b^{22}v^i,_{2})
\end{aligned}
\end{equation*}
with 
\begin{equation*}
\begin{aligned}
\sum_{m=2}^7\binom{7}{m}\partial_2^{m}(\rho_0)\partial_2^{7-m}(J^{-2}b^{22}v^i,_{2})
=\binom{7}{2}\partial_2^2(\rho_0)\partial_2^{5}(J^{-2}b^{22}v^i,_{2})+\text{l.o.t.}.
\end{aligned}
\end{equation*}
It is easy to check by \(\eqref{EEQ-6}_5\) that 
\begin{equation*}
\begin{aligned}
\|\rho_0^3\partial_2^{5}(J^{-2}b^{22}v^i,_{2})\|_{L^2}^2
\leq M_0+CtP(\sup_{0\leq s\leq t}E^{1/2}(s,v)).
\end{aligned}
\end{equation*}
This together with \eqref{ELE-302} gives
\begin{equation}\label{ELE-305}
\begin{aligned}
&\|\rho_0^4\partial_2^7(J^{-2}b^{22}v^i,_{2})+7\partial_2(\rho_0)\rho_0^3\partial_2^6(J^{-2}b^{22}v^i,_{2})\|_{L^2}^2\\
&\leq M_0+CtP(\sup_{0\leq s\leq t}E^{1/2}(s,v)). 
\end{aligned}
\end{equation}

Now, we can further extract the normal derivatives to write
\begin{equation*}
\begin{aligned}
&\rho_0J^{-2}b^{22}\partial_2^7v^i,_{2}+7\partial_2(\rho_0)J^{-2}b^{22}\partial_2^6v^i,_{2}\\
&=\rho_0\partial_2^7(J^{-2}b^{22}v^i,_{2})+7\partial_2(\rho_0)\partial_2^6(J^{-2}b^{22}v^i,_{2})\\
&\quad-\sum_{m_1=1}^7\binom{7}{m_1}\underbrace{\rho_0\partial_2^{m_1}(J^{-2}b^{22})\partial_2^{7-m_1}v^i,_{2}}_{J_{m_1}}\\
&\quad-7\sum_{m_2=1}^6\binom{6}{m_2}\underbrace{\partial_2(\rho_0)\partial_2^{m_2}(J^{-2}b^{22})\partial_2^{6-m_2}v^i,_{2}}_{J_{m_2}}.
\end{aligned}
\end{equation*}
Note that \(J_{m_1}\) and  \(J_{m_2}\) can be handled similarly as \(I_{{4}_m}\), \(I_{{5}_m}\) and \(I_{{6}_m}\). 
This fact together with \eqref{ELE-305} leads to
\begin{equation}\label{ELE-306}
\begin{aligned}
\|\rho_0^4J^{-2}b^{22}\partial_2^7v^i,_{2}+7\partial_2(\rho_0)\rho_0^3J^{-2}b^{22}\partial_2^6v^i,_{2}\|_{L^2}^2
\leq M_0+CtP(\sup_{0\leq s\leq t}E^{1/2}(s,v)).
\end{aligned}
\end{equation}

Finally, one can integrate by parts and use \(\eqref{EEQ-11}_3\) and \eqref{ELE-306} to deduce
\begin{equation*}
\begin{aligned}
\|\rho_0^4J^{-2}b^{22}\partial_2^7v^i,_{2}\|_{L^2}^2
&\lesssim\|\rho_0^4J^{-2}b^{22}\partial_2^7v^i,_{2}+7\partial_2(\rho_0)\rho_0^3J^{-2}b^{22}\partial_2^6v^i,_{2}\|_{L^2}^2\\
&\quad+ Q_{*}^2
\int_\Omega\rho_0^7|\partial_2^7v|^2\,\diff x
\leq M_0+Ct P(\sup_{0\leq s\leq t}E^{1/2}(s,v)),
\end{aligned}
\end{equation*}
which, together with \eqref{J-bound} and \eqref{b-bound}, gives \eqref{ELE-300}.

\end{proof}

\bigskip

 Finally, we conclude from Propositions \ref{pr:ell-1}-\ref{pr:ell-6} that
\begin{equation}\label{APBEL}
\begin{aligned}
E_{\text{el}}(t,v)
\leq M_0+CtP(\sup_{0\leq s\leq t}E^{1/2}(s,v))\quad \mathrm{for\ all}\ t\in [0,T].
\end{aligned}
\end{equation}

%\section{An A Priori Bound}\label{A priori Bound}
It follows from \eqref{APBEN} and \eqref{APBEL} that
\begin{equation}\label{APB-1}
\begin{aligned}
E(t,v)
\leq M_0+CtP(\sup_{0\leq s\leq t}E^{1/2}(s,v))\quad \mathrm{for\ all}\ t\in [0,T],
\end{aligned}
\end{equation}
where \(P\) denotes a generic polynomial function of its arguments, and \(C\) is an absolutely constant  depending only on \(\|\rho_0\|_{H^7(\Omega)}\). 
\eqref{APB-1} implies for suitably small \(T>0\) that 
\begin{equation}\label{APB-2}
	\begin{aligned}
		\sup_{0\leq t\leq T}E(t,v)
		\leq 2M_0.
	\end{aligned}
\end{equation}

\section{The Linearized Problem}\label{The Linearized Problem}

For given \(T>0\), let \(\mathcal{X}_T\) be a Banach space defined by
\begin{equation*}
	\begin{aligned}
		\mathcal{X}_T=\{ v\in L^\infty([0,T]; H^4(\Omega)):\ \sup_{0\leq t\leq T}E(t,v)<\infty\}
	\end{aligned}
\end{equation*}
with the norm
\begin{equation*}
	\begin{aligned}
		\|v\|_{\mathcal{X}_T}^2=\sup_{0\leq t\leq T}E(t,v).
	\end{aligned}
\end{equation*}
For given \(M_1\), define \(\mathcal{C}_T(M_1)\) to be a closed, bounded, and
convex subset of \(\mathcal{X}_T\) given by
\begin{equation*}\label{solution space}
	\begin{aligned}
		\mathcal{C}_T(M_1)&=\{ v\in \mathcal{X}_T:\|v\|_{\mathcal{X}_T}^2\leq M_1\}.
	\end{aligned}
\end{equation*}

For any given \(\bar{v}\in\mathcal{C}_T(M_1)\), define
\begin{align*}
	\bar{\eta}(x,t)=x+\int_0^t\bar{v}(x,s)\,\diff s,
\end{align*}
\begin{align*}
	\bar{J}=\bar{\eta}^1,_{1}\bar{\eta}^2,_{2}-\bar{\eta}^1,_{2}\bar{\eta}^2,_{1},
\end{align*}
\begin{equation*}
	\bar{a}=
	\begin{bmatrix}
		\bar{\eta}^2,_{2}& -\bar{\eta}^1,_{2}\\
		-\bar{\eta}^2,_{1}&\bar{\eta}^1,_{1}
	\end{bmatrix},\quad \bar{b}^{kj}=\bar{a}_l^k\bar{a}_l^j.
\end{equation*}

As for \eqref{J-bound}, \(T>0\) can be chosen such that 
\begin{align}\label{eta-bound-2}
	9/10\leq \bar{J}(x,t)\leq 10/11\quad \text{for}\ (x,t)\in \Omega\times [0,T]
\end{align}
and
\begin{align}\label{a-bound-2}
	\bar{b}^{kj}(x,t)\xi_k\xi_j\geq |\xi|^2/5\quad \text{for}\ (x,t)\in \Omega\times [0,T]\ \text{and}\ \ \xi\in \R^2.
\end{align}
The choice of  \(M_1\) and \(T\) is given in Subsection \ref{The a priori assumption}.
We then consider the following
linearized problem for \(v\):
\begin{equation}\label{existence-3}
	\begin{cases}
		\rho_0\partial_tv^i+(\rho_0^2 \bar{J}^{-2}\bar{a}_i^k),_{k}=(\rho_0\bar{J}^{-2}\bar{b}^{kj}v^i,_{j}),_{k} &\quad \mbox{in}\ \Omega\times (0,T],\\
		v=u_0 &\quad \mbox{on}\ \Omega\times \{t=0\}.
	\end{cases}
\end{equation}

\section{Degenerate-Singular Elliptic Operators and Construction of An Orthogonal Basis of \(H_{\rho_0}^1(\Omega)\)}\label{Degenerate-Singular Elliptic Operators}

In this section, we will construct an orthogonal basis \(\{w_l\}_{l=1}^\infty\) of the weighted Hilbert space \(H_{\rho_0}^1(\Omega)\).

Consider the degenerate-singular elliptic operator
%\footnote{The much simpler degenerate elliptic operator \(\mathcal{L}_tw:=-(\rho_0\bar{b}^{kj}w,_{j}),_{k}\) does not work since one can not deduce the boundary condition \eqref{predetermined boundary condition} from the equation \eqref{eigenvalue problem} following the argument in Remark \ref{re:main-2}.}
\begin{align*}\label{degenerate-singular elliptic operator}
  \mathcal{L}w:=-\frac{\mathrm{div}(\rho_0Dw)}{\rho_0}+w \quad \text{in}\ \Omega.
\end{align*}
First, we study the solvability of the degenerate-singular elliptic equation
\begin{equation}\label{degenerate-singular elliptic equation}
\begin{aligned}
\mathcal{L}w=g \quad \mathrm{in}\ \Omega
\end{aligned}
\end{equation}
for some given \(g\) to be specified later. It should be pointed out that we do not prescribe any boundary conditions for \eqref{degenerate-singular elliptic equation}. 
To this end, we define a new measure \(\diff \mu=\rho_0\diff x\), and will find a unique weak solution to \eqref{degenerate-singular elliptic equation} in the Hilbert space \(H^1(\Omega, \diff \mu)\) (which is exactly \(H_{\rho_0}^1(\Omega)\)). Thus, we define 
a symmetric bilinear form and a linear functional on \(H^1(\Omega, \diff \mu)\), respectively, as 
 \begin{align*}
    B[w, \varphi]_\mu=\int_{\Omega}(DwD\varphi+ w \varphi)\,\diff \mu
  \end{align*}
and 
\begin{equation*}
\begin{aligned}
\langle g, \varphi\rangle_\mu=\int_{\Omega} g \varphi\,\diff \mu. 
\end{aligned}
\end{equation*}

Then the weak solutions to \eqref{degenerate-singular elliptic equation} can be defined as: 
\begin{definition}\label{Weak sense}
For any given \(g\in L^2(\Omega, \diff \mu)\), we call \(w\in H^1(\Omega, \diff \mu)\) a weak solution to \eqref{degenerate-singular elliptic equation}, 
if the following equality holds
\begin{align*}
  B[w,\varphi]_\mu=\langle g, \varphi\rangle_\mu\quad \mathrm{for\ all}\ \varphi\in H^1(\Omega, \diff \mu). 
\end{align*}

\end{definition}

The existence of such weak solutions is given as follows:

\begin{theorem}\label{Existence weak solution}
 For any given \(g\in L^2(\Omega, \diff \mu)\),
there exists a unique weak solution \(\tilde{w}\in H^1(\Omega, \diff \mu)\) to \eqref{degenerate-singular elliptic equation}. 
\end{theorem}

\begin{proof}
First,  \(B[\cdot, \cdot]_\mu\) is bounded and coercive on \(H^1(\Omega, \diff \mu)\). Indeed, 
by Cauchy's inequality, one has
\begin{equation*}
\begin{aligned}
|B[w, \varphi]_\mu|&\lesssim \bigg(\int_{\Omega} \rho_0|Dw|^2\, d x\bigg)^{1/2} \bigg(\int_{\Omega} \rho_0|D\varphi|^2\,\diff x\bigg)^{1/2}\\
&\quad+\bigg(\int_{\Omega} \rho_0|w|^2\,\diff x\bigg)^{1/2} \bigg(\int_{\Omega} \rho_0|\varphi|^2\,\diff x\bigg)^{1/2}\\
&\leq C\|w\|_{H^1(\Omega, \diff \mu)}\|\varphi\|_{H^1(\Omega, \diff \mu)}
\end{aligned}
\end{equation*}
and
\begin{equation*}
\begin{aligned}
|B[w, w]_\mu|\gtrsim \int_{\Omega} \rho_0|Dw|^2\, d x+\int_{\Omega} \rho_0|w|^2\,\diff x\geq C\|w\|_{H^1(\Omega, \diff \mu)}^2.
\end{aligned}
\end{equation*}

Next, \(g: H^1(\Omega, \diff \mu)\rightarrow \mathbb{R}\) is a bounded linear functional on \(H^1(\Omega, \diff \mu)\), which follows from 
\begin{equation*}
\begin{aligned}
|\langle g, \varphi\rangle_\mu|\leq \bigg(\int_{\Omega} \rho_0g^2\, d x\bigg)^{1/2} \bigg(\int_{\Omega} \rho_0\varphi^2\,\diff x\bigg)^{1/2}
\leq \|g\|_{L^2(\Omega, \diff \mu)}\|\varphi\|_{H^1(\Omega, \diff \mu)}. 
\end{aligned}
\end{equation*}

Finally, the result is a direct consequence of the Lax-Milgram theorem.

\end{proof}

Now, we can construct an orthogonal basis of \(H_{\rho_0}^1(\Omega)\).

\begin{theorem}\label{Hilbert basis construction}
%\noindent{\bf{(i)}} Each eigenvalue of \(\mathcal{L}\) is real;\\
\noindent{\bf{(i)}} The spectrum \(\Sigma\) of \(\mathcal{L}\) can be written as
\begin{align*}
\Sigma=\{\sigma_l\}_{l=1}^\infty,
\end{align*}
where
\begin{align*}
0<\sigma_1\leq \sigma_2\leq \sigma_3\leq \cdots
\end{align*}
and
\begin{align*}
\sigma_l\rightarrow  \infty \quad   \mathrm{as}\  l\rightarrow \infty,
\end{align*}
where each eigenvalue with multiplicity \(m\) is repeated \(m\) times.\\
\noindent{\bf{(ii)}} There exists an orthonormal basis \(\{w_l\}_{l=1}^\infty\) of \(L^2(\Omega, \diff \mu)\) (\(\{w_l\}_{l=1}^\infty\) is also an orthogonal basis of \(H^1(\Omega, \diff \mu)\)), where \(w_l\) is an eigenfunction corresponding to the eigenvalue \(\sigma_l\), namely
\begin{equation}\label{eigenvalue problem}
\begin{aligned}
\mathcal{L}w_l=\sigma_lw_l\quad  \mathrm{in}\  \Omega.
\end{aligned}
\end{equation}

\end{theorem}

\begin{proof} By Theorem \ref{Existence weak solution}, the unique weak solution to \eqref{degenerate-singular elliptic equation} can be written as 
\(w_g=\mathcal{L}^{-1} g\).
Define \(S:=\mathcal{L}^{-1}\).

Next, we will show that \(S\) is a linear, bounded, compact, symmetric and positive operator on \(L^2(\Omega, \diff \mu)\). The linearity of \(S\) is obvious.  
To show the boundedness of \(S\), one notices first, 
\begin{equation*}
\begin{aligned}
|B[Sg, Sg]_\mu|\geq C\|Sg\|_{H^1(\Omega, \diff \mu)}^2,
\end{aligned}
\end{equation*}
on the other hand, 
\begin{equation*}
\begin{aligned}
|B[Sg, Sg]_\mu|=|\langle g, Sg\rangle_\mu|\leq \|g\|_{L^2(\Omega, \diff \mu)}\|Sg\|_{H^1(\Omega, \diff \mu)}.
\end{aligned}
\end{equation*}
Hence
\begin{align*}
  \|Sg\|_{L^2(\Omega, \diff \mu)}\leq \|Sg\|_{H^1(\Omega, \diff \mu)}\leq  C\|g\|_{L^2(\Omega, \diff \mu)}.
\end{align*}

The compactness of \(S\) follows from the fact that \(L^2(\Omega, \diff \mu)=L_{\rho_0}^2(\Omega)\) is compactly embedded in \(H^1(\Omega, \diff \mu)=H_{\rho_0}^1(\Omega)\) due to \eqref{eq:intro-3} (see \cite{MR0983485}). 

For \(g_i\in L^2(\Omega, \diff \mu)\) \((i=1,2)\), then \(Sg_i=w_i\in H^1(\Omega, \diff \mu)\) 
solve
\begin{eqnarray*}
\begin{aligned}
\mathcal{L}w_i=g_i\quad \mathrm{in}\ \Omega
\end{aligned}
\end{eqnarray*}
in the weak sense. Thus one has
\begin{align*}
  \langle S g_1, g_2 \rangle_\mu=\langle w_1, g_2 \rangle_\mu=\langle g_2, w_1 \rangle_\mu=B[w_2, w_1]_\mu
\end{align*}
and
\begin{align*}
  \langle g_1, S g_2 \rangle_\mu=\langle g_1, w_2 \rangle_\mu=B[w_1, w_2]_\mu=B[w_2, w_1]_\mu.
\end{align*}
Hence the symmetry of \(S\) follows.

\(S\) is positive on \(L^2(\Omega, \diff \mu)\). Indeed, for any \(g\in L^2(\Omega, \diff \mu)\), 
\begin{align*}
  \langle S g, g \rangle_\mu=\langle w, g \rangle_\mu=\langle g, w \rangle_\mu=B[w, w]_\mu\geq 0. 
\end{align*}
  
Finally, the results in \noindent{\bf{(i)}} and \noindent{\bf{(ii)}} follow directly from the theory of compact operators (see \cite{MR2597943}).

\end{proof}

The eigenfunctions \(\{w_l\}_{l=1}^\infty\) admit the following regularity estimates.

\begin{theorem}\label{regularity of eigenvalue}
It holds that
\begin{equation}\label{weighted eigenfunction regularity}
	\begin{aligned}
		&\sum_{l_1=0}^7\big\|\sqrt{\rho_0}\partial_1^{l_1}Dw_l\big\|_{L^2(\Omega)}+\sum_{\substack{l_1+l_2\leq 8\\  l_1\geq 0,\ l_2\geq 2}}\big\|\sqrt{\rho_0^{l_2}}\partial_1^{l_1}\partial_2^{l_2}w_l\big\|_{L^2(\Omega)}\\
		&\leq C\|\sqrt{\rho_0}w_l\|_{L^2(\Omega)}, \quad l=1,2,\dots.
	\end{aligned}
\end{equation}
Hence
\begin{align}\label{eigenfunction regularity}
  w_l\in H^4(\Omega),\quad l=1,2,\dots. 
\end{align}

\end{theorem}

\begin{proof} By Definition \ref{Weak sense}, the weak form of \eqref{eigenvalue problem} becomes
\begin{equation}\label{weak form}
\begin{aligned}
\int_\Omega \rho_0Dw_lD\varphi\,\diff x=(\sigma_l-1)\int_\Omega \rho_0 w_l  \varphi\,\diff x \quad \mathrm{for\ all}\ \varphi\in H^1(\Omega, \diff \mu).
\end{aligned}
\end{equation}

Noticing that the leading term (i.e., the LHS) of \eqref{weak form} enjoys a similar formulation with the elliptic part of \eqref{eq:main-3} in the weak sense, one can follow the argument of the elliptic estimates in Section \ref{Elliptic Estimates} to deduce \eqref{weighted eigenfunction regularity}.
More precisely, defining the difference quotient 
\begin{align}\label{DQ}
  D_i^\tau w(x)=\frac{w(x+\tau {\bf e}_i)-w(x)}{\tau},
\end{align}
where \({\bf e}_i\)\ \((i=1,2)\) stand for the \(i\)-th standard coordinate vector, one can deduce \eqref{weighted eigenfunction regularity} in the following manner: \\

\noindent \(\mathrm{(\bf{I})}\) choosing \(\varphi=w_l\) in \eqref{weak form} easily yields 
\begin{equation}\label{DS-1}
\begin{aligned}
\|\sqrt{\rho_0}Dw_l\|_{L^2(\Omega)}\leq C \|\sqrt{\rho_0}w_l\|_{L^2(\Omega)}.
\end{aligned}
\end{equation}

\noindent \(\mathrm{(\bf{II})}\) first choosing \(\varphi=-D_1^{-\tau} D_1^\tau w_l\) in \eqref{weak form} and using integration by parts and \eqref{DS-1} lead easily to 
\begin{equation}\label{DS-2}
\begin{aligned}
\|\sqrt{\rho_0}\partial_1Dw_l\|_{L^2(\Omega)}\leq C\|\sqrt{\rho_0}w_l\|_{L^2(\Omega)},
\end{aligned}
\end{equation}
and then repeating the argument in Subsection \ref{LOP}, and using \eqref{DS-1} and \eqref{DS-2}, one can get 
\begin{equation}\label{DS-3}
\begin{aligned}
\|\rho_0\partial_2^2w_l\|_{L^2(\Omega)}\leq C\|\sqrt{\rho_0}w_l\|_{L^2(\Omega)}.
\end{aligned}
\end{equation}

\noindent \(\mathrm{(\bf{III})}\) One can show the remaining terms on the LHS of \eqref{weighted eigenfunction regularity} have the  bound \(C\|\sqrt{\rho_0}w_l\|_{L^2}\) by choosing appropriate test functions \(\tilde{w}\) and following the arguments in Subsections \ref{IP-I}-\ref{HOP}.

\eqref{eigenfunction regularity} is an easy consequence of \eqref{weighted eigenfunction regularity} and \eqref{ineq:weighted Sobolev-2}.

\end{proof}

\section{Existence of Weak Solutions to the Linearized Problem}\label{weak solution}

 Let \(\langle\cdot, \cdot\rangle\) be the  pairing of \(H^{-1}(\Omega)\) and \(H^1(\Omega)\), and \((\cdot,\cdot)\) stand for the inner product of \(L^2(\Omega)\). In order to obtain classical solutions to \eqref{existence-3}, we start with its weak solutions. 
\begin{definition}[Weak Solution]\label{Weak Solution} A function \(v\), satisfying
	\begin{equation*}
		\begin{aligned}
		\sqrt{\rho_0}Dv\in L^2([0,T];L^2(\Omega))\quad {\rm{and}}\quad \rho_0 v_t\in L^2([0,T];H^{-1}(\Omega)),
		\end{aligned}
	\end{equation*}
	is said to be a weak solution to \eqref{existence-3} if\\	
	{\rm{(a)}}
	\begin{equation*}
		\begin{aligned}
			\big\langle\rho_0v_t^i, \phi\big\rangle+\big(\rho_0\bar{J}^{-2}\bar{b}^{kj}v^i,_{j}, \phi,_{k}\big)=\big(\rho_0^2 \bar{J}^{-2}\bar{a}_i^k, \phi,_{k}\big)
		\end{aligned}
	\end{equation*}	
	for each \(\phi\in H^1(\Omega)\) and a.e. \(0< t\leq T\), and\\
	{\rm{(b)}} \(\|\rho_0v(t,\cdot)-\rho_0v(0, \cdot)\|_{L^2(\Omega)}\to 0\ \text{as}\ t\to 0^+\), and $v(0,\cdot)=u_0(\cdot) \ \mbox{a.e.\ on}\ \Omega$. 
\end{definition}

\subsection{Galerkin's scheme}
We first use the Galerkin's scheme (see \cite{MR2597943}) to construct approximate solutions to \eqref{existence-3}. 

Let \(\{e_n\}_{n=1}^\infty\) be the eigenfunctions constructed in Theorem \ref{Hilbert basis construction}, which form an orthonormal basis of \(L_{\rho_0}^2(\Omega)\) and an orthogonal basis of \(H_{\rho_0}^1(\Omega)\). 
Given a positive integer \(n\), we set
\begin{equation}\label{existence-4}
	\begin{aligned}
		(X^i)^n(t,x)=\sum_{m_1=1}^n(\lambda_{m_1}^i)^n(t)e_{m_1}(x),
	\end{aligned}
\end{equation}
in which the coefficients \((\lambda_{m_1}^i)^n(t)\) are chosen such that
\begin{eqnarray}\label{existence-5}
	\begin{cases}
		\big(\rho_0\partial_t(X^i)^n,e_{m_2}\big)
		+\big(\rho_0\bar{J}^{-2}\bar{b}^{kj}(X^i)^n,_{j},(e_{m_2}),_{k}\big)\\
		=\big(\rho_0^2 \bar{J}^{-2}\bar{a}_i^k,(e_{m_2}),_{k}\big)
		&\qquad \mbox{in}\ (0,T],\\ \\
		(\lambda_{m_1}^i)^n=(u_0^i,e_{m_1}) &\qquad \mbox{on}\ \{t=0\},
	\end{cases}
\end{eqnarray}
where \(m_1, m_2=1,2,...,n\).
Inserting \eqref{existence-4} into \eqref{existence-5} leads to
\begin{equation}\label{existence-6}
	\left\{
	\begin{aligned}
		&[(\lambda_{m_2}^i)^n(t)\big]_t
		+\sum_{m_1=1}^n\int_\Omega\rho_0\bar{J}^{-2}\bar{b}^{kj}(e_{m_1}),_{j}(e_{m_2}),_{k}\,\diff x\cdot  (\lambda_{m_1}^i)^n(t)\\
		&=\int_\Omega \rho_0^2 \bar{J}^{-2}\bar{a}_i^k(e_{m_2}),_{k}\,\diff x	\qquad\qquad\qquad\qquad\qquad\qquad\ \mbox{in}\ (0,T],\\ 
		&(\lambda_{m_1}^i)^n=(u_0^i,e_{m_1}) \qquad\qquad\qquad\qquad\qquad\qquad\qquad\quad \mbox{on}\ \{t=0\},
	\end{aligned}
	\right.
\end{equation}
where \(m_1, m_2=1,2,...,n\).

 It follows from \(\bar{v}\in \mathcal{C}_T(M_1)\), \eqref{eta-bound-2}, \eqref{a-bound-2} and \eqref{weighted eigenfunction regularity} that 
\begin{align*} 
	\int_\Omega\rho_0\bar{J}^{-2}\bar{b}^{kj}(e_{m_1}),_{j}(e_{m_2}),_{k}\,\diff x\quad \text{and}\quad \int_\Omega \rho_0^2 \bar{J}^{-2}\bar{a}_i^k(e_{m_2}),_{k}\,\diff x
\end{align*}
are Lipschitz continuous for \(t\in[0,T]\). By the standard theory for ordinary differential equations, one can find solutions \(\lambda_{m_1}^n(t)\) \ \((m_1=1,...,n)\) to \eqref{existence-6}, thus there exist approximate solutions \(X^n(t,x)\)\ \((n=1,2,...)\) to \eqref{existence-5}.\\

\subsection{Uniform estimates}

We next show that \(\{X^n\}_{n=1}^\infty\)  satisfy some uniform estimates in \(n\geq 1\).

\begin{lemma}\label{uniform estimates} It holds uniformly for \(n\geq 1\) that
\begin{equation}\label{first uniform estimates}
	\begin{aligned}
		&\underset{t\in[0,T]}{\sup}\|\sqrt{\rho_0}X^n(t)\|_{L^2(\Omega)}^2
		+\|\sqrt{\rho_0}DX^n\|_{L^2([0,T];L^2(\Omega))}^2\\
&+\|\rho_0 \partial_tX^n\|_{L^2([0,T];H^{-1}(\Omega))}^2
		\leq C\|\sqrt{\rho_0}u_0\|_{L^2(\Omega)}^2+CT.
	\end{aligned}
\end{equation}
	
\end{lemma}

\begin{proof} Multiplying \(\eqref{existence-5}_1\) by \((\lambda_{m_2}^i)^n\), summing it over \(m_2=1,2,...,n\), 
integrating it on \([0,T_n]\) and integrating by parts yield
\begin{equation*}\label{existence-7}
	\begin{aligned}
		&\frac{1}{2}\int_\Omega \rho_0|X^n|^2\,\diff x
		+\int_0^{T_n}\int_\Omega\rho_0\bar{J}^{-2}\bar{b}^{kj}(X^i)^n,_{j}(X^i)^n,_{k}\,\diff x\diff s\\
		&=\frac{1}{2}\int_\Omega \rho_0|X^n|^2(x,0)\,\diff x
		+\int_0^{T_n}\int_\Omega\rho_0^2 \bar{J}^{-2}\bar{a}_i^k(X^i)^n,_{k}\,\diff x\diff s.
	\end{aligned}
\end{equation*}
\eqref{eta-bound-2} and Cauchy's inequality imply
\begin{equation*}\label{existence-8}
	\begin{aligned}
		\int_0^{T_n}\int_\Omega\rho_0\bar{J}^{-2}\bar{b}^{kj}(X^i)^n,_{j}(X^i)^n,_{k}\,\diff x\diff s
		\geq \frac{2}{11} \int_0^{T_n}\int_\Omega\rho_0|DX^n|^2\,\diff x\diff s
	\end{aligned}
\end{equation*}
and
\begin{equation*}\label{existence-9}
	\begin{aligned}
		\bigg|\int_0^{T_n}\int_\Omega\rho_0^2 \bar{J}^{-2}\bar{a}_i^k(X^i)^n,_{k}\,\diff x\diff s\bigg|
		\leq CT_n+\frac{1}{100}\int_0^{T_n}\int_\Omega\rho_0|DX^n|^2\,\diff x\diff s.
	\end{aligned}
\end{equation*}
Hence
\begin{equation}\label{existence-10}
	\begin{aligned}
		\int_\Omega \rho_0|X^n|^2\,\diff x
		+\int_0^{T_n}\int_\Omega\rho_0|DX^n|^2\,\diff x\diff s
		\leq C\|\sqrt{\rho_0}u_0\|_{L^2(\Omega)}^2+CT_n.
	\end{aligned}
\end{equation}

Fix any \(\phi\in H^1(\Omega)\) with \(\|\phi\|_{H^1(\Omega)}\leq 1\). Since \(H^1(\Omega)\subset H_{\rho_0}^1(\Omega)\), one may decompose \(\phi=\phi_1+\phi_2\) in \(H_{\rho_0}^1(\Omega)\) satisfying 
\begin{equation*}
	\begin{aligned}
\phi_1\in \text{span}\{e_m\}_{m=1}^n\quad {\rm{and}}\quad (\phi_2,\rho_0e_m)=0\ (m=1,...,n).
	\end{aligned}
\end{equation*}
Recalling that the functions \(\{e_m\}_{m=1}^\infty\) are orthogonal in \(H_{\rho_0}^1(\Omega)\), one has
\begin{equation*}
	\begin{aligned}
\|\phi_1\|_{H_{\rho_0}^1(\Omega)}\leq \|\phi\|_{H_{\rho_0}^1(\Omega)}\leq \|\phi\|_{H^1(\Omega)}\leq 1.
	\end{aligned}
\end{equation*}
It follows from \(\eqref{existence-5}_1\)  that 
\begin{equation}\label{existence-11}
	\begin{aligned}
\big(\rho_0\partial_t(X^i)^n, \phi_1\big)+\big(\rho_0\bar{J}^{-2}\bar{b}^{kj}(X^i)^n,_{j}, (\phi_1),_{k}\big)=\big(\rho_0^2 \bar{J}^{-2}\bar{a}_i^k, (\phi_1),_{k}\big),
	\end{aligned}
\end{equation}
for a.e. \(0\leq t\leq T\).
Hence \eqref{existence-11} yields 
\begin{equation*}
	\begin{aligned}
\big\langle\rho_0\partial_t(X^i)^n, \phi\big\rangle
&=\big(\rho_0\partial_t(X^i)^n, \phi\big)=\big(\rho_0\partial_t(X^i)^n, \phi_1\big)\\
&=\big(\rho_0^2 \bar{J}^{-2}\bar{a}_i^k, (\phi_1),_{k}\big)-\big(\rho_0\bar{J}^{-2}\bar{b}^{kj}(X^i)^n,_{j}, (\phi_1),_{k}\big),
	\end{aligned}
\end{equation*}
which furthermore implies 
\begin{equation*}
	\begin{aligned}
		|\big\langle\rho_0\partial_t(X^i)^n, \phi\big\rangle|\leq C(1+\|\sqrt{\rho_0}DX^n\|_{L^2})\|\phi_1\|_{H_{\rho_0}^1(\Omega)}
		\leq C(1+\|\sqrt{\rho_0}DX^n\|_{L^2}).
	\end{aligned}
\end{equation*}
This together with \(\|\phi\|_{H^1(\Omega)}\leq 1\) results in 
\begin{equation*}
	\begin{aligned}
	  \|\rho_0\partial_tX^n\|_{H^{-1}(\Omega)}\leq C(1+\|\sqrt{\rho_0}DX^n\|_{L^2}),
	\end{aligned}
\end{equation*}
and therefore 
\begin{equation}\label{existence-12}
	\begin{aligned}
		\int_0^{T_n}\|\rho_0\partial_tX^n\|_{H^{-1}(\Omega)}^2\,\diff t&\leq C\int_0^{T_n}\int_\Omega\rho_0|DX^n|^2\,\diff x\,\diff s+CT_n\\
		&\leq C\|\sqrt{\rho_0}u_0\|_{L^2(\Omega)}^2+CT_n,
	\end{aligned}
\end{equation}
due to \eqref{existence-10}.

It follows from \eqref{existence-10} and \eqref{existence-12} that
 \begin{equation}\label{existence-13}
	\begin{aligned}
		&\underset{t\in[0,T_n]}{\sup}\|\sqrt{\rho_0}X^n(t)\|_{L^2(\Omega)}^2
		+\|\sqrt{\rho_0}DX^n\|_{L^2([0,T_n];L^2(\Omega))}^2\\
&+\|\rho_0 \partial_tX^n\|_{L^2([0,T_n];H^{-1}(\Omega))}^2
		\leq C\|\sqrt{\rho_0}u_0\|_{L^2(\Omega)}^2+CT_n.
	\end{aligned}
\end{equation}
Note that \eqref{eta-bound-2} holds on \(\Omega\times[0,T]\), hence \(T_n\) can reach \(T\).  Consequently \eqref{first uniform estimates} follows from \eqref{existence-13}.

\end{proof}

\subsection{Existence and uniqueness of a weak solution}

Finally, we show the existence of a weak solution to \eqref{existence-3}. 

\begin{lemma}\label{Existence and uniqueness of a weak solution} There exists a unique weak solution \(v\) to \eqref{existence-3} with
	\begin{equation*}
		\begin{aligned}
			&\sqrt{\rho_0}v\in L^\infty([0,T],L^2(\Omega)),\ \sqrt{\rho_0}Dv\in L^2([0,T],L^2(\Omega)),\\
			&\rho_0\partial_tv\in L^2([0,T];H^{-1}(\Omega)).
		\end{aligned}
	\end{equation*}	
	Moreover, \(v\) satisfies the following estimate:	
	\begin{equation}\label{solution bound}
		\begin{aligned}
			&\underset{t\in[0,T]}{\sup}\|\sqrt{\rho_0}v(t)\|_{L^2(\Omega)}^2
			+\|\sqrt{\rho_0}Dv\|_{L^2([0,T];L^2(\Omega))}^2\\
			&+\|\rho_0\partial_tv\|_{L^2([0,T];H^{-1}(\Omega))}^2
			\leq C\|\sqrt{\rho_0}u_0\|_{L^2(\Omega)}^2+C(1+T).
		\end{aligned}
	\end{equation}	
	
\end{lemma}

\begin{proof}
It follows from  Lemma \ref{uniform estimates} that
\begin{equation*}
	\begin{aligned}
	\|\sqrt{\rho_0}DX^n\|_{L^2([0,T];L^2(\Omega))}\quad {\rm{and}}\quad \|\rho_0 \partial_tX^n\|_{L^2([0,T];H^{-1}(\Omega))}
	\end{aligned}
\end{equation*}	
are uniformly bounded in \(n\geq 1\).  So there exist a subsequence of \(\{X^n\}_{n=1}^\infty\) (which is still denoted by \(\{X^n\}_{n=1}^\infty\) for convenience) and a function \(v\) satisfying \(\sqrt{\rho_0}Dv \in\ L^2([0,T];L^2(\Omega))\) and \(\rho_0\partial_tv \in\ L^2([0,T];H^{-1}(\Omega))\) such that as \(n\to\infty\)
\begin{equation}\label{existence-13-add}
	\begin{cases}
		\sqrt{\rho_0}DX^n\rightharpoonup \sqrt{\rho_0}Dv &\quad \text{in}\ L^2([0,T];L^2(\Omega)),\\
		\rho_0\partial_tX^n\rightharpoonup \rho_0\partial_tv &\quad \text{in}\ L^2([0,T];H^{-1}(\Omega)).	
	\end{cases}
\end{equation}	
Then, \eqref{solution bound} follows easily from \eqref{first uniform estimates} by the lower semi-continuity of norm.

We claim that \(v\) is a weak solution to \eqref{existence-3}. Fix any positive integer \(m\geq 1\),  choose a function  \(\Phi\in C^1([0,T]; H^1(\Omega))\)
of the form 
	\begin{equation}\label{existence-14}
	\begin{aligned}
		\Phi=\sum_{m_0=1}^m\mu_{m_0}(t)e_{m_0}(x),
	\end{aligned}
\end{equation}	
where \(\{\mu_{m_0}(t)\}_{m_0=1}^m\) are any given smooth functions. Choosing \(n\geq m\), multiplying \(\eqref{existence-5}_1\) by \(\mu_{m_0}(t)\), summing up for \(m_0=1,...,m\), and integrating with respect to \(t\) over \([0,T]\), we get
\begin{equation}\label{existence-15}
	\begin{aligned}
		\int_0^T\big\langle\rho_0\partial_t(X^i)^n, \Phi\big\rangle+\big(\rho_0\bar{J}^{-2}\bar{b}^{kj}(X^i)^n,_{j}, \Phi,_{k}\big)\,\diff t=\int_0^T\big(\rho_0^2 \bar{J}^{-2}\bar{a}_i^k, \Phi,_{k}\big)\,\diff t.
	\end{aligned}
\end{equation}
Recalling \eqref{existence-13-add},  one takes the limit \(n\rightarrow\infty\) in  \eqref{existence-15} to find
\begin{equation}\label{existence-16}
	\begin{aligned}
		\int_0^T\big\langle\rho_0\partial_tv^i, \Phi\big\rangle+\big(\rho_0\bar{J}^{-2}\bar{b}^{kj}v^i,_{j}, \Phi,_{k}\big)\,\diff t=\int_0^T\big(\rho_0^2 \bar{J}^{-2}\bar{a}_i^k, \Phi,_{k}\big)\,\diff t.
	\end{aligned}
\end{equation}

Since functions of the form \eqref{existence-14} are dense in \(L^2([0,T]; H^1(\Omega))\), 
\eqref{existence-16} holds for all  \(\Phi\in L^2([0,T]; H^1(\Omega))\).	
In particular, it holds that
\begin{equation}\label{existence-weak solution}
	\begin{aligned}
\big\langle\rho_0\partial_tv^i, \phi\big\rangle+\big(\rho_0\bar{J}^{-2}\bar{b}^{kj}v^i,_{j}, \phi,_{k}\big)=\big(\rho_0^2 \bar{J}^{-2}\bar{a}_i^k, \phi,_{k}\big)
	\end{aligned}
\end{equation}
for each \(\phi\in H^1(\Omega)\) and a.e. \(0< t\leq T\).  

By Definition \ref{Weak Solution}, it remains to check that 
\begin{equation}\label{initial-add-1}
	\begin{aligned}
\|\rho_0v(t,\cdot)-\rho_0v(0, \cdot)\|_{L^2(\Omega)}\to 0\quad \text{as}\ t\to 0^+
	\end{aligned}
\end{equation}
and
\begin{equation}\label{initial-add-2}
	\begin{aligned}
v(0):=v(0, \cdot)=u_0(0, \cdot) \quad \text{a.e.\ in}\ \Omega.
	\end{aligned}
\end{equation}
First note that
\begin{equation*}
		\begin{aligned}
			&\|\rho_0v\|_{L^2([0,T];H^1(\Omega))}^2\\
&\lesssim \|\rho_0v\|_{L^2([0,T];L^2(\Omega))}^2+\|\rho_0Dv\|_{L^2([0,T];L^2(\Omega))}^2
                                                +\|D(\rho_0) v\|_{L^2([0,T];L^2(\Omega))}^2\\
                                 &\lesssim \|\rho_0v\|_{L^2([0,T];L^2(\Omega))}^2+\|\rho_0Dv\|_{L^2([0,T];L^2(\Omega))}^2\\
                                 &\lesssim \|\rho_0^{1/2}v\|_{L^\infty([0,T];L^2(\Omega))}^2+\|\rho_0^{1/2}Dv\|_{L^2([0,T];L^2(\Omega))}^2,
		\end{aligned}
	\end{equation*}	
which, together with \eqref{solution bound}, implies that \(\rho_0v\in L^2([0,T]; H^1(\Omega))\).
Recalling \(\rho_0\partial_tv\in L^2([0,T];H^{-1}(\Omega))\), so one has 
\begin{equation}\label{initial-add-3}
	\begin{aligned} 
\rho_0v\in C([0,T]; L^2(\Omega)).
	\end{aligned}
\end{equation}
Thus \eqref{initial-add-1} follows. 
Then one may deduce from \eqref{existence-16} and \eqref{initial-add-3} that  
\begin{equation}\label{existence-17}
	\begin{aligned}
		&\int_0^T-\big\langle\partial_t\Phi, \rho_0v^i\big\rangle+\big(\rho_0\bar{J}^{-2}\bar{b}^{kj}v^i,_{j}, \Phi,_{k}\big)\,\diff t\\
&=\int_0^T\big(\rho_0^2 \bar{J}^{-2}\bar{a}_i^k, \Phi,_{k}\big)\,\diff t+(\rho_0v(0), \Phi(0))
	\end{aligned}
\end{equation}
for each \(\Phi\in C^1([0,T]; H^1(\Omega))\) with \(\Phi(T)=0\). It follows from \eqref{existence-15} that 
\begin{equation}\label{existence-18}
	\begin{aligned}
		&\int_0^T-\big\langle\partial_t\Phi, \rho_0(X^i)^n\big\rangle+\big(\rho_0\bar{J}^{-2}\bar{b}^{kj}(X^i)^n,_{j}, \Phi,_{k}\big)\,\diff t\\
&=\int_0^T\big(\rho_0^2 \bar{J}^{-2}\bar{a}_i^k, \Phi,_{k}\big)\,\diff t+(\rho_0(X^i)^n(0), \Phi(0)).
	\end{aligned}
\end{equation}
Passing limits in \(n\to \infty\) in \eqref{existence-18} gives 
\begin{equation}\label{existence-19}
	\begin{aligned}
		&\int_0^T-\big\langle\partial_t\Phi, \rho_0v^i\big\rangle+\big(\rho_0\bar{J}^{-2}\bar{b}^{kj}v^i,_{j}, \Phi,_{k}\big)\,\diff t\\
&=\int_0^T\big(\rho_0^2 \bar{J}^{-2}\bar{a}_i^k, \Phi,_{k}\big)\,\diff t+(\rho_0u_0^i, \Phi(0)),
	\end{aligned}
\end{equation}
where one has used the fact \(\|\sqrt{\rho_0}X^n(0)-\sqrt{\rho_0}u_0\|_{L^2(\Omega)}\to 0\) as \(n\to\infty\). As \(\Phi(0)\) is arbitrary, comparing \eqref{existence-17} and \eqref{existence-19}, one gets 
\begin{equation*}
	\begin{aligned}
		\|\rho_0v(0)-\rho_0u_0\|_{L^2(\Omega)}=0,
	\end{aligned}
\end{equation*}
which yields
\begin{equation*}
	\begin{aligned}
		\rho_0v(0)=\rho_0u_0\quad a.e.\ \text{in}\ \Omega.
	\end{aligned}
\end{equation*}
Hence \eqref{initial-add-2} follows due to \eqref{eq:intro-3}.

The uniqueness of the weak solutions of \eqref{existence-3} is easy to check since \eqref{existence-3} is a linear problem.

\end{proof}

\section{Regularity of Weak Solutions to the Linearized Problem}\label{Regularity}

\begin{lemma}\label{Regularity lemma} The weak solution \(v\) to \eqref{existence-3} satisfies the regularity estimate
	\begin{equation}\label{re-0}
		\begin{aligned}
			\sup_{0\leq t\leq T}E(t,v)
			\leq M_1.
		\end{aligned}
	\end{equation}
Consequently, the solution map \(\bar{v}\mapsto v: \mathcal{C}_T(M_1)\rightarrow \mathcal{C}_T(M_1)\) is well-defined.

\end{lemma}

\begin{proof} To prove \eqref{re-0}, it suffices to show
\begin{equation}\label{re-1}
\begin{aligned}
E(t,v)
\leq M_0+CtP(\sup_{0\leq s\leq t}E^{1/2}(s,\bar{v}))\quad \mathrm{for\ all}\ t\in [0,T].
\end{aligned}
\end{equation}
%whose proof is similar to that of \eqref{APB-1} in Section \ref{Energy Estimates} and Section \ref{Elliptic Estimates}. So we only sketch the proof of \eqref{existence-add-0} and point out the main modifications. 

\subsection{The strategy of the estimates}\label{re-se} 

The energy estimates will be carried out successively along the arrow direction: 
{\small{
\begin{align}
&\partial_tv\rightarrow \quad\  Dv\rightarrow \qquad \partial_1v\rightarrow\quad\ \partial_1Dv\rightarrow\quad\  \partial_1^2v\rightarrow \label{EEU-1}\\ \nonumber \\
&\partial_t^2v\rightarrow \quad \partial_tDv\rightarrow\quad\ \partial_t\partial_1v\rightarrow\  \partial_t\partial_1Dv\rightarrow\ \partial_t\partial_1^2v\rightarrow \label{EEU-2}\\
&\qquad\qquad\  \partial_1^2Dv\rightarrow\quad\ \partial_1^3v\rightarrow\quad\  \partial_1^3Dv\rightarrow \quad\  \partial_1^4v\rightarrow \label{EEU-3}\\ \nonumber \\
&\partial_t^3v\rightarrow\quad \partial_t^2Dv\rightarrow \quad\ \partial_t^2\partial_1v\rightarrow\ \partial_t^2\partial_1Dv\rightarrow\partial_t^2\partial_1^2v\rightarrow \label{EEU-4}\\
&\qquad\qquad\  \partial_t\partial_1^2Dv\rightarrow\ \partial_t\partial_1^3v\rightarrow\ \ \partial_t\partial_1^3Dv\rightarrow\partial_t\partial_1^4v\rightarrow\label{EEU-5} \\
&\qquad\qquad\  \partial_1^4Dv\rightarrow\quad\ \partial_1^5v\rightarrow\quad\ \  \partial_1^5Dv\rightarrow \quad  \partial_1^6v\rightarrow \label{EEU-6}\\ \nonumber \\
&\partial_t^4v\rightarrow\quad\ \partial_t^3Dv\rightarrow\quad \partial_t^3\partial_1v\rightarrow \   \partial_t^3\partial_1Dv\rightarrow\partial_t^3\partial_1^2v\rightarrow\label{EEU-7} \\
&\qquad\qquad\ \partial_t^2\partial_1^2Dv\rightarrow\ \partial_t^2\partial_1^3v\rightarrow\ \partial_t^2\partial_1^3Dv\rightarrow\partial_t^2\partial_1^4v\rightarrow\label{EEU-8} \\
&\qquad\qquad\ \partial_t\partial_1^4Dv\rightarrow\ \ \partial_t\partial_1^5v\rightarrow\ \partial_t\partial_1^5Dv\rightarrow\partial_t\partial_1^6v\rightarrow\label{EEU-9} \\
&\qquad\qquad\ \ \partial_1^6Dv\rightarrow\quad\ \partial_1^7v\rightarrow\quad\ \partial_1^7Dv \label{EEU-10}
\end{align}}}

The estimates in the third and fifth columns (can provide useful \(L_t^2 L^2\)-type estimates) are not studied in Section \ref{Energy Estimates}, however, which are necessary to show \eqref{re-1}. In fact, only pure tangential derivative estimates (including estimates on the time derivative \(\partial_t\) and the horizontal spatial derivative \(\partial_1\)) can provide useful \(L_t^2 L^2\)-type estimates. 

The procedure of the elliptic estimates are exactly same to Section \ref{Elliptic Estimates}. 

In the remaining subsections, we will sketch the proof of the energy estimate \eqref{EEU-1}, and give detailed proofs on the energy estimates \eqref{EEU-7} and \eqref{EEU-10}. \eqref{EEU-1} is the starting point of the energy estimates; \eqref{EEU-7} is the highest-order estimate on time derivatives; \eqref{EEU-10} is the last step of the energy estimates. 
The remaining estimates can be handled in a similar way as above together with an application of a bootstrapping procedure. 

In this section,  one needs to 
use the elliptic estimates in Section \ref{Elliptic Estimates} to locate some key energy estimates in which one needs some \(H^{1/2}\)-type and sharp 
\(L^\infty\)-type estimates to handle the nonlinear terms. 
This constitutes the main difficulties of this section.

Throughout this section, we denote by
\begin{equation*}
\begin{aligned}
\bar{Q}_1(t)=Q_1(|D\bar{v}|, |D\bar{\eta}|, |\partial_tD\bar{v}|),
\end{aligned}
\end{equation*}
where \(Q_1\) is given by \eqref{Jab-bound-2}, and
then observe that 
%\footnote{\eqref{re-Q-bound-2} is stronger than \eqref{Q-bound-2} since it is based on the a priori assumption \(\bar{v}\in\mathcal{C}_T(M_1)\).}:
\begin{equation}\label{re-Q-bound-2}
\begin{aligned}
 \bar{Q}_1(t)\lesssim P(M_1)
\end{aligned}
\end{equation}
since \(\bar{v}\in\mathcal{C}_T(M_1)\). Here and thereafter \(P(\cdot)\) denotes a generic polynomial function of its arguments, which may be different from line to line. One can define similarly \(\bar{Q}_2(t), \bar{Q}_3(t), \dots\), which also have a bound \(P(M_1)\). 

For convenience, we will use the short-hand notation 
\begin{equation*}
\begin{aligned}
F(M_0, \bar{v})=M_0+Ct P(\sup_{0\leq s\leq t}E^{1/2}(s,\bar{v})).
\end{aligned}
\end{equation*}

\subsection{Lower-order estimates}\label{re-LOP}

\begin{proposition}[\eqref{EEU-1}+elliptic estimates] It holds that
		\begin{equation}\label{re-2}
			\begin{aligned}
				&\|\sqrt{\rho_0}\partial_tv\|_{L^2}^2
				+\|\sqrt{\rho_0}\partial_tDv\|_{L_t^2 L^2}^2
				\leq F(M_0, \bar{v}),\\
				&\|\sqrt{\rho_0}Dv\|_{L^2}^2
				\leq F(M_0, \bar{v}),\\
				&\|\sqrt{\rho_0}\partial_1v\|_{L^2}^2
				+\|\sqrt{\rho_0}\partial_1Dv\|_{L_t^2 L^2}^2
				\leq F(M_0, \bar{v}),\\
				&\|\sqrt{\rho_0}\partial_1Dv\|_{L^2}^2
				\leq F(M_0, \bar{v}),\\
				&\|\sqrt{\rho_0}\partial_1^2v\|_{L^2}^2
				+\|\sqrt{\rho_0}\partial_1^2Dv\|_{L_t^2 L^2}^2
				\leq F(M_0, \bar{v}),\\
				&\|\rho_0\partial_2^2v\|_{L^2}^2
				\leq F(M_0, \bar{v}).
			\end{aligned}
		\end{equation}
\end{proposition}

\begin{proof} The strategy in showing \eqref{re-2} is an application of a bootstrapping procedure from \(\eqref{re-2}_1\) to \(\eqref{re-2}_6\).

First, one can work with  \eqref{existence-5} to prove \(\eqref{re-2}_1\) and
\(\eqref{re-2}_2\) by deducing some uniform bounds for \(\{X^n\}_{n=1}^\infty\) in a similar way as showing \eqref{first uniform estimates}. Moreover, one can also obtain the corresponding time regularity of \(v\) by deducing uniform bounds for \(\{\partial_t^{l_0}X^n\}_{n=1}^\infty\) \((l_0=2,3,4)\).  

However, one can not improve further the spatial regularity of \(v\) by obtaining higher-order uniform estimates for \(\{X^n\}_{n=1}^\infty\) in view of \eqref{existence-5} directly as in the standard parabolic theory (see \cite{MR2597943}) since \eqref{existence-5} is degenerate. Next, to overcome this difficulty, we turn to derive an equation for \(v\). 

\begin{proposition} It holds that
	\begin{equation}\label{re-12}
		\begin{aligned}
			\rho_0\partial_tv^i+(\rho_0^2 \bar{J}^{-2}\bar{a}_i^k),_{k}=(\rho_0\bar{J}^{-2}\bar{b}^{kj}v^i,_{j}),_{k} &\quad \mathrm{a.e.\ in}\ \Omega\times (0,T].
		\end{aligned}
	\end{equation}
\end{proposition}

\begin{proof}[Proof of \eqref{re-12}]
	We first claim that 
	\begin{equation}\label{re-13}
		\begin{aligned}
			\big(\rho_0\partial_tv^i, \phi\big)+\big(\rho_0\bar{J}^{-2}\bar{b}^{kj}v^i,_{j}, \phi,_{k}\big)=\big(\rho_0^2 \bar{J}^{-2}\bar{a}_i^k, \phi,_{k}\big)
		\end{aligned}
	\end{equation}
	for each \(\phi\in H^1(\Omega)\) and a.e. \(0< t\leq T\).  According to \eqref{existence-weak solution}, to prove \eqref{re-13}, it suffices to show
	\begin{align}\label{re-14}
		\big\langle\rho_0\partial_tv^i, \phi\big\rangle=\big(\rho_0\partial_tv^i, \phi\big).
	\end{align}
	In fact, \(\eqref{re-2}_1\) provides \(\sqrt{\rho_0}v_t\in L^\infty([0,T]; L^2(\Omega))\), and thus
	\begin{align*} 
		\rho_0v_t\in L^\infty([0,T]; L^2(\Omega)), 
	\end{align*}
	which implies \eqref{re-14}. 
	
	Due to \eqref{eq:intro-3},  one can obtain \(v\in H^2_{\text{loc}}(\Omega)\) from \eqref{re-13} by a standard argument (see \cite{MR2597943}). Hence \eqref{re-12} follows. 
	
\end{proof}

Then, by replacing \(\partial_1\) by the tangential difference quotient \(D_1^\tau\) defined in \eqref{DQ}\footnote{This can be done similarly as in the proof of \eqref{weighted eigenfunction regularity}, so we will not expand this procedure in this section for simplicity.}, one can instead work with \eqref{re-12} to complete the remaining estimates in \eqref{re-2}.

\end{proof}

\subsection{Intermediated estimates-I}\label{re-ITP-I}

\begin{proposition} [\eqref{EEU-2}] It holds that
		\begin{equation}\label{re-100}
			\begin{aligned}
				&\|\sqrt{\rho_0}\partial_t^2v\|_{L^2}^2
				+\|\sqrt{\rho_0}\partial_t^2Dv\|_{L_t^2 L^2}^2
				\leq F(M_0, \bar{v}),\\
				&\|\sqrt{\rho_0}\partial_tDv\|_{L^2}^2
				\leq F(M_0, \bar{v}),\\
				&\|\sqrt{\rho_0}\partial_t\partial_1v\|_{L^2}^2
				+\|\sqrt{\rho_0}\partial_t\partial_1Dv\|_{L_t^2 L^2}^2
				\leq F(M_0, \bar{v}),\\
				&\|\sqrt{\rho_0}\partial_t\partial_1Dv\|_{L^2}^2
				\leq F(M_0, \bar{v}),\\
				&\|\sqrt{\rho_0}\partial_t\partial_1^2v\|_{L^2}^2
				+\|\sqrt{\rho_0}\partial_t\partial_1^2Dv\|_{L_t^2 L^2}^2
				\leq F(M_0, \bar{v}).
			\end{aligned}
		\end{equation}
\end{proposition}

\begin{proposition}[\eqref{EEU-3}+elliptic estimates] It holds that
		\begin{equation}\label{re-105}
			\begin{aligned}
				&\|\sqrt{\rho_0}\partial_1^2Dv\|_{L^2}^2
				\leq F(M_0, \bar{v}),\\
				&\|\sqrt{\rho_0}\partial_1^3v\|_{L^2}^2
				+\|\sqrt{\rho_0}\partial_1^3Dv\|_{L_t^2 L^2}^2
				\leq F(M_0, \bar{v}),\\
				&\sum_{\substack{ l_1+l_2=3\\  l_1\geq 0,\ l_2\geq 2}} \big\|\sqrt{\rho_0^{l_2}}\partial_1^{l_1}\partial_2^{l_2}v\big\|_{L^2}^2
				\leq F(M_0, \bar{v}),\\
				&\|\rho_0\partial_t\partial_2^2v\|_{L^2}^2
				\leq F(M_0, \bar{v}),\\
				&\|\sqrt{\rho_0}\partial_1^3Dv\|_{L^2}^2
				\leq F(M_0, \bar{v}),\\
				&\|\sqrt{\rho_0}\partial_1^4v\|_{L^2}^2
				+\|\sqrt{\rho_0}\partial_1^4Dv\|_{L_t^2 L^2}^2
				\leq F(M_0, \bar{v}),\\
				&\sum_{\substack{ l_1+l_2=4\\  l_1\geq 0,\ l_2\geq 2}} \big\|\sqrt{\rho_0^{l_2}}\partial_1^{l_1}\partial_2^{l_2}v\big\|_{L^2}^2
				\leq F(M_0, \bar{v}).
			\end{aligned}
		\end{equation}
\end{proposition}

\subsection{The intermediated problems-II}\label{re-ITP-II}

\begin{proposition} [\eqref{EEU-4}] It holds that
		\begin{equation}\label{re-111}
			\begin{aligned}
				&\|\sqrt{\rho_0}\partial_t^3v\|_{L^2}^2
				+\|\sqrt{\rho_0}\partial_t^3Dv\|_{L_t^2 L^2}^2
				\leq F(M_0, \bar{v}),\\
				&\|\sqrt{\rho_0}\partial_t^2Dv\|_{L^2}^2
				\leq F(M_0, \bar{v}),\\
				&\|\sqrt{\rho_0}\partial_t^2\partial_1v\|_{L^2}^2
				+\|\sqrt{\rho_0}\partial_t^2\partial_1Dv\|_{L_t^2 L^2}^2
				\leq F(M_0, \bar{v}),\\
				&\|\sqrt{\rho_0}\partial_t^2\partial_1Dv\|_{L^2}^2
				\leq F(M_0, \bar{v}),\\
				&\|\sqrt{\rho_0}\partial_t^2\partial_1^2v\|_{L^2}^2
				+\|\sqrt{\rho_0}\partial_t^2\partial_1^2Dv\|_{L_t^2 L^2}^2
				\leq F(M_0, \bar{v}).
			\end{aligned}
		\end{equation}
\end{proposition}

\begin{proposition}[\eqref{EEU-5}+elliptic estimates] It holds that
		\begin{equation}\label{re-120}
			\begin{aligned}
				&\|\sqrt{\rho_0}\partial_t\partial_1^2Dv\|_{L^2}^2
				\leq F(M_0, \bar{v}),\\
				&\|\sqrt{\rho_0}\partial_t\partial_1^3v\|_{L^2}^2
				+\|\sqrt{\rho_0}\partial_t\partial_1^3Dv\|_{L_t^2 L^2}^2
				\leq F(M_0, \bar{v}),\\
				&\sum_{\substack{ l_1+l_2=3\\  l_1\geq 0,\ l_2\geq 2}} \big\|\sqrt{\rho_0^{l_2}}\partial_t\partial_1^{l_1}\partial_2^{l_2}v\big\|_{L^2}^2
				\leq F(M_0, \bar{v}),\\
				&\|\sqrt{\rho_0}\partial_t\partial_1^3Dv\|_{L^2}^2
				\leq F(M_0, \bar{v}),\\
				&\|\sqrt{\rho_0}\partial_t\partial_1^4v\|_{L^2}^2
				+\|\sqrt{\rho_0}\partial_t\partial_1^4Dv\|_{L_t^2 L^2}^2
				\leq F(M_0, \bar{v}).
			\end{aligned}
		\end{equation}
\end{proposition}

\begin{proposition}[\eqref{EEU-6}+elliptic estimates] It holds that
		\begin{equation}\label{re-125}
			\begin{aligned}
				&\|\sqrt{\rho_0}\partial_1^4Dv\|_{L^2}^2
				\leq F(M_0, \bar{v}),\\
				&\|\sqrt{\rho_0}\partial_1^5v\|_{L^2}^2
				+\|\sqrt{\rho_0}\partial_1^5Dv\|_{L_t^2 L^2}^2
				\leq F(M_0, \bar{v}),\\
				&\sum_{\substack{ l_1+l_2=5\\  l_1\geq 0,\ l_2\geq 2}} \big\|\sqrt{\rho_0^{l_2}}\partial_1^{l_1}\partial_2^{l_2}v\big\|_{L^2}^2
				\leq F(M_0, \bar{v}),\\
				&\|\rho_0\partial_t^2\partial_2^2v\|_{L^2}^2
				\leq F(M_0, \bar{v}),\\
				&\sum_{\substack{ l_1+l_2=4\\  l_1\geq 0,\ l_2\geq 2}} \big\|\sqrt{\rho_0^{l_2}}\partial_t\partial_1^{l_1}\partial_2^{l_2}v\big\|_{L^2}^2
				\leq F(M_0, \bar{v}),\\
				&\|\sqrt{\rho_0}\partial_1^5Dv\|_{L^2}^2
				\leq F(M_0, \bar{v}),\\
				&\|\sqrt{\rho_0}\partial_1^6v\|_{L^2}^2
				+\|\sqrt{\rho_0}\partial_1^6Dv\|_{L_t^2 L^2}^2
				\leq F(M_0, \bar{v}),\\
				&\sum_{\substack{ l_1+l_2=6\\  l_1\geq 0,\ l_2\geq 2}} \big\|\sqrt{\rho_0^{l_2}}\partial_1^{l_1}\partial_2^{l_2}v\big\|_{L^2}^2
				\leq F(M_0, \bar{v}).
			\end{aligned}
		\end{equation}
\end{proposition}

\subsection{Higher-order estimates}\label{re-HOP}

\begin{proposition} [\eqref{EEU-7}] It holds that
		\begin{equation}\label{re-150}
			\begin{aligned}
				&\|\sqrt{\rho_0}\partial_t^4v\|_{L^2}^2
				+\|\sqrt{\rho_0}\partial_t^4Dv\|_{L_t^2 L^2}^2
				\leq F(M_0, \bar{v}),\\
				&\|\sqrt{\rho_0}\partial_t^3Dv\|_{L^2}^2
				\leq F(M_0, \bar{v}),\\
				&\|\sqrt{\rho_0}\partial_t^3\partial_1v\|_{L^2}^2
				+\|\sqrt{\rho_0}\partial_t^3\partial_1Dv\|_{L_t^2 L^2}^2
				\leq F(M_0, \bar{v}),\\
				&\|\sqrt{\rho_0}\partial_t^3\partial_1Dv\|_{L^2}^2
				\leq F(M_0, \bar{v}),\\
				&\|\sqrt{\rho_0}\partial_t^3\partial_1^2v\|_{L^2}^2
				+\|\sqrt{\rho_0}\partial_t^3\partial_1^2Dv\|_{L_t^2 L^2}^2
				\leq F(M_0, \bar{v}).
			\end{aligned}
		\end{equation}
\end{proposition}

\begin{proposition}[\eqref{EEU-8}+elliptic estimates] It holds that
		\begin{equation}\label{re-160}
			\begin{aligned}
				&\|\sqrt{\rho_0}\partial_t^2\partial_1^2Dv\|_{L^2}^2
				\leq F(M_0, \bar{v}),\\
				&\|\sqrt{\rho_0}\partial_t^2\partial_1^3v\|_{L^2}^2
				+\|\sqrt{\rho_0}\partial_t^2\partial_1^3Dv\|_{L_t^2 L^2}^2
				\leq F(M_0, \bar{v}),\\
				&\sum_{\substack{ l_1+l_2=3\\  l_1\geq 0,\ l_2\geq 2}} \big\|\sqrt{\rho_0^{l_2}}\partial_t^2\partial_1^{l_1}\partial_2^{l_2}v\big\|_{L^2}^2
				\leq F(M_0, \bar{v}),\\
				&\|\sqrt{\rho_0}\partial_t^2\partial_1^3Dv\|_{L^2}^2
				\leq F(M_0, \bar{v}),\\
				&\|\sqrt{\rho_0}\partial_t^2\partial_1^4v\|_{L^2}^2
				+\|\sqrt{\rho_0}\partial_t^2\partial_1^4Dv\|_{L_t^2 L^2}^2
				\leq F(M_0, \bar{v}).
			\end{aligned}
		\end{equation}
\end{proposition}

\begin{proposition}[\eqref{EEU-9}+elliptic estimates] It holds that
		\begin{equation}\label{re-170}
			\begin{aligned}
				&\|\sqrt{\rho_0}\partial_t\partial_1^4Dv\|_{L^2}^2
				\leq F(M_0, \bar{v}),\\
				&\|\sqrt{\rho_0}\partial_t\partial_1^5v\|_{L^2}^2
				+\|\sqrt{\rho_0}\partial_t\partial_1^5Dv\|_{L_t^2 L^2}^2
				\leq F(M_0, \bar{v}),\\
				&\sum_{\substack{ l_1+l_2=5\\  l_1\geq 0,\ l_2\geq 2}} \big\|\sqrt{\rho_0^{l_2}}\partial_t\partial_1^{l_1}\partial_2^{l_2}v\big\|_{L^2}^2
				\leq F(M_0, \bar{v}),\\
				&\|\sqrt{\rho_0}\partial_t\partial_1^5Dv\|_{L^2}^2
				\leq F(M_0, \bar{v}),\\
				&\|\sqrt{\rho_0}\partial_t\partial_1^6v\|_{L^2}^2
				+\|\sqrt{\rho_0}\partial_t\partial_1^6Dv\|_{L_t^2 L^2}^2
				\leq F(M_0, \bar{v}).
			\end{aligned}
		\end{equation}
\end{proposition}

\begin{proposition}[\eqref{EEU-10}+elliptic estimates] It holds that
		\begin{equation}\label{re-180}
			\begin{aligned}
				&\|\sqrt{\rho_0}\partial_1^6Dv\|_{L^2}^2
				\leq F(M_0, \bar{v}),\\
				&\|\sqrt{\rho_0}\partial_1^7v\|_{L^2}^2
				+\|\sqrt{\rho_0}\partial_1^7Dv\|_{L_t^2 L^2}^2
				\leq F(M_0, \bar{v}),\\
				&\sum_{\substack{ l_1+l_2=7\\  l_1\geq 0,\ l_2\geq 2}} \big\|\sqrt{\rho_0^{l_2}}\partial_1^{l_1}\partial_2^{l_2}v\big\|_{L^2}^2
				\leq F(M_0, \bar{v}),\\
				&\|\rho_0\partial_t^3\partial_2^2v\|_{L^2}^2
				\leq F(M_0, \bar{v}),\\
				&\sum_{\substack{ l_1+l_2=4\\  l_1\geq 0,\ l_2\geq 2}} \big\|\sqrt{\rho_0^{l_2}}\partial_t^2\partial_1^{l_1}\partial_2^{l_2}v\big\|_{L^2}^2
				\leq F(M_0, \bar{v}),\\
				&\sum_{\substack{ l_1+l_2=6\\  l_1\geq 0,\ l_2\geq 2}} \big\|\sqrt{\rho_0^{l_2}}\partial_t\partial_1^{l_1}\partial_2^{l_2}v\big\|_{L^2}^2
				\leq F(M_0, \bar{v}),\\
				&\|\sqrt{\rho_0}\partial_1^7Dv\|_{L^2}^2
				\leq F(M_0, \bar{v}),\\
				&\sum_{\substack{ l_1+l_2=8\\  l_1\geq 0,\ l_2\geq 2}} \big\|\sqrt{\rho_0^{l_2}}\partial_1^{l_1}\partial_2^{l_2}v\big\|_{L^2}^2
				\leq F(M_0, \bar{v}).
			\end{aligned}
		\end{equation}
\end{proposition}

As mentioned in Section \ref{re-se}, we will give the proofs of \eqref{re-150} and the energy estimates in \eqref{re-180}. 
In the following, we mainly focus on the terms involving the \(H^{1/2}\)-type and sharp \(L^\infty\)-type estimates.

\subsubsection{Proof of \eqref{re-150}}

\begin{proof}[Proof of \(\eqref{re-150}_1\)]
Comparing with \eqref{EE-1}, one has 
\begin{equation}\label{re-200}
\begin{aligned}
&\frac{1}{2}\int_\Omega \rho_0|\partial_t^4v|^2\,\diff x
+\int_0^t\int_\Omega \rho_0\bar{J}^{-2}\bar{b}^{kj}\partial_t^4v^i,_{j}\partial_t^4v^i,_{k}\,\diff x\diff s\\
&=\frac{1}{2}\int_\Omega \rho_0|\partial_t^4v|^2(x,0)\,\diff x+\int_0^t\int_\Omega\rho_0^2\partial_t^4(\bar{J}^{-2}\bar{a}_i^k)\partial_t^4v^i,_{k}\,\diff x\diff s\\
&\quad-\sum_{m=1}^4\binom{4}{m}\underline{\int_0^t\int_\Omega\rho_0\partial_t^m(\bar{J}^{-2}\bar{b}^{kj})\partial_t^{4-m}v^i,_{j}\partial_t^4v^i,_{k}\,\diff x\diff s.}_{:=I_m}
\end{aligned}
\end{equation}

It suffices to take care of \(I_3\) and \(I_4\), which can be estimated as follows: %\footnote{It should be pointed out that \(\|\partial_tDv\|_{H^{1/2}}\) and \(\|Dv\|_{L^\infty}\) are a priori controlled by \(E(t,v)\) in \eqref{EE-2}, however here which shall be bounded by \(\eqref{re-120}_3\) and \(\eqref{re-125}_3\), respectively. }
\begin{equation*}
\begin{aligned}
&|I_3|
\leq CP^2(M_1)\int_0^t\|\partial_t^2D\bar{v}\|_{H^{1/2}}^2\|\partial_tDv\|_{H^{1/2}}^2\,\diff s+G,\\
&|I_4|
\leq CP^2(M_1)\int_0^t\|\sqrt{\rho_0}\partial_t^3D\bar{v}\|_{L^2}^2\|Dv\|_{L^\infty}^2\,\diff s+G
\end{aligned}
\end{equation*}
with \(G=\frac{1}{100}\int_0^t\int_\Omega\rho_0|\partial_t^4Dv|^2\,\diff x\diff s+\mathcal{R}\), which are bounded by
\(G+F(M_0, \bar{v})\).
Here one has used
\begin{equation}\label{re-210}
\begin{aligned}
\|\partial_tDv\|_{H^{1/2}}^2&\lesssim \|\rho_0^{1/2}\partial_tDv\|_{L^2}^2+\|\rho_0^{1/2}\partial_tD^2v\|_{L^2}^2\lesssim \cdots\\
&\lesssim \mathcal{R}+\|\rho_0^{3/2}\partial_tD^3v\|_{L^2}^2\leq F(M_0, \bar{v}),
\end{aligned}
\end{equation}
and
\begin{equation}\label{re-212}
\begin{aligned}
\|Dv\|_{L^\infty}^2\lesssim \sum_{k=0}^2\int_\Omega \rho_0|D^{k+1}v|^2\,\diff x\lesssim \cdots
\lesssim \mathcal{R}+\|\rho_0^{5/2}D^5v\|_{L^2}^2
\leq F(M_0, \bar{v}),
\end{aligned}
\end{equation}
in which \(\eqref{re-120}_3\) and \(\eqref{re-125}_3\) have been utilized.

\end{proof}

\begin{proof}[Proof of \(\eqref{re-150}_2\)] The proof is exactly same to 
\eqref{EEP-2}.

\end{proof}

\begin{proof}[Proof of \(\eqref{re-150}_3\)] 

Applying \(\partial_t^3\partial_1\) 
to \eqref{re-12} and multiplying it by \(\partial_t^3\partial_1v^i\), one obtains by some direct calculations that
\begin{equation*}
\begin{aligned}
&\frac{1}{2}\int_\Omega \rho_0|\partial_t^3\partial_1v|^2\,\diff x+\int_0^t\int_\Omega \rho_0\bar{J}^{-2}\bar{b}^{kj}\partial_t^3\partial_1v^i,_{j}\partial_t^3\partial_1v^i,_{k}\,\diff x\diff s\\
&=\frac{1}{2}\int_\Omega \rho_0|\partial_t^3\partial_1v|^2(x,0)\,\diff x+\int_0^t\int_\Omega \partial_t^3\partial_1(\rho_0^2\bar{J}^{-2}\bar{a}_i^k)\partial_t^3\partial_1v^i,_{k}\,\diff x\diff s\\
&\quad-\int_0^t\int_\Omega \partial_1(\rho_0)\partial_t^4v^i\partial_t^3\partial_1v^i\,\diff x\diff s
-\int_0^t\int_\Omega \partial_1(\rho_0\bar{J}^{-2}\bar{b}^{kj})\partial_t^3v^i,_{j}\partial_t^3\partial_1v^i,_{k}\,\diff x\diff s\\
&\quad-\sum_{m=1}^3\underline{\int_0^t\int_\Omega \partial_1[\rho_0\partial_t^m(\bar{J}^{-2}\bar{b}^{kj})\partial_t^{3-m}v^i,_{j}]\partial_t^3\partial_1v^i,_{k}\,\diff x\diff s.}_{:=I_m}
\end{aligned}
\end{equation*}

The most difficult term is \(I_3\) which can be estimated by
\begin{equation*}
\begin{aligned}
|I_3|
&\leq \frac{1}{100}\int_0^t\int_\Omega \rho_0|\partial_t^3\partial_1Dv|^2\,\diff x\diff s
+CP^2(M_1)\int_0^t\|\partial_t^2\partial_1D\bar{v}\|_{L^2}^2\|Dv\|_{L^\infty}^2\,\diff s\\
&\quad+CP^2(M_1)\int_0^t\|\sqrt{\rho_0}\partial_t^2D\bar{v}\|_{L^2}^2\|\partial_1Dv\|_{L^\infty}^2\,\diff s+\mathcal{R}
\leq F(M_0, \bar{v}),
\end{aligned}
\end{equation*}
where one has used \eqref{re-212} and
\begin{equation}\label{re-221}
\begin{aligned}
\|\partial_1Dv\|_{L^\infty}^2
\lesssim \mathcal{R}+\|\rho_0^{5/2}\partial_1D^5v\|_{L^2}^2\leq F(M_0, \bar{v}),
\end{aligned}
\end{equation}
where \(\eqref{re-125}_8\) has been utilized in the last inequality.

\end{proof}

Next, one can use the \(L_t^2 L^2\)-type estimate in \(\eqref{re-150}_3\) to show  \(\eqref{re-150}_4\).

\begin{proof}[Proof of \(\eqref{re-150}_4\)] Comparing with the proof of 
	\eqref{EEP-5}, the only modification is the estimate \eqref{spacetime-estimate} on \(I_1\),  which should be replaced by
\begin{equation*}
\begin{aligned}
|I_1|&\lesssim \bigg|\int_0^t\int_\Omega \rho_0[\partial_1(\bar{J}^{-2}\bar{b}^{kj})\partial_t^3\partial_1v^i,_{j}+\partial_1^2(\bar{J}^{-2}\bar{b}^{kj})\partial_t^3v^i,_{j}]\partial_t^4v^i,_{k}\,\diff x\diff s\bigg|
+\mathcal{R}\\
&\lesssim P^2(M_1)\int_0^t\int_\Omega\rho_0|\partial_t^3\partial_1Dv|^2\, \diff x\diff s
+P^2(M_1)\int_0^t\|\sqrt{\rho_0}\partial_t^3Dv\|_{L^2}^2\,\diff s\\
&\quad+\int_0^t\int_\Omega\rho_0|\partial_t^4Dv|^2\,\diff x\diff s+\mathcal{R}
\leq F(M_0, \bar{v}),
\end{aligned}
\end{equation*}
where \(\eqref{re-150}_3\) has been used.
\end{proof}

\begin{proof}[Proof of \(\eqref{re-150}_5\)]	
	
One can apply \(\partial_t^3\partial_1^2\) to \eqref{re-12}, multiply it by \(\partial_t^3\partial_1^2v^i\), to obtain by some elementary calculations that
\begin{equation*}
\begin{aligned}
&\frac{1}{2}\int_\Omega \rho_0|\partial_t^3\partial_1^2v|^2\,\diff x
+\int_0^t\int_\Omega\rho_0\bar{J}^{-2}\bar{b}^{kj}\partial_t^3\partial_1^2v^i,_{j}\partial_t^3\partial_1^2v^i,_{k}\,\diff x\diff s\\
&=\frac{1}{2}\int_\Omega \rho_0|\partial_t^3\partial_1^2v|^2(x,0)\,\diff x+\int_0^t\int_\Omega \partial_1^2[\rho_0^2\partial_t^3( \bar{J}^{-2}\bar{a}_i^k)]\partial_t^3\partial_1^2v^i,_{k}\,\diff x\diff s\\
&\quad-\sum_{m=1}^2\int_0^t\int_\Omega \partial_1^m(\rho_0)\partial_t^4\partial_1^{2-m}v^i\partial_t^3\partial_1^2v^i\,\diff x\diff s\\
&\quad-\sum_{m=1}^2\int_0^t\int_\Omega \partial_1^m(\rho_0\bar{J}^{-2}\bar{b}^{kj})\partial_t^3\partial_1^{2-m}v^i,_{j}\partial_t^3\partial_1^2v^i,_{k}\,\diff x\diff s\\
&\quad-\sum_{m=1}^3\underline{\int_0^t\int_\Omega \partial_1^2[\rho_0\partial_t^m(\bar{J}^{-2}\bar{b}^{kj})\partial_t^{3-m}v^i,_{j}]\partial_t^3\partial_1^2v^i,_{k}\,\diff x\diff s.}_{:=I_m}
\end{aligned}
\end{equation*}

The new ingredient is contained in \(I_3\), which may be estimated as follows:
\begin{equation*}
\begin{aligned}
|I_3|
&\leq CP^2(M_1)\int_0^t(\|\sqrt{\rho_0}\partial_t^2\partial_1^2D\bar{v}\|_{L^2}^2\|Dv\|_{L^\infty}^2
		+\|\partial_t^2\partial_1D\bar{v}\|_{H^{1/2}}^2\|\partial_1Dv\|_{H^{1/2}}^2\\
&\quad+\int_0^t\|\partial_t^2D\bar{v}\|_{H^{1/2}}^2\|\partial_1^2Dv\|_{H^{1/2}}^2)\,\diff s
+\frac{1}{100}\int_0^t\int_\Omega\rho_0|\partial_t^3\partial_1^2Dv|^2\,\diff x\diff s+\mathcal{R}\\
&\leq \frac{1}{100}\int_0^t\int_\Omega\rho_0|\partial_t^3\partial_1^2Dv|^2\,\diff x\diff s+F(M_0, \bar{v}),
	\end{aligned}
\end{equation*}
where one has used \(\eqref{re-105}_7\) and \(\eqref{re-125}_3\) to estimate 
\begin{equation}\label{re-230}
\begin{aligned}
\|\partial_1^{l_1}Dv\|_{H^{1/2}}^2
\lesssim \mathcal{R}+\|\rho_0^{3/2}\partial_1^{l_1}D^3v\|_{L^2}^2
\leq F(M_0, \bar{v}),\quad  l_1=1,2.
\end{aligned}
\end{equation}

\end{proof}

\subsubsection{Proof of \eqref{re-180}}

\begin{proof}[Proof of \(\eqref{re-180}_1\)]
	
This can be done in the exactly same way as for \eqref{EEP-3}. The new ingredient is to explain
how to estimate \(\|\partial_1^3Dv\|_{H^{1/2}}\) in \(\eqref{EE-30}_3\). More precisely, it suffices to obtain the following estimate: 
\begin{equation*}
		\begin{aligned}
		|I_{3_3}|
		&\lesssim P^2(M_1)\int_0^t
		\|\partial_1^4D\bar{\eta}\|_{H^{1/2}}^2
		\|\partial_1^3Dv\|_{H^{1/2}}^2\,\diff s\\
		&\quad+\int_0^t\|\sqrt{\rho_0}\partial_t\partial_1^5Dv\|_{L^2}^2\,\diff s+\mathcal{R}\leq F(M_0, \bar{v}),
	\end{aligned}
	\end{equation*}	
where one has used \(\eqref{re-125}_8\) to estimate
\begin{equation}\label{re-240}
\begin{aligned}
\|\partial_1^3Dv\|_{H^{1/2}}^2
\lesssim \mathcal{R}+\|\rho_0^{3/2}\partial_1^3D^3v\|_{L^2}^2
\leq F(M_0, \bar{v}).
\end{aligned}
\end{equation}

\end{proof}

\begin{proof}[Proof of \(\eqref{re-180}_2\)] 

Applying \(\partial_1^7\) to \eqref{re-12} and multiplying it by \(\partial_1^7v^i\), one gets after some direct calculations that
\begin{equation*}
\begin{aligned}
&\frac{1}{2}\int_\Omega \rho_0|\partial_1^7v|^2\,\diff x+\int_0^t\int_\Omega \rho_0\bar{J}^{-2}\bar{b}^{kj}\partial_1^7v^i,_{j}\partial_1^7v^i,_{k}\,\diff x\diff s\\
&=\frac{1}{2}\int_\Omega \rho_0|\partial_1^7v|^2(x,0)\,\diff x+\int_0^t\int_\Omega\partial_1^7(\rho_0^2\bar{J}^{-2}\bar{a}_i^k)\partial_1^7v^i,_{k}\,\diff x\diff s\\
&\quad-\sum_{m=1}^7\int_0^t\int_\Omega \partial_1^m(\rho_0)\partial_t\partial_1^{7-m}v^i\partial_1^7v^i\,\diff x\diff s\\
&\quad-\sum_{m=1}^7\int_0^t\int_\Omega \partial_1^m(\rho_0\bar{J}^{-2}\bar{b}^{kj})\partial_1^{7-m}v^i,_{j}\partial_1^7v^i,_{k}\,\diff x\diff s.
\end{aligned}
\end{equation*}
All the terms on the RHS above can be handled by means of
\eqref{re-212}-\eqref{re-240}.

\end{proof}

\begin{proof}[Proof of \(\eqref{re-180}_7\)]
Comparing with the proof of \eqref{EEP-7}, one needs only to take care of the last four estimates in \eqref{EE-60}, i.e., 
\begin{equation*}
	\begin{aligned}
	|I_{2_{m}}|
	&\lesssim P^2(M_1)\int_0^t \|\partial_1^{m}D\bar{\eta}\|_{H^{1/2}}^2\|\partial_t\partial_1^{7-m}Dv\|_{H^{1/2}}^2\,\diff s
	+G,\quad m=4,5,\\
	|I_{2_{m}}|
	&\lesssim P^2(M_1)\int_0^t \|\sqrt{\rho_0}\partial_1^{m}D\bar{\eta}\|_{L^2}^2
	\|\partial_t\partial_1^{7-m}Dv\|_{L^\infty}^2\,\diff s
	+G,\quad m=6,7
	\end{aligned}
\end{equation*}
with \(G=\int_0^t\int_\Omega \rho_0|\partial_1^7Dv|^2\,\diff x\diff s+\mathcal{R}\).

Notice that \(G\) admits the bound \(F(M_0, \bar{v})\) 
by \(\eqref{re-180}_2\). It suffices to estimate \(\|\partial_t\partial_1^{l_1}Dv\|_{H^{1/2}}\) \((l_1=2,3)\) and \(\|\partial_t\partial_1^{l_1}Dv\|_{L^\infty}\) \((l_1=0,1)\). In fact, one can use \(\eqref{re-170}_3\) and \(\eqref{re-180}_6\) to estimate
\begin{equation*}
\begin{aligned}
\|\partial_t\partial_1^{l_1}Dv\|_{H^{1/2}}^2
\lesssim \mathcal{R}+\|\rho_0^{3/2}\partial_t\partial_1^{l_1}D^3v\|_{L^2}^2\leq F(M_0, \bar{v}),\quad l_1=2,3
\end{aligned}
\end{equation*}
and 
\begin{equation*}
\begin{aligned}
\|\partial_t\partial_1^{l_1}Dv\|_{L^\infty}^2
\lesssim \mathcal{R}+ \|\rho_0^{5/2}\partial_t\partial_1^{l_1}D^5v\|_{L^2}^2
\leq F(M_0, \bar{v}),\quad l_1=0,1.
\end{aligned}
\end{equation*}

\end{proof}

Collecting all the estimates in Subsections \ref{re-LOP}-\ref{re-HOP} completes the proof of \eqref{re-1}. 

\end{proof}

\section{Proof of Theorem \ref{th:main-1}: Existence}\label{Existence Part}

\subsection{Iterative problems}

We will show that
there exists a classical solution to \eqref{eq:main-2} by an iteration. Therefore consider the iterative problem
{\small{\begin{equation}\label{existence-21}
	\begin{cases}
		\rho_0\partial_t(v^i)^{(n)}+[\rho_0^2 (J^{-2}a_i^k)^{(n-1)}],_{k}\\
		=[\rho_0(J^{-2}b^{kj})^{(n-1)}(v^i)^{(n)},_{j}],_{k} &\ \mbox{in}\ \Omega\times (0,T],\\ \\
		v^{(n)}=u_0 &\ \mbox{on}\ \Omega\times \{t=0\}.
	\end{cases}
\end{equation}}}

When \(n=1\),  imposing 
\(\eta^{(0)}(t,x)=x+tu_0(x)\), then one can solve \eqref{existence-21} for \(n=1,2,...\) iteratively. 
Indeed, in view of Lemma \ref{Regularity lemma}, one can obtain \(\{v^{(n)}\}_{n=1}^\infty\subset \mathcal{C}_T(M_1)\) for any \(n\geq 1\).

\subsection{Uniform estimates}
Next, we show that the above iterative solutions \(\{v^{(n)}\}_{n=1}^\infty\) 
are contractive in some appropriate energy space. To this end, 
setting 
\[\sigma[v^{(n)}]:=v^{(n+1)}-v^{(n)},\]
one deduces that
{\small{\begin{equation}\label{existence-22}
	\begin{cases}
		\rho_0\partial_t\sigma[(v^i)^{(n)}]+[\rho_0^2 (J^{-2}a_i^k)^{(n)}],_{k}
-[\rho_0^2 (J^{-2}a_i^k)^{(n-1)}],_{k}\\
		=[\rho_0(J^{-2}b^{kj})^{(n)}(v^i)^{(n+1)},_{j}],_{k}
		-[\rho_0(J^{-2}b^{kj})^{(n-1)}(v^i)^{(n)},_{j}],_{k} &\ \mbox{in}\ \Omega\times (0,T],\\ \\
	\sigma[v^{(n)}]=0 &\ \mbox{on}\ \Omega\times \{t=0\}.
	\end{cases}
\end{equation}}}

Then \(\{\sigma[v^{(n)}]\}_{n=1}^\infty\)  satisfy the following uniform estimates in \(n\geq 1\).
\begin{lemma} It holds that
	\begin{equation}\label{existence-23}
		\begin{aligned}
			&\frac{\diff}{\diff t}\int_\Omega\rho_0|\sigma[v^{(n)}]|^2\,\diff x
			+\int_\Omega\rho_0|D\sigma[v^{(n)}]|^2\,\diff x\\
			&\leq Ct(1+M_1+M_1^2)\int_0^t\int_\Omega\rho_0|D\sigma[v^{(n-1)}]|^2\,\diff x\diff s\quad \mathrm{for}\ 0<t\leq T.
		\end{aligned}
	\end{equation}		
\end{lemma}

\begin{proof} Multiplying \(\eqref{existence-22}_1\) by \(\sigma[v^{(n)}]\) and integrating by parts yield
	\begin{equation}\label{existence-24}
		\begin{aligned}
			&\frac{1}{2}\frac{\diff}{\diff t}\int_\Omega\rho_0|\sigma[v^{(n)}]|^2\,\diff x\\
			&+\underline{\int_\Omega\rho_0\big[(J^{-2}b^{kj})^{(n)}(v^i)^{(n+1)},_{j}
			-(J^{-2}b^{kj})^{(n-1)}(v^i)^{(n)},_{j}\big]\sigma[(v^i)^{(n)}],_{k}\,\diff x}_{:=I_1}\\
			&=\underline{\int_\Omega\rho_0^2\big[(J^{-2}a_i^k)^{(n)}- (J^{-2}a_i^k)^{(n-1)}\big]\sigma[(v^i)^{(n)}],_{k}\,\diff x.}_{:=I_2}	
		\end{aligned}
	\end{equation}	
	
Note that
	\begin{equation}\label{existence-25}
		\begin{aligned}
			I_1&=\underline{\int_\Omega\rho_0(J^{-2}b^{kj})^{(n)}\big[(v^i)^{(n+1)},_{j}
			-(v^i)^{(n)},_{j}\big]\sigma[(v^i)^{(n)}],_{k}\,\diff x}_{:=J_1}\\
			&+\underline{\int_\Omega\rho_0(v^i)^{(n)},_{j}\big[(J^{-2}b^{kj})^{(n)}
			-(J^{-2}b^{kj})^{(n-1)}\big]\sigma[(v^i)^{(n)}],_{k}\,\diff x.}_{:=J_2}
		\end{aligned}
	\end{equation}	
First, it is easy to see that
\begin{equation}\label{existence-26}
	\begin{aligned}
	J_1
\geq \frac{2}{11}\int_\Omega\rho_0|D\sigma[v^{(n)}]|^2\,\diff x.
	\end{aligned}
\end{equation}	
To handle \(J_2\), one writes 
\begin{equation*}
\begin{aligned}
(J^{-2}b^{kj})^{(n)}-(J^{-2}b^{kj})^{(n-1)}
&=[(J^{-2})^{(n)}-(J^{-2})^{(n-1)}](b^{kj})^{(n)}\\
&\quad+(J^{-2})^{(n-1)}[(b^{kj})^{(n)}-(b^{kj})^{(n-1)}].
\end{aligned}
\end{equation*}	
By \eqref{J-formula}, one decomposes 
\begin{equation*}
\begin{aligned}
 (J^{-2})^{(n)}-(J^{-2})^{(n-1)}
 &=[(\eta^{(n-1)})^1,_{1}-(\eta^{(n)})^1,_{1}](\eta^{(n-1)})^2,_{2}\\
 &\quad+(\eta^{(n)})^1,_{1}[(\eta^{(n-1)})^2,_{2}-(\eta^{(n)})^2,_{2}]\\
 &\quad+[(\eta^{(n)})^1,_{2}-(\eta^{(n-1)})^1,_{2}](\eta^{(n)})^2,_{1}\\
 &\quad+(\eta^{(n-1)})^1,_{2}[(\eta^{(n)})^2,_{1}-(\eta^{(n-1)})^2,_{1}]
\end{aligned}
\end{equation*}	
and then uses \eqref{Eta-formula} to estimate
\begin{equation*}
\begin{aligned}
|(J^{-2})^{(n)}-(J^{-2})^{(n-1)}|
\lesssim (1+M_1^{1/2})\int_0^t|D\sigma[v^{(n-1)}]|\,\diff s
\end{aligned}
\end{equation*}	
since \(\{v^{(n)}\}_{n=1}^\infty\subset \mathcal{C}_T(M_1)\) and \(0<t\leq T<1\).
Similarly, by \eqref{Eta-formula} and \eqref{b-formula}, one can get
\begin{equation*}
\begin{aligned}
|(b^{kj})^{(n)}|
\lesssim 1+M_1^{1/2}+M_1
\end{aligned}
\end{equation*}	
and
\begin{equation*}
\begin{aligned}
|(b^{kj})^{(n)}-(b^{kj})^{(n-1)}|
\lesssim (1+M_1^{1/2})\int_0^t|D\sigma[v^{(n-1)}]|\,\diff s.
\end{aligned}
\end{equation*}
Hence,
\begin{equation*}
\begin{aligned}
|(J^{-2}b^{kj})^{(n)}-(J^{-2}b^{kj})^{(n-1)}|
\lesssim (1+M_1^{1/2}+M_1+M_1^{3/2})
\int_0^t|D\sigma[v^{(n-1)}]|\,\diff s.
\end{aligned}
\end{equation*}
Now, it holds that
\begin{equation}\label{existence-27}
	\begin{aligned}
		|J_2|
		&\leq CtM_1^{1/2}(1+M_1^{1/2}+M_1+M_1^{3/2})\int_0^t\int_\Omega\rho_0|D\sigma[v^{(n-1)}]|^2\,\diff x\diff s\\
		&\quad+\frac{1}{100}\int_\Omega\rho_0|D\sigma[v^{(n)}]|^2\,\diff x.
	\end{aligned}
\end{equation}	
Consequently, it follows from \eqref{existence-25}-\eqref{existence-27} that
	\begin{equation}\label{existence-28}
	\begin{aligned}
		I_1
		&\geq -CtM_1^{1/2}(1+M_1^{1/2}+M_1+M_1^{3/2})\int_0^t\int_\Omega\rho_0|D\sigma[v^{(n-1)}]|^2\,\diff x\diff s\\
		&\quad+\frac{1}{10}\int_\Omega\rho_0|D\sigma[v^{(n)}]|^2\,\diff x.
	\end{aligned}
\end{equation}		

Similar to \(J_2\), one can obtain
	\begin{equation}\label{existence-29}
		\begin{aligned}
			|I_2|
			&\leq Ct(1+M_1^{1/2}+M_1)\int_0^t\int_\Omega\rho_0|D\sigma[v^{(n-1)}]|^2\,\diff x\diff s\\
			&\quad+\frac{1}{100}\int_\Omega\rho_0|D\sigma[v^{(n)}]|^2\,\diff x. 
		\end{aligned}
	\end{equation}	

Substituting \eqref{existence-28} and \eqref{existence-29} into \eqref{existence-24} gives \eqref{existence-23}.

\end{proof}

\begin{corollary} It holds that
\begin{equation}\label{existence-30}
	\begin{aligned}
		&\underset{0\leq t\leq T}{\sup}\|\rho_0^{1/2}\sigma[v^{(n)}](t)\|_{L^2(\Omega)}+\|\rho_0^{1/2}D\sigma[v^{(n)}]\|_{L^2([0,T];L^2(\Omega))}\\
		&\leq \frac{1}{2}\big(\underset{0\leq t\leq T}{\sup}\|\rho_0^{1/2}\sigma[v^{(n-1)}](t)\|_{L^2(\Omega)}+
		\|\rho_0^{1/2}D\sigma[v^{(n-1)}]\|_{L^2([0,T];L^2(\Omega))}\big).
	\end{aligned}
\end{equation}	

\end{corollary}
\begin{proof}
It follows from \eqref{existence-23} that
	\begin{equation*}
		\begin{aligned}
			&\underset{0\leq t\leq T}{\sup}\|\rho_0^{1/2}\sigma[v^{(n)}](t)\|_{L^2(\Omega)}^2+\|\rho_0^{1/2}D\sigma[v^{(n)}]\|_{L^2([0,T];L^2(\Omega))}^2\\
			&\leq CT(1+M_1+M_1^2)\|\rho_0^{1/2}D\sigma[v^{(n-1)}]\|_{L^2([0,T];L^2(\Omega))}^2\\
			&\leq \frac{1}{4}\big(\underset{0\leq t\leq T}{\sup}\|\rho_0^{1/2}\sigma[v^{(n-1)}](t)\|_{L^2(\Omega)}^2+
			\|\rho_0^{1/2}D\sigma[v^{(n-1)}]\|_{L^2([0,T];L^2(\Omega))}^2\big)
		\end{aligned}
	\end{equation*}
since \(T>0\) is sufficiently small. Hence \eqref{existence-30} follows.
\end{proof}

\subsection{Convergence}

%\eqref{existence-30} implies that \(\{v^{(n)}\}_{n=1}^\infty\) is a Cauchy sequence in \(L^2([0,T]; L^2(\Omega))\) by \eqref{ineq:weighted Sobolev-1}. 
\eqref{existence-30} implies only that \(\{v^{(n)}\}_{n=1}^\infty\) is a Cauchy sequence in \(L^2([0,T]; L^2(\Omega))\), 
which is insufficient to pass limit in \(n\) for \eqref{existence-21} in time pointwisely. 
To overcome this difficulty, we need to use the weighted interpolation inequality \eqref{Weighted-13}. 
Indeed, taking \(g=\sigma[(v^i)^{(n)}](t)\) in \eqref{Weighted-13} yields that for each \(t\in[0,T]\) 
	\begin{equation}\label{existence-35}
		\begin{aligned}
	\|\sigma[(v^i)^{(n)}](t)\|_{L^2(\Omega)}\lesssim \|\sigma[(v^i)^{(n)}](t)\|_{L_{\rho_0}^2(\Omega)}^{1/2}\|\sigma[(v^i)^{(n)}](t)\|_{H_{\rho_0}^1(\Omega)}^{1/2}.
	\end{aligned}
\end{equation}
Note that \eqref{existence-30} implies that \(\|\sigma[v^{(n)}](t)\|_{L_{\rho_0}^2(\Omega)}\to 0\) as \(n\to \infty\).
And \eqref{APB-2} implies that \(\|\sigma[v^{(n)}](t)\|_{H^4(\Omega)}\) is uniformly bounded in \(n\).  This together with \eqref{existence-35} implies that 
\begin{align}\label{existence-36}
		v^{(n)}\rightarrow v\quad \text{in}\ C([0,T];L^2(\Omega))\quad \text{as} \ n\to \infty.
\end{align}
	
On the other hand, by the Gagliardo-Nirenberg interpolation inequality,  
\begin{equation}\label{existence-37}
	\begin{aligned}
		\|\sigma[v^{(n)}](t)\|_{H^s(\Omega)}\lesssim \|\sigma[v^{(n)}](t)\|_{L^2(\Omega)}^{1-\frac{s}{4}}
		\|\sigma[v^{(n)}](t)\|_{H^4(\Omega)}^{\frac{s}{4}} \quad \text{for}\ s\in(0,4).
	\end{aligned}
\end{equation}
It follows from \eqref{existence-36} and \eqref{existence-37} that 
\begin{align*}
	v^{(n)}\rightarrow v\quad \text{in}\ C([0,T];H^s(\Omega)) \quad \text{for}\ s\in(0,4), \quad \text{as} \ n\to \infty,
\end{align*}
which furthermore implies that 
\begin{align}\label{existence-38}
	v^{(n)}\rightarrow v\quad \text{in}\ C([0,T];C^2(\Omega))\quad \text{as} \ n\to \infty.
\end{align}

According to \(\eqref{existence-21}_1\), one has 
\begin{equation*}
	\begin{aligned}
		\rho_0\partial_t(v^i)^{(n)}=-[\rho_0^2 (J^{-2}a_i^k)^{(n-1)}],_{k}
		+[\rho_0(J^{-2}b^{kj})^{(n-1)}(v^i)^{(n)},_{j}],_{k},
	\end{aligned}
\end{equation*}
which, together with \eqref{existence-38}, yields that as \(n\to \infty\)
\begin{align}\label{existence-39}
	\rho_0\partial_t(v^i)^{(n)}\to -(\rho_0^2 J^{-2}a_i^k),_{k}+(\rho_0J^{-2}b^{kj}v^i,_{j}),_{k}\quad \text{in}\ C([0,T];C(\Omega)).
\end{align}
Due to \eqref{existence-39}, the
distribution limit of \(\partial_tv^{(n)}\) must be \(\partial_tv\) as \(n\to \infty\), so, in particular, \(v\) is a classical solution to \eqref{eq:main-2}.
Moreover, following the standard argument, one may show 
\[v\in C([0,T]; H^4(\Omega))\cap C^1([0,T]; H^2(\Omega)).\]

\section{Proof of Theorem \ref{th:main-1}: Uniqueness}\label{Uniqueness Part}

Let \(v\) and \(\tilde{v}\) be two solutions to \eqref{eq:main-2} on \(\Omega\times (0,T]\) with initial data \((\rho_0,u_0)\) satisfying \eqref{APB-1}. Their corresponding flow maps are defined by
\begin{equation*}
	\begin{aligned}
		\eta(x,t)=x+\int_0^tv(x,s)\,\diff s\quad \text{and}\quad
		\tilde{\eta}(x,t)=x+\int_0^t\tilde{v}(x,s)\,\diff s.
	\end{aligned}
\end{equation*}
\(J\) and \(\tilde{J}\), and  \(a\) and \(\tilde{a}\) can be defined similarly. 
Set
\begin{equation*}
	\begin{aligned}
		\delta_{v\tilde{v}}=v-\tilde{v}.
	\end{aligned}
\end{equation*}
Then, \(\delta_{v\tilde{v}}\) solves
{\small{\begin{equation}\label{U-1}
	\begin{cases}
		\rho_0\partial_t\delta_{v\tilde{v}}^i+[\rho_0^2(J^{-2}a_i^k-\tilde{J}^{-2}\tilde{a}_i^k)],_{k}\\
		=[\rho_0(J^{-2}b^{kj}v^i,_{j}-\tilde{J}^{-2}\tilde{b}^{kj}\tilde{v}^i,_{j})],_{k}&\ \mbox{in}\ \Omega\times (0,T],\\ \\
		\delta_{v\tilde{v}}=0 &\ \mbox{on}\ \Omega\times \{t=0\}.
	\end{cases}
\end{equation}}}

\noindent Finally, by \(\eqref{U-1}_1\), one can follow a similar proof of \eqref{existence-23} to show 
{\small	\begin{equation*}
		\begin{aligned}
			\frac{\diff}{\diff t}\int_\Omega\rho_0|\delta_{v\tilde{v}}|^2\,\diff x
			+\int_\Omega\rho_0|D\delta_{v\tilde{v}}|^2\,\diff x
			\leq Ct(1+M_1+M_1^2)\int_0^t\int_\Omega\rho_0|D\delta_{v\tilde{v}}|^2\,\diff x\diff s
		\end{aligned}
	\end{equation*}	}	

\noindent for \(0<t\leq T\), which, together with Gronwall's inequality and \(\eqref{U-1}_2\),  implies \(\delta_{v\tilde{v}}=0\) in \(\Omega\times (0,T]\).

\section{Well-Posedness on General Domains}\label{general domain case}
In this section, we will extend the main results in Section \ref{main results} on the simplified domain \(\Omega=\mathbb{T}\times(0,1)\) to a general domain \(\Omega\) with \(\partial\Omega\in C^8\).

\subsection{Geometry of the boundary}

\begin{definition}[\cite{MR2597943}]  \(\partial\Omega\) is said to be \(C^\alpha\) if for each \(x^*\in \partial\Omega\) there exist \(r>0\) and a \(C^\alpha\) function $\gamma : \mathbb{R}^{d-1}\rightarrow \mathbb{R}$ such that-upon relabeling and reorienting the coordinates axes if necessary-we have 
\begin{align*}
  \Omega\cap B(x^*,r)=\{x\in B(x^*,r):x_d>\gamma(x_1,\ldots,x_{d-1})\}.
\end{align*}

\end{definition}

Fix \(x^*\in \partial \Omega\), and choose \(r, \gamma\), etc. as above. Define 
\begin{equation*}
\begin{cases}
y_i=x_i=:\Phi^i(x)\quad    (i=1,2,\cdots,d-1),\\
y_d=x_d-\gamma(x_1,\ldots,x_{d-1})=:\Phi^d(x),
\end{cases}
\end{equation*}
and write \(y=\Phi(x)\).
Similarly, set 
\begin{equation*}
\begin{cases}
x_i=y_i=:\Psi^i(x)\quad    (i=1,2,\cdots,d-1),\\
x_d=y_d+\gamma(x_1,\ldots,x_{d-1})=:\Psi^d(x),
\end{cases}
\end{equation*}
and write \( x=\Psi(y)\).
Then \(\Phi=\Psi^{-1}\), and the mapping \(x\mapsto \Phi(x)=y\) ``straightens out \(\partial \Omega\)'' near \(x^*\). Observe that \(\det D\Phi=\det D\Psi=1\).  

Since \(\partial\Omega\)  is compact, we can find finitely many points \(x^m\in \partial\Omega\), radii \(r_m>0\) and \(\gamma_m\) such that \(\partial\Omega\subset \bigcup_{m=1}^K B(x^m,r_m/2)\) and 
\begin{align*}
 &\Omega\cap B(x^m,r_m)=\{x\in B(x^m,r_m):x_d>\gamma_m(x_1,\ldots,x_{d-1})\},\\
&\partial\Omega\cap B(x^m,r_m)=\{x\in B(x^m,r_m):x_d=\gamma_m(x_1,\ldots,x_{d-1})\}.
\end{align*}
Then there exist some \(\Phi_m\) and small \(s_m>0\) for \(m=1,2,\cdots, K\) such that 
\begin{equation*}
\begin{aligned}
&\Phi_m(\Omega \cap B(x^m,r_m/2))\subset U_m\subset \Phi_m(\Omega \cap B(x^m,r_m)),\\
&\Phi_m(\partial\Omega \cap B(x^m,r_m/2))\subset \Sigma_m\subset \Phi_m(\partial\Omega \cap B(x^m,r_m)), 
\end{aligned}
\end{equation*}	
where
\begin{equation*}
\begin{aligned}
U_m=:B(0,s_m)\cap \{y_n>0\},\
\Sigma_m=:B(0,s_m)\cap \{y_n=0\}.
\end{aligned}
\end{equation*}	

Take an open set \(\Omega_0\subset\subset\Omega\) such that \(\Omega\subset(\Omega_0\bigcup_{m=1}^K B(x^m,r_m/2))\). Now let \(\{\zeta_0, \zeta_1, \cdots, \zeta_K\}\) be a smooth partition of unity on \(\bar{\Omega}\), subordinate to the open covering \(\{\Omega_0, B(x^1,r_1/2),\cdots, B(x^K,r_K/2)\}\), that is, 
\begin{equation*}
\begin{aligned}
&\zeta_0, \zeta_m\in C^\infty (\mathbb{R}^d),\quad \text{and}\quad 0\leq \zeta_0, \zeta_m\leq 1 \quad \text{for}\ m=1,2,\cdots, K,\\
& \supp \zeta_0 \subset \Omega_0, \quad \text{and}\quad \supp \zeta_m \subset B(x^m,r_m/2) \quad \text{for}\ m=1,2,\cdots, K,\\
&\zeta_0(x)+\sum_{m=1}^K\zeta_m(x)=1\quad \text{for}\ x\in \bar{\Omega}, \quad \text{and}\quad \sum_{m=1}^K\zeta_m(x)=1\quad \text{for}\ x\in \partial\Omega.
\end{aligned}
\end{equation*}

In what follows, we will confine to a general domain \(\Omega\) with \(d=2\) and \(\partial\Omega\in C^8\). 

\subsection{The higher-order energy functional and main results}

Let 
\begin{equation*}
\begin{aligned}
\tilde{\partial}_1=\partial_{y_1},\quad \tilde{\partial}_2=\partial_{y_2}, \quad \tilde{D}=(\tilde{\partial}_1, \tilde{\partial}_2)
\end{aligned}
\end{equation*}	
and
\begin{equation*}
\begin{aligned}
{\widetilde{\rho_0}}_m(y)=\rho_0(\Psi_m(y)),\quad \tilde{v}_m (y)=v(\Psi_m(y)).
\end{aligned}
\end{equation*}	
The higher-order energy functional consists of the inner energy and the boundary energy as follows:
{\small{\begin{equation*}\label{HOEF-100}
\begin{aligned}
\widetilde{E}(t,v)=E^{\text{I}}(t, v)+\sum_{m=1}^KE^{\text{B}_m}(t, v),
\end{aligned}
\end{equation*}}}

\noindent where 
{\small{\begin{equation*}\label{HOEF-101}
\begin{aligned}
E^{\text{I}}(t,v)&=\sum_{l_0=0}^4\|\partial_t^{l_0}v\|_{L^2(\Omega_0)}^2
+\sum_{\substack{2l_0+l_1\leq 6\\ l_0,\ l_1\geq 0}}\big(\|\partial_t^{l_0}\partial_1^{l_1}Dv\|_{L^2(\Omega_0)}^2+\|\partial_t^{l_0}\partial_1^{l_1+1}Dv\|_{L^2(\Omega_0)}^2\big)\\
&\quad+\sum_{\substack{2l_0+l_1+l_2\leq 8\\  l_0,\ l_1\geq 0,\ l_2\geq 2}}\big\|\partial_t^{l_0}\partial_1^{l_1}\partial_2^{l_2}v\big\|_{L^2(\Omega_0)}^2
\end{aligned}
\end{equation*}}}

\noindent and 
{\small{\begin{equation}\label{HOEF-102}
\begin{aligned}
E^{\text{B}_m}(t,v)&=
\sum_{l_0=0}^4\|\sqrt{{\widetilde{\rho_0}}_m}\partial_t^{l_0}\tilde{v}_m\|_{L^2(U_m)}^2
+\sum_{\substack{2l_0+l_1\leq 6\\ l_0,\ l_1\geq 0}}\big(\|\sqrt{{\widetilde{\rho_0}}_m}\partial_t^{l_0}\tilde{\partial}_1^{l_1}\tilde{D}\tilde{v}_m\|_{L^2(U_m)}^2\\
&\quad+\|\sqrt{{\widetilde{\rho_0}}_m}\partial_t^{l_0}\tilde{\partial}_1^{l_1+1}\tilde{D}\tilde{v}_m\|_{L^2(U_m)}^2\big)\\
&\quad+\sum_{\substack{2l_0+l_1+l_2\leq 8\\  l_0,\ l_1\geq 0,\ l_2\geq 2}}\big\|\sqrt{{\widetilde{\rho_0}}_m^{l_2}}\partial_t^{l_0}\tilde{\partial}_1^{l_1}\tilde{\partial}_2^{l_2}\tilde{v}_m\big\|_{L^2(U_m)}^2.
\end{aligned}
\end{equation}}}

Define the polynomial function \(M_0\) by
\begin{align*}
M_0=P(\widetilde{E}(0,v_0)),
\end{align*}
where \(P\) denotes a generic polynomial function of its argument. Then the main results can be stated as follows:

\begin{theorem}\label{th:main-100} Let \(\partial\Omega\in C^8\) and \((\rho_0,v_0)\) satisfy \eqref{eq:intro-3} and 
\(\widetilde{E}(0,v_0)<\infty\). Then
there exist a suitably small \(T>0\) and a unique classical solution 
\begin{align*}
v\in C([0,T]; H^4(\Omega))\cap C^1([0,T]; H^2(\Omega))
\end{align*}
to \eqref{eq:main-2} such that
\begin{align*}
\sup_{0\leq t\leq T} \widetilde{E}(t,v)\leq 2M_0.
\end{align*}

Moreover, \(v\) satisfies the boundary condition
\begin{align*}
(\rho_0),_{k}a_l^ka_l^jv^i,_{j}=0 \quad \rm{on}\ \Gamma\times (0,T].
\end{align*}

\end{theorem}

\begin{theorem}\label{th:main-101} Let \(\partial\Omega\in C^8\) and \((\rho_0,u_0)\) satisfy \eqref{eq:intro-3} and 
\(\widetilde{E}(0,v_0)<\infty\). Then
there exist a \(T>0\) and a unique classical solution \((\rho(y,t), u(y,t), \Gamma(t))\) for \(t\in[0,T]\) to \eqref{eq:intro-vfb}. Moreover, \(\Gamma(t)\in C^2([0,T])\), and for \(t\in[0,T]\) and \(y\in \Omega(t)\), it holds that
\begin{equation*}
\begin{aligned}
&\rho(y,t)\in C([0,T];H^4(\Omega(t)))\cap C^1([0,T];H^3(\Omega(t))),\\
&u(y,t)\in C([0,T];H^4(\Omega(t)))\cap C^1([0,T];H^2(\Omega(t))).
\end{aligned}
\end{equation*}

Furthermore, \(u\) satisfies the boundary condition
\begin{align*}
\nabla\rho\cdot\mathbb{D}(u)=0 \quad \rm{on}\ \Gamma(t).
\end{align*}

\end{theorem}

\subsection{Some important inequalities}  \eqref{ineq:weighted Sobolev-1}-\eqref{ineq:weighted Sobolev-3} can be transferred to

\begin{lemma}\label{le:Preliminary-100} For each \(m\in\{1,2,\cdots, K\}\), it holds that
\begin{equation}\label{ineq:weighted Sobolev-100}
\begin{aligned}
&\|w\|_{H^{1/2}(U_m)}^2\lesssim \int_{U_m} {\widetilde{\rho_0}}_m(w^2+|Dw|^2)(y)\,\diff y,\\
&\int_{U_m} {\widetilde{\rho_0}}_m^{\alpha}w^2(y)\,\diff y\lesssim \int_{U_m} {\widetilde{\rho_0}}_m^{\alpha+2}(w^2+|Dw|^2)(y)\,\diff y\quad \mathrm{for}\ \alpha=0,1,\cdots,\\
&\|w\|_{L^\infty(U_m)}^2\lesssim \sum_{\theta=0}^2 \int_{U_m} {\widetilde{\rho_0}}_m|D^\theta w|^2(y)\,\diff y.
\end{aligned}
\end{equation}

\end{lemma}

\begin{proof} When \(y\in U_m\) is near \(\Sigma_m\), \eqref{ineq:weighted Sobolev-100} follow from \eqref{ineq:weighted Sobolev-1}-\eqref{ineq:weighted Sobolev-3} directly. When \(y\in U_m\) is far from \(\Sigma_m\), 
since \({\widetilde{\rho_0}}_m(y)\)  is bounded below by a positive constant, then \eqref{ineq:weighted Sobolev-100} hold obviously.  

\end{proof}

Based on Lemma \ref{le:Preliminary-100}, one can transfer all the inequalities defined on \(\Omega\) in \(x\)-coordinate in Lemma \ref{le:Preliminary-1}, Lemma \ref{le:W-1}, Corollary \ref{le:W-2}, Lemmas \ref{le:W-3}-\ref{le:W-8} to \(U_m\) in \(y\)-coordinate.

 \eqref{TE} is replaced by 

\begin{lemma}  For each \(m\in\{1,2,\cdots, K\}\), it holds that
\begin{equation}\label{TE-1}
\begin{aligned}
 \|\tilde{\partial}_1^l({\widetilde{\rho_0}}_m)/{\widetilde{\rho_0}}_m\|_{L^\infty(U_m)}\leq C\quad \mathrm{for}\ m=1,2,\cdots
\end{aligned}
\end{equation}
\end{lemma}

\begin{proof} When \(y\in U_m\) is far from \(\Sigma_m\), \eqref{TE-1} holds obviously. When \(y\in U_m\) is near \(\Sigma_m\), one can deduce \eqref{TE-1} from \eqref{TE} directly. 

\end{proof}

Finally, one can flatten the boundary and use the partition of unity to show that Lemma \ref{le:W-9} holds for a general domain.

\subsection{Proof of Theorem \ref{th:main-100}}
This will be done by following the strategy of Sections \ref{The Linearized Problem}-\ref{Uniqueness Part}.

\subsubsection{The linearized problem \eqref{existence-3}} We first claim that \(v\in L^\infty([0,T]; H^4(\Omega))\) if \(\sup_{0\leq t\leq T}\widetilde{E}(t,v)<\infty\). 
In fact, it follows from Lemma \ref{le:Preliminary-100} that
	\begin{align*}
		\|\tilde{v}_m\|_{H^4(U_m)}\lesssim (E^{\text{B}_m}(t,v))^{1/2}. 
	\end{align*}
Since \(\Phi_m, \Psi_m\in C^8\), it holds that 
     \begin{align*}
		\|v\|_{H^4(\Omega \cap B(x^m,r_m/2))}\lesssim (E^{\text{B}_m}(t,v))^{1/2}. 
	\end{align*}
By the properties of the partition of unity, one may estimate
\begin{equation*}
	\begin{aligned}
		\|v\|_{H^4(\Omega)}^2&\lesssim \|\zeta_0 v\|_{H^4(\Omega)}^2+\sum_{m=1}^K \|\zeta_m v\|_{H^4(\Omega)}^2\\
&\lesssim \|v\|_{H^4(\Omega_0)}^2+\sum_{m=1}^K \|v\|_{H^4(\Omega \cap B(x^m,r_m/2))}^2\\
&\lesssim 
E^{\text{I}}(t, v)+\sum_{m=1}^KE^{\text{B}_m}(t, v)=\widetilde{E}(t,v).
	\end{aligned}
\end{equation*}

Now, we can define the solution space of  \(v\). Given \(T>0\), let
\begin{equation*}
	\begin{aligned}
		\mathcal{X}_T=\{ v\in L^\infty([0,T]; H^4(\Omega)):\ \sup_{0\leq t\leq T}\widetilde{E}(t,v)<\infty\},
	\end{aligned}
\end{equation*}
with the norm
\begin{equation*}
	\begin{aligned}
		\|v\|_{\mathcal{X}_T}^2=\sup_{0\leq t\leq T}\widetilde{E}(t,v).
	\end{aligned}
\end{equation*}
Given \(M_1\), define \(\mathcal{C}_T(M_1)\) to be a closed, bounded, and
convex subset of \(\mathcal{X}_T\):
\begin{equation*}
	\begin{aligned}
		\mathcal{C}_T(M_1)&=\{ v\in \mathcal{X}_T:\|v\|_{\mathcal{X}_T}^2\leq M_1\}.
	\end{aligned}
\end{equation*}

To show Theorem \ref{th:main-100}, 
we will start from studying \eqref{existence-3} 
as Section \ref{The Linearized Problem}.

\subsubsection{Construction of an orthogonal basis of \(H_{\rho_0}^1(\Omega)\)} Defining the same degenerate-singular elliptic operator
\eqref{degenerate-singular elliptic equation} in \(\Omega\), one can construct an orthogonal basis \(\{\tilde{w}_{m_l}\}_{m_l=1}^\infty\) of \(H_{\rho_0}^1(\Omega)\) by following the argument in Section \ref{Degenerate-Singular Elliptic Operators}, namely Theorem \ref{Hilbert basis construction} and Theorem \ref{regularity of eigenvalue}. However, the additional regularity of the eigenfuctions \(\{\tilde{w}_{m_l}\}_{m_l=1}^\infty\), corresponding to \eqref{weighted eigenfunction regularity}, should be replaced by the following estimate in \(y\)-coordinate:
{\small \begin{equation}\label{weighted eigenfunction regularity-100}
\begin{aligned}
&\sum_{m=1}^K\sum_{l_1=0}^7\big\|\sqrt{{\widetilde{\rho_0}}_m}\tilde{\partial}_1^{l_1}\tilde{D}\tilde{w}_{m_l}\big\|_{L^2(U_m)}+\sum_{m=1}^K\sum_{\substack{l_1+l_2\leq 8\\  l_1\geq 0,\ l_2\geq 2}}\big\|\sqrt{{\widetilde{\rho_0}}_m^{l_2}}\tilde{\partial}_1^{l_1}\tilde{\partial}_2^{l_2}\tilde{w}_{m_l}\big\|_{L^2(U_m)}\\
&\leq C\|\sqrt{\rho_0}w_l\|_{L^2(\Omega)}, \quad l=1,2,\dots,
\end{aligned}
\end{equation}}

\noindent where \(\tilde{w}_{m_l} (y)=w_l(\Psi_m(y))\). 

To show \eqref{weighted eigenfunction regularity-100}, one can start from the weak formulation \eqref{weak form} for \(w_l\). 
To improve the regularity of \(w_l\), one shall first choose some tangential difference quotient as the test function \(\varphi\) in 
\eqref{weak form} to deduce the tangential estimates and then use the strong formulation for \(w_l\)
\begin{equation}\label{DS-4}
\begin{aligned}
\mathrm{div}(\rho_0Dw_l)=(\sigma_l-1)\rho_0 w_l \quad \mathrm{a.e.\ in}\ \Omega 
\end{aligned}
\end{equation}
to obtain the normal estimates.  Since \(\rho_0\)  degenerates only on \(\partial\Omega\), one needs only to take care of the boundary regularity of \(w_l\). 

To define the tangential difference quotient near \(\partial\Omega\), one needs to flatten \(\partial\Omega\). To this end, fixing \(x^m\in \partial\Omega\),  we transfer \eqref{weak form} from \(x\)-coordinate to \(y\)-coordinate, which takes the following form: 
\begin{equation}\label{weak form-100}
\begin{aligned}
\int_{U_m} {\widetilde{\rho_0}}_m c_m^{kj}\tilde{w}_{m_l},_{j} \tilde{\varphi}_m,_{k}\,\diff y=(\sigma_l-1)\int_{U_m} {\widetilde{\rho_0}}_m \tilde{w}_{m_l}  \tilde{\varphi}_m\,\diff y,
\end{aligned}
\end{equation} 
where
{\small\begin{equation*}
\begin{aligned}
\tilde{\varphi}_m(y)=\varphi(\Psi_m(y))\quad \text{and}\quad
c_m^{kj}(y)=\sum_{r=1}^2\frac{\partial\Phi_m^k}{\partial x_r}(\Psi_m(y))\frac{\partial\Phi_m^j}{\partial x_r}(\Psi_m(y)).
\end{aligned}
\end{equation*}}	

\noindent By \eqref{a-bound-2}, it is direct to check that there exists a constant \(C>0\) such that
\begin{align*}\label{a-bound-200}
	c_m^{kj}(y)\xi_k\xi_j\geq C|\xi|^2 \quad \text{for}\ y\in U_m\ \text{and}\ \xi\in \R^2. 
\end{align*}
Define the tangential difference quotient 
\begin{align*}
  D_1^\tau w(y)=\frac{w(y+\tau {\bf e}_1)-w(y)}{\tau},
\end{align*}
where \({\bf e}_1\) is the \(y_1\)-direction coordinate vector. Choose a smooth cutoff function
\begin{equation*}
\chi(y)=
\begin{cases}
1\quad \mbox{on}\ B(0,s_m/2),\\
0\quad \mbox{on}\ B(0,s_m)\setminus B(0,3s_m/4). 
\end{cases}
\end{equation*}
Now one can use  \(D_1^\tau\) and \(\chi\) to deduce the desired tangential estimates. For example, by taking 
\(\tilde{\varphi}_m=-D_1^{-\tau}(\chi^2D_1^\tau \tilde{w}_{m_l})\) in \eqref{weak form-100}, one may obtain  
\begin{equation*}\label{DS-200}
\begin{aligned}
\big\|\sqrt{{\widetilde{\rho_0}}_m}\tilde{\partial}_1\tilde{D}\tilde{w}_{m_l}\big\|_{L^2(U_m)}\leq C\|\sqrt{\rho_0}w_l\|_{L^2(\Omega)}.
\end{aligned}
\end{equation*}
Here \(\chi\) plays a role in eliminating the boundary term on \(\partial U_m\setminus\Sigma_m\) in the integration by parts.  

To obtain the normal estimates, one may also transfer \eqref{DS-4} from \(x\)-coordinate to \(y\)-coordinate yielding
\begin{equation}\label{DS-100}
\begin{aligned}
({\widetilde{\rho_0}}_m c_m^{kj}\tilde{w}_{m_l},_{j}),_{k}=(\sigma_l-1){\widetilde{\rho_0}}_m \tilde{w}_{m_l} \quad \mathrm{a.e.\ in}\ U_m. 
\end{aligned}
\end{equation}
Based on \eqref{DS-100}, one can follow the argument in Subsection \ref{LOP} to deduce the desired normal estimates in \(y\)-coordinate.

\subsubsection{Existence of weak solutions to the linearized problem \eqref{existence-3}}
One can use the Galerkin's scheme with the orthogonal basis \(\{\tilde{w}_{m_l}\}_{m_l=1}^\infty\) of \(H_{\rho_0}^1(\Omega)\) as the projection space to construct approximate solutions (which converge to the weak solution) to \eqref{existence-3}. In fact, one can show Lemma \ref{Existence and uniqueness of a weak solution} following exactly the same argument in Section \ref{weak solution}. The key point is that all the calculations can be carried out in \(x\)-coordinate since one does not need to distinguish the tangential and normal derivatives.

\subsubsection{Regularity of weak solutions to the linearized problem \eqref{existence-3}}
To carry out uniform high-order estimates for the approximate solutions (to improve the regularity of the weak solution) to \eqref{existence-3}, we will work with \(\eqref{existence-3}_1\) in \(y\)-coordinate (multiplying \(\chi^2\)), which takes the following form: 
\begin{equation}\label{existence-100}
	\begin{aligned}
		\chi^2{\widetilde{\rho_0}}_m\partial_t\tilde{v}_m^i+\chi^2({\widetilde{\rho_0}}_m^2 \tilde{\bar{J}}_m^{-2}\tilde{\bar{a}}_{i_m}^k),_{k}=\chi^2({\widetilde{\rho_0}}_m\tilde{\bar{J}}_m^{-2}\tilde{\bar{b}}_m^{kj}\tilde{v}_m^i,_{j}),_{k}  \quad \mathrm{a.e.\ in}\ U_m,
	\end{aligned}
\end{equation}
where 
\begin{equation*}
\begin{aligned}
&\tilde{\bar{J}}_m(y)=\bar{J}(\Psi_m(y)),
\quad  \tilde{\bar{b}}_m^{kj}=\tilde{\bar{a}}_{i_m}^{k}\tilde{\bar{a}}_{i_m}^{j},\\
&\tilde{\bar{a}}_{i_m}^{k}(y)=\sum_{r=1}^2\bar{a}_i^{r}(\Psi_m(y))\frac{\partial\Phi_m^k}{\partial x_r}(\Psi_m(y)).
\end{aligned}
\end{equation*}	
The reason of using the cutoff function \(\chi^2\) lies in that one can eliminate the boundary term on \(\partial U_m\setminus\Sigma_m\) when implementing integration by parts. 
Now, one can carry out the arguments in Sections \ref{Energy Estimates}, \ref{Elliptic Estimates} and \ref{Regularity} to  show \(v\in \mathcal{C}_T(M_1)\).

\subsubsection{Existence and uniqueness of classical solutions to the problem \eqref{eq:main-2}} 
The proof is exactly same as that in Sections \ref{Existence Part} and \ref{Uniqueness Part}.

%\subsubsection{Remarks}
%\begin{remark} The inner energy  \(E^{\text{I}}(t,v)\) is defined in \(\Omega_0\) where \(\eqref{existence-3}_1\) is not degenerate, which can be estimated by interior estimates on \(v\) in \(x\)-coordinate. The boundary energy \(E^{\text{B}_m}(t,v)\) is defined in \(U_m\) (in \(y\)-coordinate) which corresponds to \(\Omega\cap B(x^m,r_m)\) (in \(x\)-coordinate) where \(\eqref{existence-3}_1\) is degenerate, which can be estimated by boundary estimates on \(\tilde{v}_m\) in \(y\)-coordinate by flattening out the boundary \(\partial\Omega\cap B(x^m,r_m)\). 

%The total energy \(\widetilde{E}(t,v)\) defined by \eqref{HOEF-101} is convenient to study the regularity of the solutions to \eqref{existence-3}. 

%\end{remark}

\section*{Acknowledgment}
Li's research was supported by the National Natural Science Foundation of China (Grant Nos. 11931010 and 12331007), and by the key research project of Academy for Multidisciplinary Studies, Capital Normal University, and by the Capacity Building for Sci-Tech Innovation-Fundamental Scientific Research Funds 007/20530290068. Wang's research was supported by the Grant No. 830018 from China. 
Xin's research was supported by the Zheng Ge Ru Foundation and by Hong Kong RGC Earmarked Research Grants CUHK14301421, CUHK14300819, CUHK14302819, CUHK14300917,  Basic and Applied Basic Research Foundations of Guangdong Province 20201\
31515310002, and the Key Project of National Nature Science Foundation of China (Grant No. 12131010). 

Part of this research was carried out when the second author visited the Institute of Mathematical Sciences, CUHK.
He acknowledges the institute for the hospitality.

\end{document}